\documentclass{article}
\usepackage{arxiv}
\usepackage[utf8]{inputenc}
\usepackage[T1]{fontenc}
\usepackage{amsmath,amssymb,amsthm,mathtools}
\usepackage{graphicx}
\usepackage{booktabs}
\usepackage{makecell}
\usepackage{microtype}
\usepackage{enumitem}
\usepackage[numbers,sort&compress]{natbib}
\usepackage{algorithm}
\usepackage{algpseudocode}
\usepackage[hidelinks]{hyperref}

\theoremstyle{plain}
\newtheorem{theorem}{Theorem}[section]
\newtheorem{lemma}[theorem]{Lemma}
\newtheorem{proposition}[theorem]{Proposition}
\newtheorem{corollary}[theorem]{Corollary}

\theoremstyle{definition}
\newtheorem{definition}[theorem]{Definition}
\newtheorem{assumption}{Assumption}

\theoremstyle{remark}
\newtheorem{remark}[theorem]{Remark}

\numberwithin{equation}{section}

\newcommand{\R}{\mathbb R}

\DeclarePairedDelimiterX{\inner}[2]{\langle}{\rangle}{#1,#2}

\title{The Minimax Complexity of Federated Optimization under Interpolation: The Role of Oracle Models}
\author{%
  Kirill Novoselov \\
  Innopolis University \\
  \texttt{k.novoselov@innopolis.university}
  \And
  Yuriy Dorn \\
  Lomonosov Moscow State University 
  \AND
  Darina Dvinskikh \\
  HSE University \\
  \texttt{dmdvinskikh@hse.ru}
  \And
  Alexander Gasnikov \\
  Innopolis University, Trusted AI Research Center, RAS, \\
  Steklov Mathematical Institute, RAS \\
  \texttt{gasnikov@yandex.ru}
}
\date{}
\renewcommand{\shorttitle}{Federated Optimization under Interpolation}

\begin{document}
\maketitle
\makeatletter\global\@inlabelfalse\global\@newlistfalse\makeatother


\begin{abstract}
We study the minimax complexity of smooth convex stochastic optimization
with intermittent communication: $m$ workers run for $N$ rounds, and
between consecutive communication rounds each worker makes $K$
sequential stochastic gradient queries. We focus on the interpolation
regime, in which all sample losses share a common minimizer, so that the
gradient noise vanishes at the solution. Let $L$ be the smoothness of the
objective, $H_s\ge L$ the smoothness of the sample losses, and $R$ the
distance from the initial point to a minimizer. We consider three oracle
models that provide increasing information: an unbiased stochastic gradient oracle $g(x,\xi)$ satisfying $\mathbb{E}_{\xi}\|g(x,\xi)\|^2
    \leq 2H_s\bigl(F(x)-F(x^*)\bigr)$, gradients of the sample
losses at fresh samples, and repeated gradient queries to stored
samples. Up to logarithmic factors, all our results are described by two
rates:
\[
    R^2\min\left\{L,\;\frac{H_s}{KN},\;
        \frac{L}{N^2}+\frac{H_s}{mKN}\right\}
    \qquad\text{and}\qquad
    R^2\min\left\{L,\;
        \frac{\min\{L,H_s/K\}}{N^2}+\frac{H_s}{mKN}\right\},
\]
 Under the weak growth condition, the first rate is the
minimax rate, attained by the best of three simple baselines, as in the
classical bounded-variance setting. With repeated access to samples, the
second rate becomes the minimax rate for general convex objectives. For general convex
objectives with fresh sample gradients, the minimax rate lies between the
two rates. Our lower bounds apply to arbitrary randomized
intermittent-communication algorithms.
\end{abstract}
\section{Introduction}
\label{sec:introduction}

Modern machine learning models are overparameterized and interpolate the
training data: a single parameter vector minimizes the loss of every
example. In this regime the noise of stochastic gradients vanishes at the
solution, and stochastic methods attain faster rates than in the
classical setting. SGD with a constant step size converges at the rate of
gradient descent \citep{schmidt2013fast,ma2018power,vaswani2019fast},
momentum methods attain accelerated rates
\citep{vaswani2019fast,liu2020accelerating}, and minibatch accelerated
methods achieve a linear speedup in the minibatch size
\citep{woodworth2021even}. These results concern a single machine.

A natural question is how interpolation affects distributed training,
where communication between machines is the bottleneck, as in federated
learning \citep{mcmahan2017communication,kairouz2021advances}. The
standard model of this regime is intermittent communication
\citep{woodworth2018graph}: $m$ workers run for $N$ rounds, each worker
makes $K$ sequential stochastic gradient queries in every round, and the
workers communicate only at the end of rounds. For convex objectives with
$L$-Lipschitz gradients and stochastic gradients whose variance is
uniformly bounded by $\sigma^2$, \citet{woodworth2021minmax} determined
the minimax complexity of this model up to logarithmic factors: the
optimal method is the better of minibatch accelerated SGD
\citep{lan2012optimal}, with $N$ updates and batch size $mK$, and
accelerated SGD on a single worker. Under interpolation, the noise
vanishes at the solution and scales with the smoothness $H_s$ of the
individual losses, and the minimax complexity of intermittent
communication in this regime has remained open.

Besides interpolation itself, the answer depends on what the oracle
reveals. Discussing how their lower bound might be circumvented,
\citet[Sections~5.2 and~5.3]{woodworth2021minmax} pointed to two features
of statistical learning. First, stochastic gradients are usually
gradients $\nabla f(x;z)$ of smooth convex losses at random examples $z$.
The hard oracle of their lower bound is not of this form, and they asked
whether this structure yields better guarantees or whether the lower
bound persists. Second, an algorithm can often evaluate the gradient of
the same example at several points, as variance-reduced methods do, and
distributed variants of SVRG improve on their lower bound. We study
intermittent communication under interpolation in three oracle models
that isolate these features, ordered by the information that a query
provides.
\begin{itemize}
    \item \emph{Weak growth} (Assumption~\ref{ass:wg}): the oracle is
    unbiased, and its second moment is at most $2H_s$ times the
    suboptimality of the query. Nothing else is assumed. Interpolation
    implies this condition for sample gradients
    (Lemma~\ref{lem:relations}). The upper bounds in Sections ~\ref{sec:weak-growth} and ~\ref{sec:sample-gradients} rely only on this condition.
    \item \emph{Fresh sample gradients} (Assumption~\ref{ass:sg}): the
    oracle returns the gradient of an $H_s$-smooth convex sample loss at a
    fresh example, and all sample losses are minimized at a common point.
    This is the setting of \citet[Section~5.2]{woodworth2021minmax} under
    interpolation.
    \item \emph{Sample reuse} (Assumption~\ref{ass:pool}): the examples
    form a stored training set, and the gradient of the same example can
    be computed at several points. This is the setting of
    \citet[Section~5.3]{woodworth2021minmax} under interpolation.
\end{itemize}
Here $L\le H_s$ is the smoothness of the population objective. The resulting minimax bounds are summarized in 
Table~\ref{tab:results}.

\subsection{Contributions}

\begin{itemize}
    \item \emph{Weak growth} (Section~\ref{sec:weak-growth}). We determine
    the minimax rate up to the factor $\log^2(em)$, for general convex and
    for quadratic populations. For general convex populations, the lower
    bound extends the construction of \citet{woodworth2021minmax} to
    relative noise, and the optimal method is the best of three baselines:
    returning $x_0$, serial SGD on one worker, and minibatch AC-SA. For
    quadratic populations, local SGD steps reduce the effective smoothness
    to $\min\{L,H_s/K\}$, and an accelerated method on the resulting
    spectral surrogate is optimal.

    \item \emph{Fresh sample gradients} (Section~\ref{sec:sample-gradients}).
    We prove a lower bound with a new hard instance, since the
    construction of \citet{woodworth2021minmax} is not a sample-gradient
    oracle. This answers the question of
    \citet[Section~5.2]{woodworth2021minmax} for quadratic populations
    under interpolation: the structure of sample gradients gives the same
    minimax rate as the weak growth condition. For general convex
    populations a gap remains, and we show that closing it from above
    requires structure beyond the weak growth condition.

    \item \emph{Sample reuse} (Section~\ref{sec:reuse}). In the setting of
    \citet[Section~5.3]{woodworth2021minmax} under interpolation, we
    determine the minimax rate up to the factor $J\log^2(em)$ with
    $J=O(\log(KN))$.
    The upper bound is attained by an accelerated method on empirical
    Moreau envelopes, which every worker computes by SVRG on its own batch.
    The lower bound holds for arbitrary reuse of samples, and for an
    infinite pool it coincides with the lower bound for fresh sample
    gradients.
\end{itemize}
All lower bounds hold for arbitrary randomized algorithms in finite
dimension, without linear-span assumptions.

\begin{table}[h]
\centering
\small
\renewcommand{\arraystretch}{1.6}
\caption{Bounds on the minimax risk in the intermittent communication
model with $m$ workers, $N$ rounds, and $K$ queries per worker per
round, for convex objectives with $L$-Lipschitz gradients and
$\|x_0-x^*\|\le R$. Numerical constants are omitted. The result of
\citet{woodworth2021minmax} holds up to logarithmic factors. Here $H_s$ is
the constant of the weak growth condition (the smoothness of the sample
functions), $\sigma^2$ the variance, and $S$ the pool size.}
\label{tab:results}
\begin{tabular}{@{}lccc@{}}
\toprule
Oracle & Lower bound & Upper bound & Reference\\
\midrule
\makecell[l]{Bounded variance}
 & \multicolumn{2}{c}{$\frac{LR^2}{K^2N^2}+\min\Bigl\{LR^2,\frac{\sigma R}{\sqrt{mKN}}\Bigr\}
   +\min\Bigl\{\frac{LR^2}{N^2},\frac{\sigma R}{\sqrt{KN}}\Bigr\}$}
 & \citet{woodworth2021minmax}\\
\midrule
\makecell[l]{Weak growth,\\ quadratic $F$}
 & \multicolumn{2}{c}{$R^2\min\Bigl\{L,\frac{\min\{L,H_s/K\}}{N^2}
   +\frac{H_s}{mKN}\Bigr\}$}
 & \makecell{\textbf{This paper}\\ Theorem~\ref{thm:weak-growth-quadratic}}\\
\makecell[l]{Weak growth,\\ convex $F$}
 & \multicolumn{2}{c}{$R^2\min\Bigl\{L,\frac{H_s}{KN},\frac{L}{N^2}
   +\frac{H_s}{mKN}\Bigr\}$}
 & \makecell{\textbf{This paper}\\ Theorem~\ref{thm:weak-growth-convex}}\\
\midrule
\makecell[l]{Sample gradients,\\ quadratic $F$}
 & \multicolumn{2}{c}{$R^2\min\Bigl\{L,\frac{\min\{L,H_s/K\}}{N^2}
   +\frac{H_s}{mKN}\Bigr\}$}
 & \makecell{\textbf{This paper}\\ Theorem~\ref{thm:sample-gradient-quadratic}}\\
\makecell[l]{Sample gradients,\\ convex $F$}
 & $R^2\min\Bigl\{L,\frac{\min\{L,H_s/K\}}{N^2}+\frac{H_s}{mKN}\Bigr\}$
 & $R^2\min\Bigl\{L,\frac{H_s}{KN},\frac{L}{N^2}+\frac{H_s}{mKN}\Bigr\}$
 & \makecell{\textbf{This paper}\\ Theorem~\ref{thm:sample-gradient-convex}}\\
\midrule
\makecell[l]{Sample reuse,\\ convex $F$}
 & \multicolumn{2}{c}{$R^2\min\Bigl\{L,\frac{\min\{L,H_s/K\}}{N^2}
   +\frac{H_s}{\min\{S,mKN\}}\Bigr\}$}
 & \makecell{\textbf{This paper}\\ Theorem~\ref{thm:pool}}\\
\bottomrule
\end{tabular}
\end{table}

\subsection{Related work}

\paragraph{Intermittent communication and local updates.}
The intermittent communication model was studied by
\citet{woodworth2018graph} within a graph-oracle framework for parallel
stochastic optimization, together with lower bounds for it. Local SGD was analyzed by
\citet{stich2019local} and compared with minibatch SGD by
\citet{woodworth2020local}, who showed that local SGD strictly dominates
minibatch SGD for quadratic objectives, that its accelerated variant is
minimax optimal for quadratics, and that it does not dominate minibatch
SGD for general convex objectives; heterogeneous objectives were studied by
\citet{woodworth2020minibatch}. The min-max complexity under bounded
variance was determined by \citet{woodworth2021minmax}. We refer to
\citet{kairouz2021advances} for a survey of federated learning. Our
results show that the advantage of local updates on quadratics persists
under weak growth, and that sample reuse extends it to general convex
populations under interpolation.

\paragraph{Interpolation and growth conditions.}
Fast convergence of SGD under the strong growth condition, where the
second moment of the stochastic gradient is bounded by a multiple of
$\|\nabla F\|^2$, was shown by \citet{schmidt2013fast}. The benefits of
interpolation for SGD and its minibatch speedup were explained by
\citet{ma2018power}. \citet{vaswani2019fast} introduced the weak growth
condition, which interpolation implies for smooth finite sums, and showed
that SGD with Nesterov acceleration attains the deterministic accelerated
rate under strong growth.
For least squares, \citet{liu2020accelerating} accelerated SGD in the
interpolation regime. Rates that adapt to the minimal loss for smooth
nonnegative losses go back to \citet{srebro2010smoothness}, and accelerated
minibatch methods with such rates were given by \citet{cotter2011better}
and, with the optimal dependence on the minibatch size, by
\citet{woodworth2021even}. We use the accelerated scheme of the latter in
Appendix~\ref{app:appendix_C}. These works concern a single machine;
we study the intermittent communication model.

\paragraph{Local subproblems.}
Solving a proximal subproblem on a minibatch is the idea of minibatch-prox
\citep{wang2017memory}. Our method with sample reuse computes such
proximal points locally by SVRG \citep{johnson2013accelerating} and
uses them as gradients of empirical Moreau envelopes inside an
accelerated method; interpolation guarantees that the envelopes share the
minimizer of the population.

\paragraph{Lower bounds.}
Our lower bounds use hidden chains, which go back to
Nemirovsky and Yudin~\citep{nemirovsky1983problem},
combined with random rotations to avoid linear-span
assumptions, as in Woodworth et al.~\citep{woodworth2021minmax}.
For the weak growth condition, we extend the rare-progress
oracle of Woodworth et al.~\citep{woodworth2021minmax}
to relative noise.

For quadratic populations, we use a random-rotation
representation of adaptive matrix-vector algorithms,
building on techniques developed by
Braverman et al.~\citep{braverman2020gradient}
and closely related to those recently employed by
\citet{shamir2026simple} in the centralized setting.
Our analysis additionally accounts for the depth of
batched adaptive queries and uses a Chebyshev-type
certificate tailored to smooth convex quadratic objectives.
We combine these techniques with Bernoulli masking of
entire sample losses, an approach already used in our
previous work on federated variational
inequalities~\citep{novoselov2026optimal},
to establish lower bounds for interpolating
sample-gradient oracles under intermittent communication.

\section{Setting}
\label{sec:setting}

Throughout, $0<L\le H_s$ and $R>0$ are constants, $m,K,N\ge1$ are
integers, $d\ge1$ is the dimension, and $x_0\in\R^d$ is the initial
point. All norms and inner products are Euclidean, and all logarithms are
natural.

\paragraph{Intermittent communication.}

\begin{definition}[Intermittent communication]
\label{def:ic}
An algorithm in the class $\mathrm{IC}(m,K,N)$ runs $m$ workers for $N$
rounds. In every round, each worker makes at most $K$ sequential oracle
queries. Each query of a worker is a measurable function of the public
parameters, the random seed of the algorithm, all messages exchanged
before the current round, and the queries and replies of this worker in
the current round. At the end of every round the workers exchange
arbitrary messages. The output $\widehat x\in\R^d$ is a measurable
function of the seed and of all messages. The seed is independent of the
instance and of the oracle randomness.
\end{definition}

No restriction is imposed on memory, local computation, message size, or
the norms of queries and outputs; in particular, queries need not lie in
the span of previous replies. The public parameters are
$m,K,N,L,H_s,R,d,x_0$, and the pool size $S$ in
Assumption~\ref{ass:pool}. The objective, its minimizer, and the oracle
law are unknown to the algorithm.

\paragraph{Oracle assumptions.}
We consider three oracle models, ordered by the information that a
query provides.

\begin{assumption}[Weak growth condition]
\label{ass:wg}
The objective $F:\R^d\to\R$ is convex, has an $L$-Lipschitz gradient,
and has a minimizer $x^*$ with $\|x_0-x^*\|\le R$. There are a
probability distribution $\mathcal D$ on a measurable space $\Xi$ and a
jointly measurable map $g:\R^d\times\Xi\to\R^d$ such that, for every
$x\in\R^d$,
\begin{equation}
    \mathbb E_{\xi\sim\mathcal D}\, g(x,\xi)=\nabla F(x),
    \qquad
    \mathbb E_{\xi\sim\mathcal D}\|g(x,\xi)\|^2
    \le 2H_s\bigl(F(x)-F(x^*)\bigr).
    \label{eq:wg}
\end{equation}
On a query $x$, the oracle draws a fresh sample $\xi\sim\mathcal D$ and
returns only $g(x,\xi)$.
\end{assumption}

The map $g(\cdot,\xi)$ in Assumption~\ref{ass:wg} need not be a gradient
field or continuous.

\begin{assumption}[Sample gradients under interpolation]
\label{ass:sg}
There are a probability distribution $\mathcal D$ on a measurable space
$\Xi$ and differentiable convex functions $f_\xi:\R^d\to\R$, $\xi\in\Xi$,
such that $(x,\xi)\mapsto f_\xi(x)$ and $(x,\xi)\mapsto\nabla f_\xi(x)$
are jointly measurable and, for some $x^*\in\R^d$ with
$\|x_0-x^*\|\le R$ and $\mathcal D$-almost every $\xi$,
\begin{equation}
    \|\nabla f_\xi(x)-\nabla f_\xi(y)\|\le H_s\|x-y\|
    \quad\text{for all }x,y\in\R^d,
    \qquad
    \nabla f_\xi(x^*)=0.
    \label{eq:sg}
\end{equation}
The objective
\[
    F(x)=\mathbb E_{\xi\sim\mathcal D}\bigl[f_\xi(x)-f_\xi(x^*)\bigr]
\]
has an $L$-Lipschitz gradient. On a query $x$, the oracle draws a fresh
sample $\xi\sim\mathcal D$, independent of everything else, and returns
only the vector $g(x,\xi)=\nabla f_\xi(x)$.
\end{assumption}

The second condition in \eqref{eq:sg} is the interpolation assumption:
$x^*$ minimizes almost every sample function. In particular,
$F(x^*)=0=\min F$.

\begin{assumption}[Sample reuse]
\label{ass:pool}
The functions $f_\xi$ and $F$ satisfy the conditions of
Assumption~\ref{ass:sg}, and $S\in\mathbb N\cup\{\infty\}$, $S\ge1$.
Before the run, samples
$\xi_1,\xi_2,\ldots\stackrel{\mathrm{iid}}{\sim}\mathcal D$ are drawn
independently of the seed of the algorithm; only $\xi_1,\ldots,\xi_S$
are available. A query is a pair $(i,x)$ with $i\in\mathbb N$ and
$x\in\R^d$; the oracle returns $\nabla f_{\xi_i}(x)$ if $i\le S$ and a
failure symbol otherwise. Every query, including a repeated one, is one
of the at most $K$ queries of the worker in the current round. The label
$i$ carries no information about $\xi_i$, and labels may be included in
messages.
\end{assumption}

An \emph{instance} is a choice of the objects in an assumption. We denote
by $\mathcal W_d$ and $\mathcal S_d$ the sets of instances satisfying
Assumptions~\ref{ass:wg} and~\ref{ass:sg}, and by
$\mathcal W_d^{\mathrm{quad}}$ and $\mathcal S_d^{\mathrm{quad}}$ their
subsets with a quadratic objective
\[
    F(x)=\frac12\langle x-x^*,\mathsf A(x-x^*)\rangle,
    \qquad 0\preceq \mathsf A\preceq L\mathrm I;
\]
the sample functions $f_\xi$ need not be quadratic.

\paragraph{Minimax risk.}
For a set $\mathcal C$ of instances and a model of oracle access, the
minimax risk is
\[
    \mathfrak M_d(\mathcal C)
    =\inf_{\mathcal A\in\mathrm{IC}(m,K,N)}\;
     \sup_{\text{instances in }\mathcal C}\;
     \mathbb E\bigl[F(\widehat x)-F(x^*)\bigr],
\]
where the expectation is over the oracle randomness and the seed of
$\mathcal A$; the instance is fixed before both are drawn. We write
\[
    \mathfrak M_d^{\mathrm{wg}}=\mathfrak M_d(\mathcal W_d),\quad
    \mathfrak M_d^{\mathrm{wg,quad}}
        =\mathfrak M_d(\mathcal W_d^{\mathrm{quad}}),\quad
    \mathfrak M_d^{\mathrm{sg}}=\mathfrak M_d(\mathcal S_d),\quad
    \mathfrak M_d^{\mathrm{sg,quad}}
        =\mathfrak M_d(\mathcal S_d^{\mathrm{quad}})
\]
for the fresh oracles of Assumptions~\ref{ass:wg} and~\ref{ass:sg}, and
$\mathfrak M_d^{\mathrm{pool}}(S)$ for $\mathfrak M_d(\mathcal S_d)$
under Assumption~\ref{ass:pool}. The fresh oracle of
Assumption~\ref{ass:sg} is the special case $S=\infty$ of
Assumption~\ref{ass:pool} in which every label is queried at most once;
hence
\begin{equation}
    \mathfrak M_d^{\mathrm{pool}}(\infty)\le\mathfrak M_d^{\mathrm{sg}},
    \label{eq:pool-fresh}
\end{equation}
see Appendix~\ref{app:B-prelim}.

\section{Weak growth: tight minimax rates}
\label{sec:weak-growth}

We begin by considering the stochastic oracle satisfying
Assumption~\ref{ass:wg}, that satisfy the weak growth
condition (WGC)~\ref{eq:wg}. The WGC has been studied in
the context of stochastic optimization by
Vaswani et al.~\cite{vaswani2019fast}, who demonstrated
its relevance to statistical learning under interpolation.
In particular, they established the following result,
showing that the WGC arises naturally for smooth convex
finite-sum objectives whose component functions share
a common minimizer.

\begin{proposition}[Vaswani et al., 2019]
    Let
    \[
        f(w) = \frac{1}{n}\sum_{i=1}^{n} f_i(w)
    \]
    be a convex finite-sum objective, where each component
    $f_i$ is convex and $L_i$-smooth. Suppose that the model
    interpolates the data, i.e., there exists a common
    minimizer $w^*$ of all component functions. Then,
    for every $w$,
    \begin{equation}
        \mathbb{E}_i\left[\|\nabla f_i(w)\|^2\right]
        \leq
        2L_{\max}\bigl(f(w)-f(w^*)\bigr),
        \label{eq:wgc_interpolation}
    \end{equation}
    where $L_{\max}=\max_{i=1,\ldots,n}L_i$.
\end{proposition}

The proof can be found in
Vaswani et al.~\cite{vaswani2019fast}.

As we will show in the subsequent sections, the WGC
follows from the stronger structural assumptions introduced
in Assumptions~\ref{ass:sg} and~\ref{ass:pool}. Thus, while the WGC is imposed
explicitly in the present section, it holds automatically
under the oracle models considered in Sections~\ref{sec:sample-gradients}
and~\ref{sec:reuse}.
Importantly, the upper bounds derived in this section
rely only on the WGC~\ref{eq:wg}, without exploiting the
additional structure of these oracle models. They
therefore provide a natural baseline for assessing whether
the stronger assumptions lead to improved convergence
guarantees.

\begin{theorem}[General convex populations under weak growth]
\label{thm:weak-growth-convex}
Let Assumption~\ref{ass:wg} hold. Then, in every dimension $d\ge1$,
\[
    \mathfrak M_d^{\mathrm{wg}}
    \le 288R^2\min\left\{L,\;\frac{H_s}{KN},\;
        \frac{L}{N^2}+\frac{H_s}{mKN}\right\},
\]
and, for every $d\ge d_{\mathrm{wg}}$,
\[
    \mathfrak M_d^{\mathrm{wg}}
    \ge\frac{R^2}{2^{24}\log^2(em)}\min\left\{L,\;\frac{H_s}{KN},\;
        \frac{L}{N^2}+\frac{H_s}{mKN}\right\},
\]
where $d_{\mathrm{wg}}<\infty$ is defined in \eqref{eq:d-wg} and depends
only on $m$, $K$, $N$, and $H_s/L$.
\end{theorem}

\begin{theorem}[Quadratic populations under weak growth]
\label{thm:weak-growth-quadratic}
Let Assumption~\ref{ass:wg} hold with a quadratic objective $F$. Then,
in every dimension $d\ge1$,
\[
    \mathfrak M_d^{\mathrm{wg,quad}}
    \le 8R^2\min\left\{L,\;
        \frac{\min\{L,H_s/K\}}{N^2}+\frac{H_s}{mKN}\right\},
\]
and, for every $d\ge2048m(KN)^2$,
\[
    \mathfrak M_d^{\mathrm{wg,quad}}
    \ge\frac{R^2}{73728\log^2(em)}\min\left\{L,\;
        \frac{\min\{L,H_s/K\}}{N^2}+\frac{H_s}{mKN}\right\}.
\]
\end{theorem}

The upper bound in Theorem~\ref{thm:weak-growth-convex} is achieved by a
parameter-based selection among returning $x_0$, serial averaged SGD on
one worker, and common-query minibatch AC-SA \citep{lan2012optimal} with
$N$ accelerated updates and batch size $mK$. The upper bound in
Theorem~\ref{thm:weak-growth-quadratic} is achieved by a parameter-based
selection between returning $x_0$ and Algorithm~\ref{alg:spectral-blocks},
in which every worker runs $K$ local SGD steps from a common point and
the server performs an accelerated update with the averaged block
displacements. The upper bounds are proved in
Appendices~\ref{app:appendix_B_convex} and~\ref{app:appendix_B_quadratic},
the lower bounds in Appendices~\ref{app:lower-wg}
and~\ref{app:lower-quadratic}, and the proofs are completed in
Appendix~\ref{app:main-proofs}.

\paragraph{General convex populations.}
The lower bound in Theorem~\ref{thm:weak-growth-convex} extends the
lower bound of \citet{woodworth2021minmax} from bounded variance to
relative noise. As in their construction, the hard oracle reveals the
next direction of a hidden chain only with a small probability $p$, so
that the progress of the whole system within a round is the maximum of
$m$ binomial counts, which is of order $\log m$. 

Together with the statistical term, Theorem~\ref{thm:weak-growth-convex}
shows that under weak growth the minimax rate is attained by the best of
three baselines: returning $x_0$, serial SGD on one worker, and minibatch
AC-SA. This mirrors the conclusion of \citet{woodworth2021minmax} for
bounded variance, where the better of minibatch and single-machine
accelerated SGD is optimal: for general convex populations, local updates
beyond these baselines do not help.

\paragraph{Quadratic populations.}
For quadratic populations, local updates do help, because the gradients
are linear. The expected displacement of $K$ local SGD steps with step
size $1/H_s$ from a common point $y$ equals the gradient at $y$ of the
spectral surrogate
\[
    F_K(y)=\frac12\langle y-x^*,\mathsf A_K(y-x^*)\rangle,
    \qquad
    \mathsf A_K=\frac{H_s}K\Bigl[\mathrm I-\Bigl(\mathrm I
        -\frac{\mathsf A}{H_s}\Bigr)^K\Bigr],
\]
whose smoothness is $\min\{L,H_s/K\}$ (Lemma~\ref{quad:certificate}). An
accelerated outer method on $F_K$ therefore replaces $L$ by
$\min\{L,H_s/K\}$ in the communication term, while the averages of the
local iterates control the original objective $F$. For least squares,
linearity of the gradients was used to accelerate SGD by
\citet{liu2020accelerating}. The lower bound in
Theorem~\ref{thm:weak-growth-quadratic} is attained by sample-gradient
instances, which we describe in Section~\ref{sec:sample-gradients}.

We also note that our lower-bound construction may be useful
for establishing corresponding complexity lower bounds in
the centralized setting, thereby complementing the study
of stochastic optimization under the weak growth condition
initiated by Vaswani et al.~\cite{vaswani2019fast}.

\section{Fresh sample gradients: an open gap}
\label{sec:sample-gradients}

In this section we consider the oracle of Assumption~\ref{ass:sg}. It is
motivated by statistical learning, where $F(x)=\mathbb E_{z}\,\ell(x;z)$
is the expected loss over the data distribution and a stochastic
gradient is the gradient of the loss at a fresh example,
$g(x,z)=\nabla\ell(x;z)$; a minibatch gradient is the average of such
sample gradients. The interpolation condition $\nabla f_\xi(x^*)=0$
describes overparameterized models, in which one parameter vector fits
every example \citep{ma2018power,vaswani2019fast}.

This assumption gives the oracle a nontrivial structure. Every reply is
the gradient of a convex $H_s$-smooth random function that is minimized
at the same point $x^*$. We ask whether this structure leads to faster
rates than the weak growth condition alone, that is, whether the bounds
of Section~\ref{sec:weak-growth} can be improved. The basic consequence of
the structure is the weak growth condition itself; it is the workhorse of
all our upper bounds.

\begin{lemma}[Interpolation implies weak growth]
\label{lem:relations}
Let Assumption~\ref{ass:sg} hold. Then, for every $x\in\R^d$,
\[
    \mathbb E\,\nabla f_\xi(x)=\nabla F(x),
    \qquad
    \mathbb E\|\nabla f_\xi(x)\|^2\le2H_s\bigl(F(x)-F(x^*)\bigr),
\]
that is, the oracle $g(x,\xi)=\nabla f_\xi(x)$ satisfies
Assumption~\ref{ass:wg} with the same $F$ and $x^*$. Consequently,
$\mathfrak M_d^{\mathrm{sg}}\le\mathfrak M_d^{\mathrm{wg}}$ and
$\mathfrak M_d^{\mathrm{sg,quad}}\le\mathfrak M_d^{\mathrm{wg,quad}}$.
\end{lemma}

The proof is given in Appendix~\ref{app:B-prelim}.

\begin{theorem}[Quadratic populations with sample gradients]
\label{thm:sample-gradient-quadratic}
Let Assumption~\ref{ass:sg} hold with a quadratic objective $F$. Then,
in every dimension $d\ge1$,
\[
    \mathfrak M_d^{\mathrm{sg,quad}}
    \le 8R^2\min\left\{L,\;
        \frac{\min\{L,H_s/K\}}{N^2}+\frac{H_s}{mKN}\right\},
\]
and, for every $d\ge2048m(KN)^2$,
\[
    \mathfrak M_d^{\mathrm{sg,quad}}
    \ge\frac{R^2}{73728\log^2(em)}\min\left\{L,\;
        \frac{\min\{L,H_s/K\}}{N^2}+\frac{H_s}{mKN}\right\}.
\]
\end{theorem}

\begin{theorem}[General convex populations with sample gradients]
\label{thm:sample-gradient-convex}
Let Assumption~\ref{ass:sg} hold. Then, in every dimension $d\ge1$,
\[
    \mathfrak M_d^{\mathrm{sg}}
    \le 288R^2\min\left\{L,\;\frac{H_s}{KN},\;
        \frac{L}{N^2}+\frac{H_s}{mKN}\right\},
\]
and, for every $d\ge2048m(KN)^2$,
\[
    \mathfrak M_d^{\mathrm{sg}}
    \ge\frac{R^2}{73728\log^2(em)}\min\left\{L,\;
        \frac{\min\{L,H_s/K\}}{N^2}+\frac{H_s}{mKN}\right\}.
\]
\end{theorem}

By Lemma~\ref{lem:relations}, the upper bounds are those of
Theorems~\ref{thm:weak-growth-quadratic} and~\ref{thm:weak-growth-convex},
attained by the same algorithms. The lower bound is the same in both
theorems; it is proved in Appendix~\ref{app:lower-quadratic}, and the
proofs are completed in Appendix~\ref{app:main-proofs}.

\paragraph{The lower bound.}
The lower bound of \citet{woodworth2021minmax} for intermittent
communication does not apply under Assumption~\ref{ass:sg}. It is proved
for oracles whose variance is bounded by $\sigma^2$, and its hard oracle
returns a randomly truncated or rescaled gradient of a hidden chain, where
the truncation depends on the query. Such replies are not gradients of
fixed convex sample functions minimized at a common point, as
Assumption~\ref{ass:sg} requires. The same applies to the hard oracle of
Theorem~\ref{thm:weak-growth-convex}. We therefore build a new hard instance. The idea of
Bernoulli masking of entire sample losses was previously
used in our work on composite variational inequalities
\citep{novoselov2026optimal}. Related random-rotation
and polynomial approximation techniques were also employed
by \citet{shamir2026simple} to establish lower bounds
for centralized matrix-vector oracle algorithms. The sample functions are whole quadratics $f_\xi=(\xi/p)F$
masked by Bernoulli variables $\xi\sim\operatorname{Bernoulli}(p)$, so
that a worker receives the exact gradient of $F$ only at the queries with
$\xi=1$, and $F$ is a randomly rotated quadratic whose spectrum is placed
at Chebyshev nodes. A depth argument for adaptive matrix queries, without
any linear-span assumption, gives the term $\min\{L,H_s/K\}/N^2$, and a statistical term $H_s/(mKN)$
is a simple consequence of lower bound from \cite{woodworth2021even}.

\paragraph{The gap.}
For quadratic populations, Theorems~\ref{thm:weak-growth-quadratic}
and~\ref{thm:sample-gradient-quadratic} give the same rate: the structure
of sample gradients adds nothing to the weak growth condition, and the
improvement over general convex populations comes from the linearity of
the gradients. For general convex populations,
Theorem~\ref{thm:sample-gradient-convex} leaves a gap between the rates
\[
    R^2\min\left\{L,\frac{H_s}{KN},\frac{L}{N^2}+\frac{H_s}{mKN}\right\}
    \qquad\text{and}\qquad
    R^2\min\left\{L,\frac{\min\{L,H_s/K\}}{N^2}+\frac{H_s}{mKN}\right\}.
\]
Theorem~\ref{thm:weak-growth-convex} shows that the upper bound cannot be
improved by an algorithm whose guarantee relies only on
Lemma~\ref{lem:relations}: any improvement must use the structure of
Assumption~\ref{ass:sg} beyond the weak growth condition. On the other
hand, Theorem~\ref{thm:pool} shows that the lower bound is attained, up
to logarithmic factors, when samples can be reused. The gap is therefore
the question of whether fresh access to samples is strictly weaker than
reuse for general convex populations.

\section{Sample reuse: closing the gap}
\label{sec:reuse}

In statistical learning the training set is usually stored, and the loss
of the same example can be differentiated at several points, as in
multi-epoch training. Assumption~\ref{ass:pool} models this situation:
the samples form a pool of size $S$, and a query may reuse any label.
With fresh samples, each sample function is observed through a single
gradient. With reuse, a worker can query its own batch of samples
repeatedly within a round and therefore optimize the empirical objective
of the batch, rather than only estimate its gradient.

This possibility gives a new upper bound. Within one round, a worker
takes a batch $\mathsf S$ of $n$ samples and approximately computes, by
SVRG \citep{johnson2013accelerating} on this batch, the proximal point
\[
    p_{\mathsf S}(y)=\arg\min_u\Bigl\{f_{\mathsf S}(u)+\frac\tau2\|u-y\|^2\Bigr\},
    \qquad
    f_{\mathsf S}=\frac1n\sum_{\xi\in\mathsf S}f_\xi,
\]
at the common point $y$, using gradient queries only. The vector
$\tau(y-p_{\mathsf S}(y))$ is the gradient of the empirical Moreau
envelope of the batch, which is $\tau$-smooth with $\tau=8H_s/n$ and, by
interpolation, is minimized at the same point $x^*$ as every sample
function. The accelerated method of \citet{woodworth2021even}, applied to
these envelopes, therefore sees the smoothness $\tau$, of order
$H_s\log(KN)/K$, instead of $L$, and a final proximal step on a fresh
batch transfers the guarantee from the envelope back to $F$. Thus reuse
plays for general convex populations the role that the linearity of the
gradients plays for quadratic populations in
Section~\ref{sec:weak-growth}: it replaces $L$ by roughly $H_s/K$ in the
communication term. The scheme is related to minibatch-prox
\citep{wang2017memory}, where every machine solves a regularized
subproblem on its minibatch. It is stated as
Algorithm~\ref{alg:envelope}.

\begin{theorem}[Minimax bounds with sample reuse]
\label{thm:pool}
Let Assumption~\ref{ass:pool} hold, and let
\[
    J=\left\lceil\log_2\bigl(625K(K+8)(N+1)^4\bigr)\right\rceil.
\]
Then, in every dimension $d\ge1$,
\[
    \mathfrak M_d^{\mathrm{pool}}(S)
    \le10^7J\,R^2\min\left\{L,\;
        \frac{\min\{L,H_s/K\}}{N^2}+\frac{H_s}{\min\{S,mKN\}}\right\},
\]
and, for every $d\ge d_{\mathrm{pool}}$,
\[
    \mathfrak M_d^{\mathrm{pool}}(S)
    \ge\frac{R^2}{10^6\log^2(em)}\min\left\{L,\;
        \frac{\min\{L,H_s/K\}}{N^2}+\frac{H_s}{\min\{S,mKN\}}\right\},
\]
where $d_{\mathrm{pool}}<\infty$ is defined in \eqref{eq:d-pool} and
depends only on $m$, $K$, $N$, and $H_s/L$.
\end{theorem}

The upper bound is achieved by a parameter-based selection among
returning $x_0$, serial averaged SGD on one worker, common-query
minibatch AC-SA, and Algorithm~\ref{alg:envelope}; it is proved in
Appendix~\ref{app:appendix_C}. The proof is completed in
Appendix~\ref{app:main-proofs}.

\paragraph{The lower bound.}
The lower bound holds for arbitrary adaptive reuse of labels. In the hard
instance every sample function carries its own Bernoulli coins on the
edges of a hidden chain, so repeated queries of a sample reveal nothing
beyond its coins (Appendix~\ref{app:lower-pool}). The statistical term
depends on the number $\min\{S,mKN\}$ of distinct samples that can be
inspected: with a small pool, more queries do not help.

For $S=\infty$, both bounds of Theorem~\ref{thm:pool} coincide with the
lower bound of Theorem~\ref{thm:sample-gradient-convex} up to the factor
$J\log^2(em)$, while
$\mathfrak M_d^{\mathrm{pool}}(\infty)\le\mathfrak M_d^{\mathrm{sg}}$ by
\eqref{eq:pool-fresh}. The lower bound of
Theorem~\ref{thm:sample-gradient-convex} is therefore the minimax rate
with reuse, and the open question of Section~\ref{sec:sample-gradients} is
whether the same rate can be attained with fresh samples.

\section{Discussion and Future work}
\label{sec:discussion}

In this work, we establish upper and lower complexity bounds
for distributed stochastic convex optimization under several
oracle models motivated by statistical learning. Our results
demonstrate the importance of oracle assumptions in determining
the minimax complexity of intermittent communication.

Under the weakest oracle model, which assumes only the weak
growth condition, we show that for general convex objectives,
local updates cannot improve the minimax rate beyond that
achieved by simple baselines. However, for quadratic population
objectives, local computation becomes beneficial and leads to
an improved communication-dependent term. The sample-reuse
oracle extends this improvement to general convex objectives,
allowing an accelerated rate through repeated access to
individual sample functions.

These findings highlight that the weak growth condition alone
does not fully capture the potential benefits of local
computation. Additional structure, either in the objective
function or in the oracle's access to individual samples,
can fundamentally change the achievable convergence rates.
Nevertheless, several important questions remain open.

\paragraph{The gap for fresh sample gradients.}
For general convex population objectives with fresh sample
gradients, Theorem~\ref{thm:sample-gradient-convex} leaves
a gap between the upper and lower bounds.
By Theorem~\ref{thm:weak-growth-convex}, any improvement
of the upper bound must exploit properties of sample
gradients beyond the weak growth condition.
On the other hand, Theorem~\ref{thm:pool} shows that
the rate appearing in the lower bound can be attained,
up to logarithmic factors, when repeated access to
samples is allowed.
It remains open whether the gap can be closed by an
algorithm using only fresh samples, or whether repeated
sample access leads to a genuine separation between
the two oracle models.

\paragraph{Beyond convexity.}
Convex optimization provides a natural starting point,
but many overparameterized learning models are inherently
nonconvex. Nevertheless, certain overparameterized
nonlinear systems, including sufficiently wide neural
networks, have been shown to satisfy variants of the
Polyak--{\L}ojasiewicz condition, such as PL$^*$,
under suitable assumptions and in appropriate regions
of the parameter space~\citep{liu2022loss};
see also \citet{karimi2016linear}.

This motivates studying the minimax complexity of
intermittent communication for nonconvex interpolating
problems satisfying PL-type conditions.
In particular, it would be interesting to determine
whether the advantages of local computation observed
for quadratic objectives and under sample reuse
persist beyond convexity, and how the choice of
stochastic oracle affects the resulting minimax rates.

\bibliographystyle{plainnat}
\bibliography{references}

\clearpage
\appendix
\section{Lower bounds}
\label{app:appendix_A}

\paragraph{Conventions.}
By translation, $x_0=0$ in all constructions. Expectations of
nonnegative random variables are understood in the extended sense, so an
algorithm with infinite expected risk satisfies every lower bound. We
use two elementary reductions.
\begin{enumerate}
    \item
    \emph{Averaging over instances.} Let a random instance be drawn from
    a law that may depend on the algorithm but is independent of its seed
    and of the oracle randomness. If the expected risk, averaged also
    over the instance, is at least $c$, then some fixed instance has
    expected risk at least $c$.

    \item
    \emph{Padding.} An instance in $\R^{d_0}$ is extended to $\R^d$,
    $d\ge d_0$, by letting all functions and oracle maps depend only on
    the first $d_0$ coordinates, with zero replies in the remaining
    coordinates. The extension belongs to the same class with the same
    constants, and an algorithm in $\R^d$ is simulated in $\R^{d_0}$ by
    projecting its queries and its output onto the first $d_0$
    coordinates. Hence a lower bound in dimension $d_0$ holds in every
    dimension $d\ge d_0$.
\end{enumerate}

\subsection{Tools}
\label{app:lower-tools}

\begin{lemma}[Spherical caps]
\label{lem:sphere}
Let $u$ be uniformly distributed on the unit sphere of an
$n$-dimensional Euclidean space, $n\ge2$, and let $z$ be a fixed vector
with $\|z\|\le\bar r$. Then, for every $b>0$,
\[
    \mathbb{P}\bigl(|\langle u,z\rangle|>b\bigr)
    \le2\exp\left(-\frac{(n-1)b^2}{2\bar r^2}\right).
\]
The same bound holds conditionally if $z$ is measurable with respect to
a $\sigma$-algebra given which $u$ is uniform on such a sphere.
\end{lemma}

\begin{proof}
If $b\ge\|z\|$, the probability is zero. Otherwise put
$c=b/\|z\|\in(0,1)$; by rotation invariance the probability equals
$\mathbb{P}\left(|u_1|>c\right)$. Write $u=Z/\|Z\|$ with a standard Gaussian vector
$Z\in\R^n$, and let $W=\sum_{j=2}^nZ_j^2$. Then
$|u_1|>c$ if and only if $|Z_1|>\sqrt{c^2W/(1-c^2)}$. Since $Z_1$ is
independent of $W$ and $\mathbb{P}\left(|Z_1|>a\right)\le2e^{-a^2/2}$ for $a\ge0$,
\[
    \mathbb{P}\left(|u_1|>c\right)
    \le2\,\mathbb E\exp\left(-\frac{c^2W}{2(1-c^2)}\right)
    =2(1-c^2)^{(n-1)/2}
    \le2\exp\left(-\frac{(n-1)c^2}{2}\right),
\]
where the equality is the Laplace transform of the $\chi^2_{n-1}$
distribution. It remains to use $c\ge b/\bar r$. The conditional version is
the same computation for every fixed value of the conditioning
information.
\end{proof}

\begin{lemma}[Maximum of success counts]
\label{lem:max-binomial}
Let $p\in[0,1]$. For every $w=1,\ldots,m$, let
$\mathcal G\subseteq\mathcal G_{w,0}\subseteq\mathcal G_{w,1}
\subseteq\cdots\subseteq\mathcal G_{w,K}$ be $\sigma$-algebras and let
$I_{w,k}\in\{0,1\}$ be $\mathcal G_{w,k}$-measurable with
\[
    \mathbb E\bigl[e^{I_{w,k}}\mid\mathcal G_{w,k-1}\bigr]\le1-p+pe,
    \qquad k=1,\ldots,K.
\]
Let $S_w=\sum_{k=1}^KI_{w,k}$. Then
\[
    \mathbb E\Bigl[e^{\max_wS_w}\,\Big|\,\mathcal G\Bigr]\le me^{(e-1)Kp},
    \qquad
    \mathbb E\Bigl[\max_wS_w\,\Big|\,\mathcal G\Bigr]\le\log m+(e-1)Kp.
\]
In particular, this holds for independent
$S_1,\ldots,S_m\sim\operatorname{Bin}(K,p)$.
\end{lemma}

\begin{proof}
Conditioning successively on $\mathcal G_{w,K-1},\ldots,\mathcal G_{w,0}$
and using $1-p+pe\le e^{(e-1)p}$ gives
$\mathbb E[e^{S_w}\mid\mathcal G_{w,0}]\le e^{(e-1)Kp}$, and the same
bound holds given $\mathcal G$. Since $e^{\max_wS_w}\le\sum_we^{S_w}$,
the first inequality follows. The second follows from the first by
Jensen's inequality,
$\mathbb E[\max_wS_w\mid\mathcal G]
\le\log\mathbb E[e^{\max_wS_w}\mid\mathcal G]$.
\end{proof}

The next lemma turns a convex function with bounded gradients into a
globally convex function whose values and gradients far from the origin
do not depend on it. Let
\[
    \Gamma_0(t)=
    \begin{cases}
        0,& t\le0,\\
        3t^2-2t^3,& 0<t<1,\\
        1,& t\ge1 .
    \end{cases}
\]

\begin{lemma}[Radial extension]
\label{lem:radial}
Let $\varphi:\R^d\to[0,\infty)$ be convex and differentiable, with an
$L_\varphi$-Lipschitz gradient and $\|\nabla\varphi\|\le M_\varphi$
everywhere, and let $\varphi(x^*)=0$ for some $\|x^*\|\le R$. Let $c>0$
and $\rho\ge\max\{16R,64M_\varphi/c\}$, and define
\begin{equation}
    \Gamma(x)=\Gamma_0\Bigl(\frac{\|x\|-2R}{\rho}\Bigr),
    \qquad
    \mathcal T\varphi(x)
    =\bigl(1-\Gamma(x)\bigr)\varphi(x)+2c\bigl(\|x\|-R\bigr)_+^2 .
    \label{eq:radial}
\end{equation}
Then:
\begin{enumerate}
    \item
    $\mathcal T\varphi$ is convex and has an $(L_\varphi+5c)$-Lipschitz
    gradient;

    \item
    $\varphi\le\mathcal T\varphi$ and $\mathcal T\varphi(x^*)=0$;

    \item
    $\|\nabla\Gamma\|\le\frac3{2\rho}$ and
    $\nabla\mathcal T\varphi
    =(1-\Gamma)\nabla\varphi-\varphi\nabla\Gamma
    +4c(\|x\|-R)_+\frac{x}{\|x\|}$, with the last term equal to zero at
    $x=0$;

    \item
    if $\|x\|\ge2R+\rho$, then
    $\mathcal T\varphi(x)=2c(\|x\|-R)^2$ and
    $\nabla\mathcal T\varphi(x)=4c(\|x\|-R)\frac{x}{\|x\|}$; these do not
    depend on $\varphi$.
\end{enumerate}
\end{lemma}

\begin{proof}
Write $r=\|x\|$ and $\pi(x)=2c(r-R)_+^2$. The function $\pi$ is convex,
being a nondecreasing convex function of the norm, and for $r>R$
\[
    \nabla\pi(x)=4c(r-R)\frac x r,
    \qquad
    \nabla^2\pi(x)=4c\frac{xx^\top}{r^2}
       +4c\Bigl(1-\frac Rr\Bigr)\Bigl(\mathrm I-\frac{xx^\top}{r^2}\Bigr).
\]
Hence the eigenvalues of $\nabla^2\pi$ lie in $[0,4c]$, and they are at
least $2c$ when $r\ge2R$.

\emph{Bounds on the cutoff.}
On the annulus $2R<r<2R+\rho$ put $t=(r-2R)/\rho$. Then
\[
    \nabla\Gamma(x)=\frac{\Gamma_0'(t)}{\rho}\frac xr,
    \qquad
    \nabla^2\Gamma(x)=\frac{\Gamma_0''(t)}{\rho^2}\frac{xx^\top}{r^2}
       +\frac{\Gamma_0'(t)}{\rho r}\Bigl(\mathrm I-\frac{xx^\top}{r^2}\Bigr).
\]
Since $0\le\Gamma_0'(t)=6t(1-t)\le\min\{3/2,6t\}$ and $|\Gamma_0''(t)|\le6$, and since
$\frac{\Gamma_0'(t)}{\rho r}\le\frac{6(r-2R)}{\rho^2r}\le\frac6{\rho^2}$, we get
\[
    \|\nabla\Gamma\|\le\frac3{2\rho},
    \qquad
    \|\nabla^2\Gamma\|_{\mathrm{op}}\le\frac6{\rho^2}.
\]
Outside the annulus $\Gamma$ is constant. This proves the bound in
Part~3; the formula for $\nabla\mathcal T\varphi$ is the product rule.

\emph{Hessian bounds.}
The gradient of $\varphi$ is Lipschitz, hence differentiable almost
everywhere, and at every such point
$0\preceq\nabla^2\varphi\preceq L_\varphi\mathrm I$ by convexity. Since
$\varphi(x^*)=0$ and $\|\nabla\varphi\|\le M_\varphi$, on the annulus
$\varphi(x)\le M_\varphi\|x-x^*\|\le M_\varphi(3R+\rho)$. The product
rule gives, almost everywhere on the annulus,
\[
    \nabla^2\bigl[(1-\Gamma)\varphi\bigr]
    =(1-\Gamma)\nabla^2\varphi
     -\nabla\Gamma\,\nabla\varphi^\top-\nabla\varphi\,\nabla\Gamma^\top
     -\varphi\nabla^2\Gamma ,
\]
and the last three terms have operator norm at most
\[
    2\cdot\frac3{2\rho}M_\varphi+M_\varphi(3R+\rho)\frac6{\rho^2}
    =\Bigl(9+\frac{18R}{\rho}\Bigr)\frac{M_\varphi}{\rho}
    \le\frac{81}8\cdot\frac c{64}<\frac c4,
\]
where we used $\rho\ge16R$ and $\rho\ge64M_\varphi/c$. Combining this
with the eigenvalue bounds for $\nabla^2\pi$, we obtain almost
everywhere: for $r<2R$, where $\Gamma=0$,
$0\preceq\nabla^2\mathcal T\varphi\preceq(L_\varphi+4c)\mathrm I$; on
the annulus,
$(2c-\frac c4)\mathrm I\preceq\nabla^2\mathcal T\varphi
\preceq(L_\varphi+\frac c4+4c)\mathrm I$; for $r>2R+\rho$, where
$\Gamma=1$, $0\preceq\nabla^2\mathcal T\varphi\preceq4c\mathrm I$. Thus
\begin{equation}
    0\preceq\nabla^2\mathcal T\varphi\preceq(L_\varphi+5c)\mathrm I
    \qquad\text{almost everywhere.}
    \label{eq:radial-hessian}
\end{equation}

\emph{Part 1.}
The function $\psi=\mathcal T\varphi$ is continuously differentiable,
and $\nabla\psi$ is locally Lipschitz. Hence the distributional
derivative of $\nabla\psi$ is its almost everywhere derivative. For a
standard mollifier $\kappa_\varepsilon$, the function
$\psi_\varepsilon=\psi*\kappa_\varepsilon$ is smooth with
$\nabla^2\psi_\varepsilon=(\nabla^2\psi)*\kappa_\varepsilon$, so
\eqref{eq:radial-hessian} gives
$0\preceq\nabla^2\psi_\varepsilon\preceq(L_\varphi+5c)\mathrm I$
everywhere. Thus $\psi_\varepsilon$ is convex with an
$(L_\varphi+5c)$-Lipschitz gradient. As $\varepsilon\to0$,
$\psi_\varepsilon\to\psi$ and $\nabla\psi_\varepsilon\to\nabla\psi$
pointwise, and both properties pass to the limit.

\emph{Part 2.}
We show $\Gamma\varphi\le\pi$. For $r\le2R$ this holds since $\Gamma=0$.
On the annulus, $\Gamma_0(t)\le3t^2$ and $t\le(r-R)/\rho$ give
$\Gamma\le3(r-R)^2/\rho^2$, hence
\[
    \Gamma\varphi
    \le\Bigl(3+\frac{9R}\rho\Bigr)\frac{M_\varphi}{\rho}(r-R)^2
    \le\frac{57}{16}\cdot\frac c{64}(r-R)^2\le2c(r-R)^2 .
\]
For $r\ge2R+\rho$, we have $r-R\ge r/2$, $r+R\le3r/2$, and
$r\ge64M_\varphi/c$, so
$\varphi(x)\le M_\varphi(r+R)\le\frac32M_\varphi r\le\frac c2r^2\le\pi(x)$.
Therefore $\mathcal T\varphi-\varphi=\pi-\Gamma\varphi\ge0$. Since
$\|x^*\|\le R$, both $\Gamma$ and $\pi$ vanish at $x^*$, and
$\mathcal T\varphi(x^*)=\varphi(x^*)=0$.

\emph{Part 4.}
This follows from $\Gamma=1$ for $r\ge2R+\rho$.
\end{proof}

The map $\mathcal T$ is affine in $\varphi$: if a measurable family
$\varphi_\zeta$ satisfies the hypotheses of Lemma~\ref{lem:radial} with
common $x^*$, $L_\varphi$, $M_\varphi$, then so does
$\mathbb E\varphi_\zeta$, and
$\mathbb E\,\mathcal T\varphi_\zeta=\mathcal T(\mathbb E\varphi_\zeta)$.

\subsection{The statistical term}
\label{app:lower-stat}

\begin{lemma}[Statistical lower bound]
\label{lem:stat}
Let $S\in\mathbb N\cup\{\infty\}$, $S\ge1$, and $d\ge1$. For every
algorithm in $\mathrm{IC}(m,K,N)$ with access to a pool of size $S$
there is an instance of $\mathcal S_d^{\mathrm{quad}}$ such that
\begin{equation}
    \mathbb E F(\widehat x)
    \ge\frac{R^2}8\min\left\{L,\frac{H_s}{\min\{S,mKN\}}\right\}.
    \label{eq:stat}
\end{equation}
Consequently, $\mathfrak M_d^{\mathrm{pool}}(S)$ is at least the
right-hand side of \eqref{eq:stat}, and each of
$\mathfrak M_d^{\mathrm{sg,quad}}$, $\mathfrak M_d^{\mathrm{sg}}$,
$\mathfrak M_d^{\mathrm{wg,quad}}$, $\mathfrak M_d^{\mathrm{wg}}$ is at
least $\frac{R^2}8\min\{L,\frac{H_s}{mKN}\}$.
\end{lemma}

\begin{proof}
Put $n_S=\min\{S,mKN\}$ and $p_0=\min\{L/H_s,1/(2n_S)\}$. Let
$\theta\in\{-1,1\}$ be a uniform sign, independent of the seed, and
consider the sample functions
\[
    f_\xi(x)=\frac{\xi H_s}{2}(x_1-\theta R)^2,
    \qquad \xi\sim\operatorname{Bernoulli}(p_0).
\]
They are convex, $H_s$-smooth, and minimized at $x^*=\theta Re_1$, and
$F_\theta(x)=\frac{p_0H_s}2(x_1-\theta R)^2$ is a quadratic with
smoothness $p_0H_s\le L$. Hence, for each $\theta$, this is an instance
of $\mathcal S_d^{\mathrm{quad}}$.

Consider the \emph{zero run}: the run of the algorithm in which every
query with a valid label receives the reply $0$. It is a function of
the seed only; let $\widehat X$ be its output and $\mathcal I$ the set of
valid labels it queries. Since there are at most $mKN$ queries and at
most $S$ valid labels, $|\mathcal I|\le n_S$. On the event
$E_0=\{\xi_i=0\text{ for all }i\in\mathcal I\}$, induction over the
queries shows that the actual run coincides with the zero run: every
query is the same, and its reply is $0$. Since $\mathcal I$ depends only
on the seed, which is independent of the pool and of $\theta$,
\[
    \mathbb{P}(E_0\mid\text{seed},\theta)
    =(1-p_0)^{|\mathcal I|}\ge1-n_Sp_0\ge\frac12 .
\]
Therefore, using $\widehat x=\widehat X$ on $E_0$,
\[
    \mathbb E F_\theta(\widehat x)
    \ge\mathbb E\bigl[F_\theta(\widehat X)\,
        \mathbb{P}(E_0\mid\text{seed},\theta)\bigr]
    \ge\frac12\,\mathbb E F_\theta(\widehat X).
\]
Since $\widehat X$ is independent of $\theta$,
\[
    \mathbb E_\theta F_\theta(\widehat X)
    =\frac{p_0H_s}{4}\bigl[(\widehat X_1-R)^2+(\widehat X_1+R)^2\bigr]
    =\frac{p_0H_s}{2}\bigl(\widehat X_1^2+R^2\bigr)
    \ge\frac{p_0H_sR^2}2 .
\]
Averaged over $\theta$, the risk is at least
$\frac{p_0H_sR^2}4=\frac{R^2}4\min\{L,\frac{H_s}{2n_S}\}$, which is at
least the right-hand side of \eqref{eq:stat}. One of the two signs
attains the average.

Fresh access is the case $S=\infty$ with distinct labels, see
\eqref{eq:pool-fresh}, for which $n_S=mKN$. Since the hard instance
belongs to $\mathcal S_d^{\mathrm{quad}}$, the bound holds for
$\mathfrak M_d^{\mathrm{sg,quad}}$, and Lemma~\ref{lem:relations}
together with $\mathcal S_d^{\mathrm{quad}}\subseteq\mathcal S_d$ and
$\mathcal W_d^{\mathrm{quad}}\subseteq\mathcal W_d$ transfers it to the
other three minimax risks.
\end{proof}

\subsection{Quadratic populations}
\label{app:lower-quadratic}

The construction first gives one law of convex quadratics that is hard
for every algorithm with a bounded number of batches of exact gradient
queries, and then masks whole sample functions by Bernoulli variables
to convert the number of successful local queries into that number of
batches. No linear-span assumption is used.

\paragraph{Adaptive matrix queries.}

\begin{lemma}[Batched Haar representation]
\label{qlow:haar}
Let $\Lambda\in\R^{d\times d}$ be symmetric and $b_0\in\R^d$. Let $U$ be
Haar distributed on the orthogonal group $O(d)$, and put
$A=U\Lambda U^\top$ and $b=Ub_0$. Consider a procedure that knows
$\Lambda$ and $b_0$, receives $b$, then in $D$ adaptive batches submits
at most $m$ vectors per batch and receives their images under $A$, and
finally outputs a vector $z$; its seed is independent of $U$ and all its
maps are measurable. Then on one probability space one can construct
$U$, the seed, the transcript of the procedure, and independent standard
Gaussian vectors $g_{r,i}\in\R^d$, $1\le r\le D$, $1\le i\le m$, and
$g_0\in\R^d$, such that $(U,\text{seed},\text{transcript})$ has the same
law as in the original experiment, the Gaussian vectors are independent
of the seed, and
\begin{equation}
    U^\top z\in
    \operatorname{span}\{\Lambda^jb_0:0\le j\le D\}
    +\sum_{r=1}^D\sum_{i=1}^m
     \operatorname{span}\{\Lambda^jg_{r,i}:0\le j\le D-r+1\}
    +\operatorname{span}\{g_0\}.
    \label{qlow:representation}
\end{equation}
\end{lemma}

\begin{proof}
\emph{Uniform extensions.}
Let $E\subseteq\R^d$ be a subspace and $W:E\to\R^d$ a linear isometry,
and let $O$ be uniformly distributed on the set $\mathcal O(W)$ of
orthogonal maps that coincide with $W$ on $E$; equivalently, the
restriction of $O$ to $E^\perp$ is a Haar distributed isometry of
$E^\perp$ onto $(WE)^\perp$. Invariance of this Haar measure under
rotations of $E^\perp$ and of $(WE)^\perp$ gives the following two
facts.
\begin{enumerate}
    \item[(i)]
    If $w_1,\ldots,w_k$ is an orthonormal family in $(WE)^\perp$, then
    $(O^\top w_1,\ldots,O^\top w_k)$ is a uniformly distributed
    orthonormal $k$-frame in $E^\perp$, and conditional on this frame
    $O$ is uniform on the set of orthogonal maps that extend $W$ and map
    the frame to $(w_1,\ldots,w_k)$. The same frame law is obtained by
    the Gram--Schmidt orthonormalization of
    $\Pi_{E^\perp}\gamma_1,\ldots,\Pi_{E^\perp}\gamma_k$ for independent
    standard Gaussian vectors $\gamma_l\in\R^d$; every vector of that
    frame lies in $E+\operatorname{span}\{\gamma_1,\ldots,\gamma_k\}$.

    \item[(ii)]
    If $e_1,\ldots,e_k$ is an orthonormal family in $E^\perp$, then
    $(Oe_1,\ldots,Oe_k)$ is a uniformly distributed orthonormal
    $k$-frame in $(WE)^\perp$, and conditional on it $O$ is uniform on the
    set of orthogonal maps that extend $W$ and map $(e_l)$ to this
    frame.
\end{enumerate}

\emph{Construction.}
Fix the seed. We reveal $U$ progressively. At every stage there is an
exposed subspace $E$ and the isometry $W=U|_E$, and the following
invariant holds: conditional on everything constructed so far, $U$ is
uniform on $\mathcal O(W)$.

Initially, draw $b$ uniformly on the sphere of radius $\|b_0\|$ and
expose $E_0=\operatorname{span}\{b_0\}$ with $Ub_0=b$. This is the law
of $Ub_0$, and conditional on $Ub_0=b$ the map $U$ is uniform on
$\mathcal O(W)$.

Consider batch $r\in\{1,\ldots,D\}$, and let $E_{r-1}$ be the exposed
subspace before it. The queries $v_1,\ldots,v_m$ (padded by zeros) are
functions of the seed and of the previous replies. Let
$w_1,\ldots,w_k$, $k\le m$, be an orthonormal basis, chosen by a fixed
measurable rule, of the orthogonal projection of
$\operatorname{span}\{v_1,\ldots,v_m\}$ onto $(UE_{r-1})^\perp$. Draw
standard Gaussian vectors $g_{r,1},\ldots,g_{r,m}$ independently of
everything constructed so far, let $q_1,\ldots,q_k$ be the Gram--Schmidt
orthonormalization of $\Pi_{E_{r-1}^\perp}g_{r,1},\ldots,
\Pi_{E_{r-1}^\perp}g_{r,k}$, and expose $Uq_l=w_l$. By~(i), this has the
correct conditional law and preserves the invariant with
$E'=E_{r-1}+\operatorname{span}\{q_1,\ldots,q_k\}$. Every $v_i$ lies in
$UE'$, so $U^\top v_i\in E'\subseteq E_{r-1}+
\operatorname{span}\{g_{r,1},\ldots,g_{r,m}\}$. The replies are
$Av_i=U(\Lambda U^\top v_i)$ with $\Lambda U^\top v_i\in\Lambda E'$. Let
$e_1,\ldots,e_{k'}$ be an orthonormal basis of the projection of
$\Lambda E'$ onto $E'^\perp$, and expose $Ue_1,\ldots,Ue_{k'}$ by~(ii),
using auxiliary randomness independent of everything else. The exposed
subspace becomes $E_r=E'+\Lambda E'$, the isometry on $E_r$ determines
all replies of the batch, and
\begin{equation}
    E_r\subseteq E_{r-1}+\Lambda E_{r-1}
    +\operatorname{span}\{g_{r,i},\Lambda g_{r,i}:1\le i\le m\}.
    \label{qlow:exposed}
\end{equation}

After the last batch, the output $z$ is a function of the seed and of
all replies. Draw $g_0$ independently of everything. If the projection
of $z$ onto $(UE_D)^\perp$ is nonzero, let $w$ be its normalization, let
$q$ be the normalization of $\Pi_{E_D^\perp}g_0$, and expose $Uq=w$
by~(i). In both cases $U^\top z\in E_D+\operatorname{span}\{g_0\}$.
Finally, draw $U$ from the uniform law on $\mathcal O(W)$ for the final
exposed isometry $W$. On null events where a Gram--Schmidt step
degenerates, any fixed measurable convention may be used.

\emph{Verification.}
By the invariant, every exposed value has the conditional law prescribed
by the original experiment given the previous ones, and the final
completion has the correct conditional law. Hence, for every fixed seed,
$(U,\text{transcript})$ has the original joint law; in particular $U$ is
Haar distributed and independent of the seed. Each Gaussian vector is
drawn independently of everything constructed before it, so the Gaussian
vectors are independent standard Gaussian vectors, independent of the
seed. Finally, \eqref{qlow:exposed} and $E_0=\operatorname{span}\{b_0\}$
give by induction
\[
    E_r\subseteq\operatorname{span}\{\Lambda^jb_0:0\le j\le r\}
    +\sum_{r'=1}^r\sum_{i=1}^m
     \operatorname{span}\{\Lambda^jg_{r',i}:0\le j\le r-r'+1\},
\]
because multiplication by $\Lambda$ raises every degree by one and the
new terms of batch $r$ have degrees $0$ and $1=r-r+1$. Taking $r=D$ and
using $U^\top z\in E_D+\operatorname{span}\{g_0\}$ proves
\eqref{qlow:representation}.
\end{proof}

\paragraph{Two certificates.}

\begin{lemma}[Chebyshev derivative certificate]
\label{qlow:scalar}
Let $n$ be odd, $a>0$, and $\lambda_j=a\cos(j\pi/n)$, $0\le j\le n$.
There are positive weights $w_j$ with $\sum_{j=0}^nw_j=1$ such that
\[
    \sum_{j=0}^n w_j\,\chi(\lambda_j)^2\ge \frac{a^2}{n^2}
    \qquad\text{for every polynomial $\chi$ with}\quad
    \deg\chi\le n,\quad\chi'(0)=1.
\]
\end{lemma}

\begin{proof}
The nodes are distinct, strictly decreasing, nonzero because $n$ is odd,
and symmetric: $\lambda_{n-j}=-\lambda_j$. Let $\ell_j$ be the Lagrange
basis polynomials of degree $n$ on these nodes, $c_j=\ell_j'(0)$, and
$P(z)=\prod_{j=0}^n(z-\lambda_j)$. The polynomial $P$ has $n+1$ roots
symmetric about zero, so it is even and $P'(0)=0$. Differentiating
$\ell_j(z)=P(z)/[(z-\lambda_j)P'(\lambda_j)]$ at zero gives
\[
    c_j=\frac{-\lambda_jP'(0)-P(0)}{\lambda_j^2P'(\lambda_j)}
    =-\frac{P(0)}{\lambda_j^2P'(\lambda_j)} .
\]
Since the nodes decrease, $P'(\lambda_j)=\prod_{i\ne j}(\lambda_j-\lambda_i)$
has sign $(-1)^j$, so all products $(-1)^jc_j$ have the same sign. The
Chebyshev polynomial $T_n$, defined by $T_n(\cos\vartheta)=\cos(n\vartheta)$,
satisfies $T_n(\lambda_j/a)=(-1)^j$. Interpolation of $T_n(z/a)$, a
polynomial of degree $n$, is exact, hence
\[
    \sum_{j=0}^n|c_j|
    =\Bigl|\sum_{j=0}^n(-1)^jc_j\Bigr|
    =\Bigl|\frac{d}{dz}T_n(z/a)\Big|_{z=0}\Bigr|
    =\frac{|T_n'(0)|}a=\frac na,
\]
where $|T_n'(0)|=n|\sin(n\pi/2)|=n$ follows by differentiating
$T_n(\cos\vartheta)=\cos(n\vartheta)$ at $\vartheta=\pi/2$. Put
$w_j=|c_j|a/n$. For $\chi$ as in the statement, exact interpolation and
the Cauchy--Schwarz inequality give
\[
    1=\chi'(0)^2=\Bigl(\sum_{j=0}^nc_j\chi(\lambda_j)\Bigr)^2
    \le\Bigl(\sum_{j=0}^n|c_j|\Bigr)\sum_{j=0}^n|c_j|\chi(\lambda_j)^2
    =\frac{n^2}{a^2}\sum_{j=0}^nw_j\chi(\lambda_j)^2 .
\]
\end{proof}

\begin{lemma}[Gaussian projection loss]
\label{qlow:projection}
Let $t,k,h\ge1$ be integers with $h\ge\max\{16k,32\log(8t)\}$. For
$j=1,\ldots,t$, let $S_j$ be the span of $k$ independent standard
Gaussian vectors in $\R^h$ (dependence between different $j$ is allowed),
and let $e_j\in\R^h$ be fixed unit vectors. Then
\[
    \mathbb E\sqrt{1-\max_{j\le t}\|\Pi_{S_j}e_j\|^2}
    \ge\frac3{4\sqrt2}>\frac12 .
\]
\end{lemma}

\begin{proof}
Fix $j$. By rotation invariance, $\|\Pi_{S_j}e_j\|^2$ has the law of
$X/Z$, where $X=\sum_{i=1}^kG_i^2$ and $Z=\sum_{i=1}^hG_i^2$ for
independent standard Gaussian $G_i$. The event $\{X/Z>1/2\}$ is contained
in $\{X>h/4\}\cup\{Z<h/2\}$. By the Chernoff bound with the Gaussian
moment generating function, and using $k\le h/16$ and
$\frac23\le\log2\le1$,
\[
    \mathbb{P}\left(X>h/4\right)\le e^{-h/16}\,2^{k/2}\le e^{-h/16+h/32}=e^{-h/32},
    \qquad
    \mathbb{P}\left(Z<h/2\right)\le e^{h/4}\,2^{-h/2}\le e^{-h/12}.
\]
A union bound over $j$ shows that some $\|\Pi_{S_j}e_j\|^2$ exceeds $1/2$
with probability at most $2te^{-h/32}\le1/4$. On the complementary event
the square root is at least $1/\sqrt2$.
\end{proof}

\begin{lemma}[Residual after Gaussian polynomial information]
\label{qlow:residual}
Let $n$, $a$, $\lambda_j$, and $w_j$ be as in
Lemma~\ref{qlow:scalar}. Let $\R^{(n+1)h}=V_0\oplus\cdots\oplus V_n$ be an
orthogonal decomposition into blocks of dimension $h$, let
$\Lambda|_{V_j}=\lambda_j\mathrm I$, and fix unit vectors $e_j\in V_j$.
Put
\[
    z_0=R\sum_{j=0}^n\sqrt{w_j}\,e_j,
    \qquad b_0=\Lambda z_0 ,
\]
let $g_1,\ldots,g_k$ be independent standard Gaussian vectors in
$\R^{(n+1)h}$, and assume $h\ge\max\{16k,32\log(8(n+1))\}$. Let
\[
    v=p(\Lambda)b_0+\sum_{i=1}^kp_i(\Lambda)g_i,
\]
where $p$ and $p_i$ are polynomials whose coefficients may depend
arbitrarily on the Gaussian vectors and on further randomness, with
$\deg p\le n-2$ and arbitrary $\deg p_i$. Then
$\mathbb E\|\Lambda v+b_0\|\ge\frac{aR}{2n}$.
\end{lemma}

\begin{proof}
Let $S_j\subseteq V_j$ be the span of the projections of
$g_1,\ldots,g_k$ onto $V_j$; these projections are independent standard
Gaussian vectors in $V_j$. Since $\Lambda$ acts on $V_j$ as the scalar
$\lambda_j$, the $V_j$-component of $\Lambda v+b_0$ equals
\[
    R\sqrt{w_j}\,\chi(\lambda_j)e_j+s_j,
    \qquad s_j\in S_j,
    \qquad \chi(\lambda)=\lambda+\lambda^2p(\lambda).
\]
For every realization, $\deg\chi\le n$ and $\chi'(0)=1$. Projecting each
component onto $S_j^\perp\cap V_j$ and using Lemma~\ref{qlow:scalar},
\[
    \|\Lambda v+b_0\|^2
    \ge R^2\sum_{j=0}^nw_j\chi(\lambda_j)^2
        \bigl(1-\|\Pi_{S_j}e_j\|^2\bigr)
    \ge\frac{a^2R^2}{n^2}\Bigl(1-\max_j\|\Pi_{S_j}e_j\|^2\Bigr).
\]
Taking square roots and applying Lemma~\ref{qlow:projection} with
$t=n+1$ proves the claim. The argument is pointwise in the polynomial
coefficients, so their dependence on the Gaussian vectors is irrelevant.
\end{proof}

\paragraph{A hard law of quadratics.}

\begin{lemma}[Quadratic depth bound]
\label{qlow:depth}
Let $m,s\ge1$ be integers and $L',R>0$. There is a distribution of
convex quadratic functions in dimension at most $768ms^2$, each with an
$L'$-Lipschitz gradient, minimum value zero, and a unique minimizer of
norm $R$, such that every randomized algorithm that makes $s$ adaptive
batches of at most $m$ simultaneous queries, each returning $F(x)$ and
$\nabla F(x)$, and then outputs $\widehat x$, satisfies
\begin{equation}
    \mathbb E F(\widehat x)\ge\frac{L'R^2}{288s^2}.
    \label{qlow:depth-rate}
\end{equation}
The distribution does not depend on the algorithm.
\end{lemma}

\begin{proof}
Set $n=2s+3$, $k=2ms+1$, $h=128ms$, and $a=\sqrt{L'}$, and let $\Lambda$,
$b_0$, $z_0$ be as in Lemma~\ref{qlow:residual}; the dimension is
$(n+1)h=(2s+4)\cdot128ms\le768ms^2$. The assumptions of
Lemma~\ref{qlow:residual} hold: $16k\le48ms\le h$, and
$32\log(8(n+1))=32\log(16s+32)\le128s\le h$, because
$\log(16s+32)\le4s$ holds at $s=1$ and the difference of the two sides
increases with $s$.

Let $U$ be Haar distributed, independent of the algorithm, and define
\[
    A=U\Lambda U^\top,\qquad b=Ub_0,\qquad
    F(x)=\frac12\|Ax+b\|^2 .
\]
The Hessian $A^2=U\Lambda^2U^\top$ is positive definite, because no node
vanishes, and has norm $a^2=L'$. The unique minimizer
$-A^{-1}b=-Uz_0$ has norm $\|z_0\|=R$, and the minimum value is zero.

Give $b$ to the algorithm. A batch of $m$ queries of $F$ and $\nabla F$
is answered by two batches of $m$ matrix products, first $Ax$ and then
$A(Ax+b)$, which also give $F(x)=\frac12\|Ax+b\|^2$. Hence the algorithm
is a procedure as in Lemma~\ref{qlow:haar} with $D=2s$ batches of width
$m$. By that lemma, after granting all polynomial powers of the
$2sm+1=k$ Gaussian vectors,
$U^\top\widehat x=p(\Lambda)b_0+\sum_{i=1}^kp_i(\Lambda)g_i$ with
$\deg p\le2s=n-3$. Lemma~\ref{qlow:residual} gives
\[
    \mathbb E\|A\widehat x+b\|
    =\mathbb E\|\Lambda U^\top\widehat x+b_0\|
    \ge\frac{\sqrt{L'}R}{2n}.
\]
By Jensen's inequality,
$\mathbb E F(\widehat x)\ge\frac12(\mathbb E\|A\widehat x+b\|)^2
\ge\frac{L'R^2}{8n^2}$, and $n=2s+3\le6s$ proves
\eqref{qlow:depth-rate}. The law of $F$ is the pushforward of the Haar
law of $U$, which by Lemma~\ref{qlow:haar} is the same for every
algorithm.
\end{proof}

\paragraph{Masking whole sample functions.}

\begin{theorem}[Quadratic communication bound]
\label{qlow:masked}
Let $0<L'\le L$, $p=L'/H_s$, and
\[
    b_p=\mathbb E\max_{w\le m}Z_w,
    \qquad Z_1,\ldots,Z_m\stackrel{\mathrm{iid}}{\sim}
    \operatorname{Bin}(K,p).
\]
For every $d\ge768m(KN)^2$ and every algorithm in $\mathrm{IC}(m,K,N)$
with fresh access, there is an instance of
$\mathcal S_d^{\mathrm{quad}}$ such that
\begin{equation}
    \mathbb E F(\widehat x)\ge\frac{L'R^2}{2304(1+Nb_p)^2}.
    \label{qlow:masked-rate}
\end{equation}
\end{theorem}

\begin{proof}
Let $s=\min\{KN,\max\{1,\lceil2Nb_p\rceil\}\}$ and draw $F$ from the law
of Lemma~\ref{qlow:depth} with parameters $(m,s,L',R)$, padded to
$\R^d$. The sample functions are
\[
    f_\xi=\frac\xi pF,
    \qquad \xi\sim\operatorname{Bernoulli}(p),
\]
with $\xi$ independent of $F$. Each $f_\xi$ is a nonnegative convex
quadratic with an $H_s$-Lipschitz gradient, minimized at the minimizer of
$F$, and $\mathbb E f_\xi=F$. Thus every $F$ in the support defines an
instance of $\mathcal S_d^{\mathrm{quad}}$.

With fresh access, the $k$-th query of worker $w$ in round $n$ uses an
independent variable $\xi_{w,n,k}$. Give the whole array of these
variables to the algorithm before the run; this can only help it. Let
$Z_{w,n}=\sum_{k=1}^K\xi_{w,n,k}$ and
$\Delta=\sum_{n=1}^N\max_{w\le m}Z_{w,n}$.

For a fixed array, the run can be reproduced by an algorithm that makes
$\Delta$ batches of at most $m$ exact queries of $(F,\nabla F)$. Indeed, a
query with $\xi=0$ has the known reply $0$, and a query with $\xi=1$ has
the reply $\nabla F(x)/p$. In round $n$, let the $j$-th batch contain
the $j$-th query with $\xi=1$ of every worker that has at least $j$ such
queries in this round. This query depends only on the seed, on the
messages before round $n$, and on the earlier replies of the same
worker in round $n$, all of which are available after the first $j-1$
batches of the round. Hence round $n$ needs $\max_wZ_{w,n}$ batches, and
the whole run needs $\Delta$ batches.

On the event $\{\Delta\le s\}$, which is determined by the array, pad the
reproducing algorithm to exactly $s$ batches. For every fixed array in
this event it is an $s$-batch algorithm, and the array is independent of
$F$, so Lemma~\ref{qlow:depth} gives
$\mathbb E[F(\widehat x)\mid\text{array}]\ge\frac{L'R^2}{288s^2}$ on
$\{\Delta\le s\}$. We have $\mathbb E \Delta=Nb_p$ and $\Delta\le KN$. If $s=KN$, then
$\Delta\le s$ surely; otherwise $s\ge2Nb_p$, and Markov's inequality gives
$\mathbb{P}\left(\Delta>s\right)\le Nb_p/s\le1/2$. Therefore
$\mathbb E F(\widehat x)\ge\frac{L'R^2}{576s^2}$, and $s\le1+2Nb_p\le
2(1+Nb_p)$ proves \eqref{qlow:masked-rate} for the random instance;
averaging selects a fixed one. The dimension of the law is at most
$768ms^2\le768m(KN)^2$.
\end{proof}

\begin{theorem}[Lower bound for quadratic populations]
\label{qlow:allbudget}
For every $d\ge2048m(KN)^2$,
\[
    \mathfrak M_d^{\mathrm{sg,quad}}
    \ge\frac{R^2}{73728\log^2(em)}\min\left\{L,\;
        \frac{\min\{L,H_s/K\}}{N^2}+\frac{H_s}{mKN}\right\}.
\]
\end{theorem}

\begin{proof}
We first show that
\begin{equation}
    \mathfrak M_d^{\mathrm{sg,quad}}
    \ge\frac{R^2\min\{L,H_s/K\}}{36864\,N^2\log^2(em)} .
    \label{qlow:comm}
\end{equation}
If $\log(em)\le K$, let $p=\min\{L/H_s,\log(em)/K\}$ and $L'=pH_s\le L$.
By Lemma~\ref{lem:max-binomial} and $e-1\le2$,
$b_p\le\log m+2Kp\le3\log(em)$, so $1+Nb_p\le4N\log(em)$.
Theorem~\ref{qlow:masked} gives the lower bound
$\frac{L'R^2}{36864N^2\log^2(em)}$, and
$L'=\min\{L,H_s\log(em)/K\}\ge\min\{L,H_s/K\}$.

If $\log(em)>K$, take deterministic samples $f_\xi=F$, where $F$ is drawn
from the law of Lemma~\ref{qlow:depth} with $L'=L$ and $s=KN$; these are
admissible because $L\le H_s$. An IC algorithm with exact gradients is
reproduced by $KN$ batches of width $m$, the $(n,k)$-th batch consisting
of the $k$-th queries of all workers in round $n$. Lemma~\ref{qlow:depth}
gives the lower bound $\frac{LR^2}{288K^2N^2}\ge
\frac{LR^2}{288N^2\log^2(em)}$, which implies \eqref{qlow:comm}.

By Lemma~\ref{lem:stat},
$\mathfrak M_d^{\mathrm{sg,quad}}\ge\frac{R^2}8\min\{L,\frac{H_s}{mKN}\}$.
The maximum of two lower bounds is at least their average, and
$\frac1{16}\ge\frac1{73728\log^2(em)}$, so
\begin{align*}
    \mathfrak M_d^{\mathrm{sg,quad}}
    &\ge\frac{R^2}{73728\log^2(em)}
    \left[\frac{\min\{L,H_s/K\}}{N^2}
          +\min\left\{L,\frac{H_s}{mKN}\right\}\right]\\
    &\ge\frac{R^2}{73728\log^2(em)}\min\left\{L,\;
        \frac{\min\{L,H_s/K\}}{N^2}+\frac{H_s}{mKN}\right\},
\end{align*}
where the last inequality holds because the bracket equals
$\frac{\min\{L,H_s/K\}}{N^2}+\frac{H_s}{mKN}$ if $\frac{H_s}{mKN}\le L$
and is at least $L$ otherwise. All constructions have dimension at most $768m(KN)^2$, so by
padding the bound holds for every $d\ge768m(KN)^2$, in particular for
every $d\ge2048m(KN)^2$.
\end{proof}

\subsection{Weak growth condition}
\label{app:lower-wg}

The hard oracle reveals the next direction of a hidden chain only with
probability $p$ per query, and corrects the remaining bias with a rare
Bernoulli variable. Throughout this subsection,
\begin{equation}
\begin{gathered}
    \ell=\frac L{1024},
    \qquad
    q=\left\lceil\max\left\{8N\log(em),\;
        64\sqrt{\frac{\ell KN}{H_s}}\right\}\right\rceil,
    \qquad
    \delta=R\Bigl(\sum_{j=1}^qj^2\Bigr)^{-1/2},\\
    p=\frac{128\ell}{H_sq},
    \qquad
    \varepsilon=\frac1{16mKN},
    \qquad
    \alpha=\min\left\{\frac\delta8,\;
        \delta\sqrt{\frac{H_s\varepsilon}{800\ell}}\right\},
\end{gathered}
\label{eq:wg-parameters}
\end{equation}
and
\begin{equation}
    d_{\mathrm{wg}}
    =q+2+\left\lceil\frac{2(130R)^2}{\alpha^2}
        \log\bigl(16q(mKN+1)\bigr)\right\rceil .
    \label{eq:d-wg}
\end{equation}
Since $R/\alpha=\bigl(\sum_{j\le q}j^2\bigr)^{1/2}
\max\{8,\sqrt{800\ell/(H_s\varepsilon)}\}$, the threshold
$d_{\mathrm{wg}}$ depends only on $m$, $K$, $N$, and $H_s/L$. Note that
$q\ge8$, that $p\le\frac1{8q}$ because $1024\ell=L\le H_s$, and that
$\sum_{j\le q}j^2\in[q,q^3]$ gives
\begin{equation}
    \delta\sqrt q\le R,
    \qquad
    \delta^2q\ge\frac{R^2}{q^2}.
    \label{eq:wg-delta}
\end{equation}

\paragraph{The chain.}
Let $h_\delta(u)=u^2/2$ for $|u|\le\delta$ and
$h_\delta(u)=\delta|u|-\delta^2/2$ otherwise; its derivative $h_\delta'$
is the clipping onto $[-\delta,\delta]$. For $y\in\R^q$ put
$y_{q+1}=0$ and
\begin{equation}
    f(y)=\ell\sum_{i=1}^q h_\delta(y_i-y_{i+1}-\delta),
    \qquad
    \operatorname{prog}(y)=\max\bigl(\{i\le q:|y_i|>\alpha\}\cup\{0\}\bigr).
    \label{eq:wg-chain}
\end{equation}
For $0\le t\le q$, let $P_t$ keep the first $t$ coordinates of a vector
in $\R^q$ and set the others to zero, and let $e_1,\ldots,e_q$ be the
standard basis.

\begin{lemma}[The chain]
\label{wg:chain-lemma}
The function $f$ is nonnegative and convex, has a $4\ell$-Lipschitz
gradient with $\|\nabla f\|\le2\ell R$, and vanishes at
$r_0=\delta(q,q-1,\ldots,1)$, where $\|r_0\|=R$. If
$t=\operatorname{prog}(y)\le q/2$, then
\begin{equation}
    f(y)\ge\frac{\ell\delta^2q}{16},
    \label{wg:tail-gap}
\end{equation}
and $v=P_ty$ satisfies
\begin{equation}
\begin{gathered}
    \nabla f(v)=P_t\nabla f(v)+[\nabla f(v)]_{t+1}e_{t+1},
    \qquad
    \bigl|[\nabla f(v)]_{t+1}\bigr|\le2\ell\delta,\\
    \|\nabla f(y)-\nabla f(v)\|\le4\ell\sqrt q\,\alpha,
    \qquad
    |f(y)-f(v)|\le2\ell\delta q\alpha,
\end{gathered}
    \label{wg:correction-bounds}
\end{equation}
Moreover, $f(v)$ and $\nabla f(v)$ depend only on $y_1,\ldots,y_t$.
\end{lemma}

\begin{proof}
Let $D$ be the matrix with $(Dy)_i=y_i-y_{i+1}$. Since
$(y_i-y_{i+1})^2\le2y_i^2+2y_{i+1}^2$, we have $\|D\|\le2$. Then
$\nabla f(y)=\ell D^\top h_\delta'(Dy-\delta\mathbf 1)$, where
$h_\delta'$ acts coordinatewise. As $h_\delta$ is convex with a
$1$-Lipschitz derivative bounded by $\delta$, the function $f$ is convex
with a $4\ell$-Lipschitz gradient, and
$\|\nabla f\|\le2\ell\delta\sqrt q\le2\ell R$ by \eqref{eq:wg-delta}.
Since $Dr_0=\delta\mathbf 1$, $f(r_0)=0$, and
$\|r_0\|^2=\delta^2\sum_jj^2=R^2$.

Let $t=\operatorname{prog}(y)\le q/2$. For $i>t$ we have $|y_i|\le\alpha$
and $|y_{i+1}|\le\alpha$, hence
$|y_i-y_{i+1}-\delta|\ge\delta-2\alpha\ge\delta/2$ and
$h_\delta(y_i-y_{i+1}-\delta)\ge h_\delta(\delta/2)=\delta^2/8$. There are
$q-t\ge q/2$ such indices, which gives \eqref{wg:tail-gap}.

Let $w=h_\delta'(Dv-\delta\mathbf 1)$. Since $v_i=0$ for $i>t$, we have
$w_i=-\delta$ for $i>t$. With $w_0=0$,
$[\nabla f(v)]_j=\ell(w_j-w_{j-1})$, which vanishes for $j\ge t+2$ and
has absolute value at most $2\ell\delta$ for $j=t+1$. This proves the
identity and the first inequality in \eqref{wg:correction-bounds}. The
coordinates of $y-v$ are the $y_i$ with $i>t$, so
$\|y-v\|\le\sqrt q\,\alpha$; the last two inequalities follow from the
Lipschitz property of $\nabla f$ and the bound on $\|\nabla f\|$.
Finally, $v$ is determined by $y_1,\ldots,y_t$.
\end{proof}

\paragraph{The instance.}
Let $U=(u_1,\ldots,u_q)\in\R^{d\times q}$ have orthonormal columns, and
put $F_U(x)=f(U^\top x)$. By Lemma~\ref{wg:chain-lemma}, $F_U$ satisfies
the hypotheses of Lemma~\ref{lem:radial} with $x^*=Ur_0$,
$L_\varphi=4\ell$, and $M_\varphi=2\ell R$. Apply that lemma with $c=\ell$
and $\rho=128R$, and let $\Gamma$ be the corresponding cutoff and
\[
    \Phi_U=\mathcal T F_U .
\]
Then $\Phi_U$ is convex with a $9\ell$-Lipschitz gradient, $9\ell\le L$,
$\Phi_U\ge F_U$, $\Phi_U(Ur_0)=0$, and for $\|x\|\ge130R$,
$\Phi_U(x)=2\ell(\|x\|-R)^2$ with a gradient that does not depend on $U$.

Each oracle call draws independent $B\sim\operatorname{Bernoulli}(p)$ and
$C\sim\operatorname{Bernoulli}(\varepsilon)$. For $y\in\R^q$ with
$t=\operatorname{prog}(y)\le q/2$ and $v=P_ty$, put
\begin{align*}
    \widehat{\nabla f}(y;B,C)
    &=P_t\nabla f(v)+\frac Bp\,[\nabla f(v)]_{t+1}e_{t+1}
      +\frac C\varepsilon\bigl(\nabla f(y)-\nabla f(v)\bigr),\\
    \widehat f(y;C)
    &=f(v)+\frac C\varepsilon\bigl(f(y)-f(v)\bigr).
\end{align*}
The oracle $G_U(x;B,C)$ is defined as follows. If $\|x\|\ge130R$, it
returns $4\ell(\|x\|-R)\frac x{\|x\|}$. If $\|x\|<130R$ and
$\operatorname{prog}(U^\top x)>q/2$, it returns $\nabla\Phi_U(x)$.
Otherwise, it returns
\begin{equation}
    G_U(x;B,C)
    =(1-\Gamma(x))\,U\,\widehat{\nabla f}(U^\top x;B,C)
     -\widehat f(U^\top x;C)\,\nabla\Gamma(x)
     +4\ell(\|x\|-R)_+\frac x{\|x\|},
    \label{wg:oracle}
\end{equation}
with the radial term equal to zero at $x=0$. The oracle is jointly
measurable.

\begin{lemma}[Admissibility]
\label{wg:moment-lemma}
For every $U$, the pair $(\Phi_U,G_U)$ is an instance of $\mathcal W_d$:
for every $x\in\R^d$,
\[
    \mathbb E\,G_U(x;B,C)=\nabla\Phi_U(x),
    \qquad
    \mathbb E\|G_U(x;B,C)\|^2\le2H_s\Phi_U(x).
\]
\end{lemma}

\begin{proof}
The function $\Phi_U$ is convex, has an $L$-Lipschitz gradient, and is
minimized at $Ur_0$, which has norm $R$, with $\Phi_U(Ur_0)=0$. By
Lemma~\ref{lem:smooth}, Part~2,
$\|\nabla\Phi_U\|^2\le18\ell\,\Phi_U\le H_s\Phi_U$. In the first two
cases of the definition of $G_U$ the reply equals $\nabla\Phi_U(x)$, by
Lemma~\ref{lem:radial}, Part~4, and the claim follows.

In the remaining case, where $G_U$ is given by \eqref{wg:oracle}, let $y=U^\top x$, $t=\operatorname{prog}(y)\le q/2$, and
$v=P_ty$. By the identity in \eqref{wg:correction-bounds},
$\mathbb E\widehat{\nabla f}(y;B,C)=\nabla f(y)$, and
$\mathbb E\widehat f(y;C)=f(y)$. Since $\nabla F_U(x)=U\nabla f(y)$,
Lemma~\ref{lem:radial}, Part~3, gives
$\mathbb E G_U(x;B,C)=\nabla\Phi_U(x)$. The centered reply equals
\[
    \Bigl(\frac Bp-1\Bigr)(1-\Gamma)[\nabla f(v)]_{t+1}Ue_{t+1}
    +\Bigl(\frac C\varepsilon-1\Bigr)
     \Bigl[(1-\Gamma)U\bigl(\nabla f(y)-\nabla f(v)\bigr)
          -\bigl(f(y)-f(v)\bigr)\nabla\Gamma\Bigr].
\]
The two scalar factors are independent and centered, with variances at
most $1/p$ and $1/\varepsilon$. By \eqref{wg:correction-bounds}, the
bound $\|\nabla\Gamma\|\le\frac3{2\rho}=\frac3{256R}$ of
Lemma~\ref{lem:radial}, Part~3, and \eqref{eq:wg-delta}, the norm of the
bracket is at most
$4\ell\sqrt q\alpha+2\ell\delta q\alpha\frac3{256R}\le5\ell\sqrt q\alpha$.
Using the definitions of $p$ and $\alpha$ and then \eqref{wg:tail-gap},
\[
    \mathbb E\|G_U(x;B,C)-\nabla\Phi_U(x)\|^2
    \le\frac{4\ell^2\delta^2}{p}+\frac{25\ell^2q\alpha^2}{\varepsilon}
    \le\frac{H_s\ell\delta^2q}{32}+\frac{H_s\ell\delta^2q}{32}
    \le H_sF_U(x)\le H_s\Phi_U(x).
\]
Adding $\|\nabla\Phi_U(x)\|^2\le H_s\Phi_U(x)$ proves the second claim.
\end{proof}

\begin{remark}[Information revealed by one reply]
\label{wg:information}
Let $\|x\|<130R$, fix an integer $t\le q/2$, and evaluate the right-hand
side of \eqref{wg:oracle} with $C=0$ and with this $t$ in place of
$\operatorname{prog}(U^\top x)$ in the definition of $\widehat{\nabla f}$
and $\widehat f$, where $v=P_tU^\top x$. By the last assertion of
Lemma~\ref{wg:chain-lemma}, the numbers $f(v)$ and $[\nabla f(v)]_j$ are
functions of $\langle u_1,x\rangle,\ldots,\langle u_t,x\rangle$, and
$UP_t\nabla f(v)=\sum_{j\le t}[\nabla f(v)]_ju_j$, $Ue_{t+1}=u_{t+1}$.
Hence the result is a function of $x$, $B$, and $u_1,\ldots,u_t$, and,
if $B=1$, also of $u_{t+1}$. For $t=\operatorname{prog}(U^\top x)$ it
equals $G_U(x;B,0)$.
\end{remark}

\paragraph{The lower bound.}

\begin{theorem}[Lower bound under weak growth]
\label{wg:lower-theorem}
For every $d\ge d_{\mathrm{wg}}$,
\[
    \mathfrak M_d^{\mathrm{wg}}
    \ge\frac{R^2}{2^{24}\log^2(em)}\min\left\{L,\;\frac{H_s}{KN},\;
        \frac{L}{N^2}+\frac{H_s}{mKN}\right\}.
\]
\end{theorem}

\begin{proof}
Fix an algorithm. Let $U$ be uniformly distributed on the set of
orthonormal $q$-frames in $\R^d$, independently of the seed. Index the
oracle calls by slots $(w,n,k)$ (worker, round, call) and predraw
independent pairs $(B_{w,n,k},C_{w,n,k})$ for all $mKN$ slots; a worker
that makes fewer than $K$ queries in a round leaves the remaining slots
unused.

\emph{Step 1: a virtual run.}
Each worker carries a prefix length $s_w$, initially zero. After the
slot $(w,n,k)$, whether or not it is used, $s_w$ increases by
$B_{w,n,k}$; at the end of every round all $s_w$ are replaced by their
maximum. The prefix lengths are thus functions of the array of $B$'s
only. The virtual run applies the maps of the algorithm to the following
virtual replies. At a query $x$ in a slot where the current prefix
length is $s$:
\begin{itemize}
    \item if $s>q/2$, the reply is $0$;
    \item if $s\le q/2$ and $\|x\|\ge130R$, the reply is
    $4\ell(\|x\|-R)x/\|x\|$;
    \item if $s\le q/2$ and $\|x\|<130R$, the reply is given by
    \eqref{wg:oracle} with $C=0$, the $B$ of the slot, and
    $t=\operatorname{prog}(U^\top x)$ in the definition of
    $\widehat{\nabla f}$ and $\widehat f$ replaced by
    $\tilde t=\max(\{j\le s:|\langle u_j,x\rangle|>\alpha\}\cup\{0\})$.
\end{itemize}
By Remark~\ref{wg:information}, a virtual reply at a slot with prefix
length $s\le q/2$ uses only $u_1,\ldots,u_s$, and also $u_{s+1}$ if
$B=1$, in which case the prefix length after the slot is $s+1$.

\emph{Step 2: the depth of the virtual run.}
Let $Z_{w,n}=\sum_kB_{w,n,k}$ and $\Delta=\sum_{n=1}^N\max_wZ_{w,n}$. Every
prefix length at any time is at most $\Delta$. The variables $Z_{w,n}$ are
independent $\operatorname{Bin}(K,p)$, so Lemma~\ref{lem:max-binomial},
independence over rounds, and Markov's inequality give
\begin{equation}
    \mathbb{P}\left(\Delta>q/2\right)
    \le e^{-q/2}\,\mathbb E e^\Delta
    \le\exp\Bigl(-\frac q2+N\log m+2KNp\Bigr)
    \le e^{-5q/16}<\frac18 ,
    \label{wg:depth-tail}
\end{equation}
because $N\log m\le q/8$ and $2KNp=\frac{256\ell KN}{H_sq}\le\frac q{16}$
by \eqref{eq:wg-parameters}, and $q\ge8$.

\emph{Step 3: hidden projections.}
Fix the seed and the arrays of $B$'s and $C$'s. Consider a slot where the
prefix length is $s\le q/2$. The virtual query $\tilde x$ at this slot is
a function of $u_1,\ldots,u_s$: it depends on the messages of previous
rounds and on the earlier virtual replies of the same worker in the
current round, and each of these replies was given at a slot whose
prefix length after the slot is at most $s$, so by Step~1 it uses only
$u_1,\ldots,u_s$. In the same way, the virtual output $\tilde x_{\rm out}$
is a function of $u_1,\ldots,u_\Delta$ when $\Delta\le q/2$. Conditional on the
seed, the arrays, and $u_1,\ldots,u_s$, each $u_j$ with $j>s$ is uniform
on the unit sphere of the $(d-s)$-dimensional space
$\operatorname{span}\{u_1,\ldots,u_s\}^\perp$. Lemma~\ref{lem:sphere},
applied to $\tilde x\mathbf 1\{\|\tilde x\|<130R\}$, gives
\[
    \mathbb{P}\Bigl(|\langle u_j,\tilde x\rangle|>\alpha,\ \|\tilde x\|<130R
       \,\Big|\,\text{seed, arrays},u_1,\ldots,u_s\Bigr)
    \le2\exp\left(-\frac{(d-q-1)\alpha^2}{2(130R)^2}\right),
\]
and the same bound holds for $\tilde x_{\rm out}$ with $s=\Delta\le q/2$. Let
$E_{\rm proj}$ be the event that $|\langle u_j,\tilde x\rangle|\le\alpha$
for every virtual query $\tilde x$ with $\|\tilde x\|<130R$ at a slot
with prefix length $s\le q/2$ and every $j>s$, and the same for
$\tilde x_{\rm out}$ if $\Delta\le q/2$ and $\|\tilde x_{\rm out}\|<130R$, with
$j>\Delta$. A union bound over at most $mKN+1$ vectors and $q$ indices and
$d\ge d_{\mathrm{wg}}$ give
\begin{equation}
    \mathbb{P}(E_{\rm proj}^c)
    \le2q(mKN+1)\exp\left(-\frac{(d-q-1)\alpha^2}{2(130R)^2}\right)
    \le\frac18 .
    \label{wg:projection-union}
\end{equation}

\emph{Step 4: coupling.}
Let $E$ be the intersection of $\{C_{w,n,k}=0\text{ for all slots}\}$,
$\{\Delta\le q/2\}$, and $E_{\rm proj}$. Since $mKN\varepsilon=\frac1{16}$, the
bounds \eqref{wg:depth-tail} and \eqref{wg:projection-union} give
$\mathbb{P}(E)\ge1-\frac1{16}-\frac18-\frac18\ge\frac12$. On $E$ the actual run
coincides with the virtual run. Indeed, suppose all earlier replies
coincide, so that the current queries coincide, and let $x$ be the query
at a slot with prefix length $s\le \Delta\le q/2$. If $\|x\|\ge130R$, both
replies equal the public vector. Otherwise, $|\langle u_j,x\rangle|\le
\alpha$ for all $j>s$ by $E_{\rm proj}$, so
$\operatorname{prog}(U^\top x)=\tilde t\le q/2$, and both replies are
given by \eqref{wg:oracle} with the same $t$, the same
$B$, and $C=0$. By induction over the slots, the two runs and their
outputs coincide.

\emph{Step 5: the risk.}
On $E$, let $\widehat x$ be the common output. If $\|\widehat x\|<130R$,
then $\operatorname{prog}(U^\top\widehat x)\le \Delta\le q/2$ by $E_{\rm proj}$,
and \eqref{wg:tail-gap} gives
$\Phi_U(\widehat x)\ge F_U(\widehat x)\ge\frac{\ell\delta^2q}{16}$. If
$\|\widehat x\|\ge130R$, then
$\Phi_U(\widehat x)=2\ell(\|\widehat x\|-R)^2\ge\ell R^2
\ge\frac{\ell\delta^2q}{16}$. Since $\Phi_U\ge0$, \eqref{eq:wg-delta}
gives
\begin{equation}
    \mathbb E\,\Phi_U(\widehat x)\ge\frac{\ell\delta^2q}{32}
    \ge\frac{\ell R^2}{32q^2},
    \label{wg:chain-risk}
\end{equation}
where the expectation is over $U$, the seed, and the oracle variables.
Averaging over $U$ selects a fixed frame, that is, a fixed instance of
$\mathcal W_d$ by Lemma~\ref{wg:moment-lemma}.

\emph{Step 6: the rate.}
The quantity inside the ceiling defining $q$ is at least one, so
$q^2\le4\max\{64N^2\log^2(em),4096\ell KN/H_s\}$. Substituting this and
$\ell=L/1024$ into \eqref{wg:chain-risk},
\begin{equation}
    \mathfrak M_d^{\mathrm{wg}}
    \ge R^2\min\left\{\frac{L}{2^{23}N^2\log^2(em)},
                      \frac{H_s}{2^{19}KN}\right\}
    \ge\frac{R^2}{2^{23}\log^2(em)}
     \min\left\{\frac{L}{N^2},\frac{H_s}{KN}\right\}.
    \label{wg:communication-lower}
\end{equation}
Lemma~\ref{lem:stat} gives
$\mathfrak M_d^{\mathrm{wg}}\ge\frac{R^2}8\min\{L,\frac{H_s}{mKN}\}$, and
$\frac18\ge\frac{2^{-23}}{\log^2(em)}$. It remains to note that
\begin{align*}
    \min\left\{L,\frac{H_s}{KN},\frac L{N^2}+\frac{H_s}{mKN}\right\}
    &\le\min\left\{\frac{H_s}{KN},\frac L{N^2}\right\}
      +\min\left\{L,\frac{H_s}{mKN}\right\}\\
    &\le2\max\left\{\min\left\{\frac{H_s}{KN},\frac L{N^2}\right\},\;
      \min\left\{L,\frac{H_s}{mKN}\right\}\right\}.
\end{align*}
Indeed, if $\frac{H_s}{KN}\le\frac L{N^2}$, the left-hand side is at most
$\frac{H_s}{KN}$; otherwise it is at most
$\min\{L,\frac L{N^2}+\frac{H_s}{mKN}\}\le\frac L{N^2}
+\min\{L,\frac{H_s}{mKN}\}$. Hence the larger of the two lower bounds is
at least $\frac{R^2}{2^{24}\log^2(em)}
\min\{L,\frac{H_s}{KN},\frac L{N^2}+\frac{H_s}{mKN}\}$.
\end{proof}

\subsection{Sample reuse}
\label{app:lower-pool}

In this construction every sample function carries its own independent
Bernoulli coins on the edges of a hidden chain. Repeated queries of the
same label therefore reveal nothing beyond these coins. Throughout this
subsection,
\begin{equation}
\begin{gathered}
    p=\min\left\{\frac L{H_s},\frac{\log(em)}K\right\},
    \qquad
    \nu=N\bigl[\log m+(e-1)Kp\bigr],
    \qquad
    s=1+\min\{KN,\lceil4\nu\rceil\},\\
    r=2(s+1),
    \qquad
    a=\frac R{\sqrt r},
    \qquad
    \alpha=\frac a{4(s+1)},
    \qquad
    \rho=\frac{32H_sR}{L},
\end{gathered}
\label{eq:pool-parameters}
\end{equation}
and
\begin{equation}
    d_{\mathrm{pool}}
    =r+2+\left\lceil\frac{8(2R+\rho)^2}{\alpha^2}
        \log\bigl(8(mKN+1)r\bigr)\right\rceil .
    \label{eq:d-pool}
\end{equation}
Since $(2R+\rho)/\alpha=4(s+1)\sqrt r\,(2+32H_s/L)$, the threshold
$d_{\mathrm{pool}}$ depends only on $m$, $K$, $N$, and $H_s/L$.

\paragraph{The chain.}
For $t\ge0$ let $\eta_a(t)=t^2$ if $t\le a$ and $\eta_a(t)=2at-a^2$ if
$t>a$, and for $\beta\ge0$ let $\psi_\beta(t)=\eta_a\bigl((|t|-\beta)_+\bigr)$.
The function $\psi_\beta$ is convex and continuously differentiable,
$\psi_\beta'$ is $2$-Lipschitz, $|\psi_\beta'|\le2a$, and $\psi_\beta$ and
$\psi_\beta'$ vanish on $[-\beta,\beta]$. Let $V=(v_1,\ldots,v_r)$ be an
orthonormal frame in $\R^d$. A sample is a vector of coins
$\xi=(b_1,\ldots,b_{r-1})$ with independent
$\operatorname{Bernoulli}(p)$ coordinates, and
\begin{equation}
    f_{V,\xi}(x)=\frac{H_s}{256}\left[
        p\,\psi_0\bigl(\langle v_1,x\rangle-a\bigr)
        +\sum_{l=1}^{r-1}b_l\,\psi_\alpha\bigl(\langle v_l-v_{l+1},x\rangle\bigr)
    \right].
    \label{eq:pool-component}
\end{equation}
Let $f_V$ denote the same expression with $p$ and all $b_l$ replaced by
one, so that $\mathbb E_\xi f_{V,\xi}=pf_V$.

\begin{lemma}[The chain with coins]
\label{lo:chain-properties}
Every function \eqref{eq:pool-component} is nonnegative and convex, has an
$\frac{H_s}{32}$-Lipschitz gradient with
$\|\nabla f_{V,\xi}\|\le\frac{H_sR}{64}$, and vanishes at
$x^*=a\sum_{l=1}^rv_l$, where $\|x^*\|=R$.
\end{lemma}

\begin{proof}
Let $W:\R^d\to\R^r$ map $x$ to
$(\langle v_1,x\rangle,\langle v_1-v_2,x\rangle,\ldots,
\langle v_{r-1}-v_r,x\rangle)$. In the coordinates of the frame,
$w_1^2+\sum_{l<r}(w_l-w_{l+1})^2\le4\sum_lw_l^2$, so $\|W\|\le2$. The
gradient of \eqref{eq:pool-component} is $\frac{H_s}{256}W^\top\Psi(x)$,
where $\Psi(x)\in\R^r$ collects the derivatives of the scalar terms
multiplied by their coefficients in $[0,1]$. Each entry of $\Psi$ is
$2$-Lipschitz as a function of the corresponding entry of $Wx$ and is
bounded by $2a$, so
$\|\nabla f_{V,\xi}(x)-\nabla f_{V,\xi}(y)\|\le\frac{H_s}{256}\cdot2\cdot2
\cdot2\|x-y\|=\frac{H_s}{32}\|x-y\|$ and
$\|\nabla f_{V,\xi}\|\le\frac{H_s}{256}\cdot2\cdot2a\sqrt r
=\frac{H_sR}{64}$. Convexity and nonnegativity follow from the scalar
functions. Finally, $\langle v_1,x^*\rangle=a$ and
$\langle v_l-v_{l+1},x^*\rangle=0$, and $\|x^*\|=a\sqrt r=R$.
\end{proof}

\paragraph{The instance.}
Apply Lemma~\ref{lem:radial} with $c=L/32$, $L_\varphi=H_s/32$,
$M_\varphi=H_sR/64$, and $\rho$ from \eqref{eq:pool-parameters}; indeed
$64M_\varphi/c=\rho\ge32R$. The sample functions are
$\widetilde f_{V,\xi}=\mathcal Tf_{V,\xi}$. They are nonnegative, convex,
minimized at $x^*$, and have $\frac{6H_s}{32}$-Lipschitz gradients. Since
$\mathcal T$ is affine, the population is
$F_V=\mathbb E_\xi\widetilde f_{V,\xi}=\mathcal T(pf_V)$, and
Lemma~\ref{lem:radial} applied to $pf_V$ (with $L_\varphi=pH_s/32\le L/32$)
shows that $F_V$ has an $L$-Lipschitz gradient and
\begin{equation}
    F_V\ge pf_V .
    \label{lo:population-domination}
\end{equation}
Thus, for every frame, $(\widetilde f_{V,\xi})$ is an instance of
$\mathcal S_d$. For $\|x\|\ge2R+\rho$, all sample values and gradients
equal those of $\frac L{16}(\|x\|-R)^2$.

\paragraph{A simulated run.}
Fix an algorithm, and let $V$ be uniformly distributed on the orthonormal
$r$-frames, independently of the pool and of the seed. We strengthen the
algorithm in three ways, each of which can only reduce its risk: every
reply also contains the sample value; complete local histories are
exchanged at the end of every round; and the oracle makes the following
disclosures. Each worker $w$ has a prefix length $j_w$, meaning that it
knows $v_1,\ldots,v_{j_w}$; initially $j_w=1$ for all $w$, and at the end
of every round all disclosures are shared and every $j_w$ is replaced by
$\max_wj_w$. At a query $(i,x)$ of worker $w$ with a valid label and
pre-query prefix length $j<r$, the oracle discloses to $w$ the coins
$b_l$ of $\xi_i$ with $l<j$ and inspects $b_j$; if $b_j=1$, it discloses
$v_{j+1}$ and sets $j_w=j+1$. A query with a valid label and $j=r$
discloses all coins of $\xi_i$, and a query with an invalid label
discloses nothing.

The \emph{true} strengthened oracle returns the value and gradient of
$\widetilde f_{V,\xi_i}$ at $x$. The \emph{simulated} oracle makes the
same disclosures and returns: the public radial reply if
$\|x\|\ge2R+\rho$; the true reply if $j=r$; and otherwise the value and
gradient at $x$ of $\mathcal Tf^{[j]}_{V,\xi_i}$, where
\[
    f^{[j]}_{V,\xi}(x)=\frac{H_s}{256}\left[
        p\,\psi_0\bigl(\langle v_1,x\rangle-a\bigr)
        +\sum_{l=1}^{j}b_l\,\psi_\alpha\bigl(\langle v_l-v_{l+1},x\rangle\bigr)
    \right],
\]
the term $l=j$ being omitted when $b_j=0$. A simulated reply uses only
the frame vectors and coins disclosed to the worker after the query. If
$\|x\|\ge2R+\rho$, or if
\begin{equation}
    |\langle v_l,x\rangle|\le\frac\alpha2\quad\text{for every }l>j,
    \label{lo:local-geometric-condition}
\end{equation}
the simulated and true replies coincide. Indeed, under
\eqref{lo:local-geometric-condition},
$|\langle v_l-v_{l+1},x\rangle|\le\alpha$ for $l>j$, so the terms with
$l>j$ in \eqref{eq:pool-component} vanish together with their gradients
at $x$; hence $f_{V,\xi}$ and $f^{[j]}_{V,\xi}$ have the same value and
gradient at $x$, and by Lemma~\ref{lem:radial}, Part~3, so do their
images under $\mathcal T$. The \emph{simulated run} is the run of the
strengthened algorithm with the simulated oracle; it is defined for every
frame, pool, and seed.

\begin{lemma}[Progress of the simulated run]
\label{lo:progress}
Let $Z$ be the total increase of the common prefix length over the $N$
rounds of the simulated run. Then $\mathbb E Z\le\nu$ and $Z\le KN$.
\end{lemma}

\begin{proof}
\emph{Known frontier coins are zero.} Let $j$ be the current prefix
length of worker $w$, and suppose a coin $b_j$ of some sample has been
disclosed to $w$. It was disclosed either at a query with pre-query
prefix length $j'>j$, which is impossible because prefix lengths do not
decrease and the prefix length of $w$ would exceed $j$, or as the
inspected coin at a query with pre-query prefix length $j$; in the
latter case, if it were one, that prefix length would have become $j+1$,
and the prefix length of $w$ would exceed $j$ as well. Coins disclosed to
other workers in the current round are not disclosed to $w$. Hence
$b_j=0$.

\emph{One round.} Fix a round, let $\mathcal G$ be the shared information
at its start, and let $\mathcal G_{w,k}$ be generated by $\mathcal G$ and
the first $k$ queries of worker $w$ in the round, with their replies and
disclosures. Let $I_{w,k}\in\{0,1\}$ indicate that the $k$-th query of $w$
advances its prefix ($I_{w,k}=0$ if the query is not made). The
information in $\mathcal G_{w,k-1}$ is a function of the seed, of $V$, and
of the coins disclosed to $w$, and the label $i$ and the frontier $j$ of
the $k$-th query are $\mathcal G_{w,k-1}$-measurable. If the coin $b_j$
of $\xi_i$ has been disclosed to $w$, it is zero and $I_{w,k}=0$.
Otherwise, since the coins are independent of each other, of $V$, and of
the seed, and the coins disclosed to $w$ were chosen sequentially from
previously disclosed information, $b_j$ is
$\operatorname{Bernoulli}(p)$ conditional on $\mathcal G_{w,k-1}$. In both
cases
$\mathbb E[e^{I_{w,k}}\mid\mathcal G_{w,k-1}]\le1-p+pe$. All workers start
the round with the same prefix length, so the increase of the common
prefix length in the round is $\max_w\sum_kI_{w,k}$, and
Lemma~\ref{lem:max-binomial} bounds its conditional expectation given
$\mathcal G$ by $\log m+(e-1)Kp$. Summing over the rounds gives
$\mathbb E Z\le\nu$. Each round increases the common prefix length by at
most $K$, so $Z\le KN$.
\end{proof}

\begin{lemma}[Hidden projections]
\label{lo:geometry}
Let $d\ge d_{\mathrm{pool}}$. With probability at least $3/4$, every
query $x$ of the simulated run with $\|x\|<2R+\rho$ satisfies
\eqref{lo:local-geometric-condition} with its pre-query prefix length
$j$, and its output $\widehat x$, if $\|\widehat x\|<2R+\rho$, satisfies
$|\langle v_l,\widehat x\rangle|\le\alpha/2$ for every $l$ larger than the
final common prefix length. On this event, the simulated run and the run
with the true strengthened oracle coincide.
\end{lemma}

\begin{proof}
Fix the seed and the pool. Consider a query slot, let $\bar j$ be the
pre-query prefix length of the querying worker in the simulated run, and
let $j\le r$. We claim that the event $\{\bar j\le j\}$, and on this event the
query, are determined by $v_1,\ldots,v_j$. To see this, reconstruct the
simulated run from $v_1,\ldots,v_j$ round by round, simulating the
workers of a round separately, which is possible because they do not
interact within a round. If in a round before the round of the slot some
worker would need $v_{j+1}$, the common prefix length at the start of the
round of the slot exceeds $j$, so $\bar j>j$. If in the round of the slot the
querying worker would need $v_{j+1}$ before the slot, again $\bar j>j$.
Otherwise all replies that the query depends on are computable from
$v_1,\ldots,v_j$, so is the query, and $\bar j\le j$. The same argument
applies to the output and the final common prefix length.

For a uniformly distributed frame, conditional on $v_1,\ldots,v_j$ every
$v_l$ with $l>j$ is uniformly distributed on the unit sphere of the
$(d-j)$-dimensional space $\operatorname{span}\{v_1,\ldots,v_j\}^\perp$.
Since the event $\{\bar j=j\}$ and the query on it are determined by
$v_1,\ldots,v_j$, Lemma~\ref{lem:sphere}, applied to the query capped to
$\{\|x\|<2R+\rho\}$ with the threshold $\alpha/2$, gives for every $l>\bar j$
\[
    \mathbb{P}\Bigl(|\langle v_l,x\rangle|>\frac\alpha2,\
        \|x\|<2R+\rho\Bigr)
    \le2\exp\left(-\frac{(d-r-1)\alpha^2}{8(2R+\rho)^2}\right).
\]
The same bound holds for the output. There are at most $mKN$ queries,
one output, and at most $r$ indices $l$ for each, so by a union bound and
\eqref{eq:d-pool} the failure probability is at most
$2(mKN+1)r\exp\bigl(-\frac{(d-r-1)\alpha^2}{8(2R+\rho)^2}\bigr)\le\frac14$.
Integrating over the seed and the pool gives the same bound.

On this event, suppose the two runs coincide before some query. Then the
query is the same in both runs and, since either $\|x\|\ge2R+\rho$ or
\eqref{lo:local-geometric-condition} holds, the two replies coincide;
the disclosures coincide because they follow the same rule. By induction,
the two runs coincide.
\end{proof}

\begin{lemma}[Loss gap]
\label{lo:gap}
If $|\langle v_{s+1},x\rangle|\le\alpha/2$, then
\[
    f_V(x)\ge\frac{9H_sR^2}{8192(s+1)^2}.
\]
\end{lemma}

\begin{proof}
Put $t_0=|a-\langle v_1,x\rangle|$ and
$t_l=(|\langle v_l-v_{l+1},x\rangle|-\alpha)_+$ for $1\le l\le s$; note
that $s\le r-1$. Since
$a=(a-\langle v_1,x\rangle)+\sum_{l=1}^s\langle v_l-v_{l+1},x\rangle
+\langle v_{s+1},x\rangle$,
\[
    a\le t_0+\sum_{l=1}^s(t_l+\alpha)+\frac\alpha2,
    \qquad\text{hence}\qquad
    \sum_{l=0}^st_l\ge a-\Bigl(s+\frac12\Bigr)\alpha\ge\frac{3a}4 .
\]
The anchor term of $f_V$ equals $\eta_a(t_0)$, and the edge terms with
$1\le l\le s$ equal $\eta_a(t_l)$. If some $t_l\ge a$, then
$\eta_a(t_l)\ge a^2\ge\frac{9a^2}{16(s+1)}$. Otherwise
$\eta_a(t_l)=t_l^2$ for all $l\le s$, and the Cauchy--Schwarz inequality
gives $\sum_{l=0}^st_l^2\ge\frac{9a^2}{16(s+1)}$. In both cases
$f_V(x)\ge\frac{H_s}{256}\cdot\frac{9a^2}{16(s+1)}$, and
$a^2=\frac{R^2}{2(s+1)}$.
\end{proof}

\begin{theorem}[Lower bound with sample reuse]
\label{lo:combined}
Let $S\in\mathbb N\cup\{\infty\}$, $S\ge1$. For every
$d\ge d_{\mathrm{pool}}$,
\[
    \mathfrak M_d^{\mathrm{pool}}(S)
    \ge\frac{R^2}{10^6\log^2(em)}\min\left\{L,\;
        \frac{\min\{L,H_s/K\}}{N^2}+\frac{H_s}{\min\{S,mKN\}}\right\}.
\]
\end{theorem}

\begin{proof}
\emph{The chain term.}
By Lemma~\ref{lo:progress}, $\mathbb{P}\left(Z\le s-1\right)\ge3/4$: if $s-1=KN$ this
holds surely, and otherwise $s-1=\lceil4\nu\rceil$ and Markov's inequality
applies ($\nu>0$ because $p>0$). The final common prefix length is $1+Z$.
Intersect $\{Z\le s-1\}$ with the event of Lemma~\ref{lo:geometry}; the
intersection has probability at least $1/2$. On it, the output
$\widehat x$ of the strengthened algorithm equals the simulated output,
and its final common prefix length is at most $s$. If
$\|\widehat x\|<2R+\rho$, then $|\langle v_{s+1},\widehat x\rangle|\le
\alpha/2$, and \eqref{lo:population-domination} and Lemma~\ref{lo:gap}
give $F_V(\widehat x)\ge\frac{9pH_sR^2}{8192(s+1)^2}$. If
$\|\widehat x\|\ge2R+\rho$, then
$F_V(\widehat x)=\frac L{16}(\|\widehat x\|-R)^2\ge\frac L{16}\rho^2
=\frac{64H_s^2R^2}{L}\ge64H_sR^2$, which is larger. Hence
\[
    \mathbb E F_V(\widehat x)\ge\frac{9pH_sR^2}{16384(s+1)^2},
\]
and averaging over $V$ selects a fixed instance. By
\eqref{eq:pool-parameters}, $Kp\le\log(em)$, so
$\nu\le eN\log(em)<3N\log(em)$ and
$s+1\le3+4\nu\le15N\log(em)$. Also
$pH_s=\min\{L,H_s\log(em)/K\}\ge\min\{L,H_s/K\}$. Therefore
\[
    \mathfrak M_d^{\mathrm{pool}}(S)
    \ge\frac{9R^2\min\{L,H_s/K\}}{16384\cdot225\,N^2\log^2(em)}
    =\frac{R^2\min\{L,H_s/K\}}{409600\,N^2\log^2(em)} .
\]
The pool size did not enter this argument: a finite pool consists of $S$
independent coin vectors, and invalid labels disclose nothing.

\emph{Combination.}
Lemma~\ref{lem:stat} gives
$\mathfrak M_d^{\mathrm{pool}}(S)\ge\frac{R^2}8\min\{L,\frac{H_s}{\min\{S,mKN\}}\}$.
The larger of the two bounds is at least their average, and
$\frac1{16}\ge\frac1{819200\log^2(em)}$, so
\begin{align*}
    \mathfrak M_d^{\mathrm{pool}}(S)
    &\ge\frac{R^2}{819200\log^2(em)}
    \left[\frac{\min\{L,H_s/K\}}{N^2}
      +\min\left\{L,\frac{H_s}{\min\{S,mKN\}}\right\}\right]\\
    &\ge\frac{R^2}{10^6\log^2(em)}\min\left\{L,\;
        \frac{\min\{L,H_s/K\}}{N^2}+\frac{H_s}{\min\{S,mKN\}}\right\},
\end{align*}
where the last inequality holds exactly as in the proof of
Theorem~\ref{qlow:allbudget}.
\end{proof}

\section{Upper bounds with fresh samples}
\label{app:appendix_B}

In this appendix the oracle returns, at every query $x$, a vector
$g(x,\xi)$ with a fresh sample $\xi$ and the two properties in
\eqref{eq:wg}, that is, Assumption~\ref{ass:wg} holds. By
Lemma~\ref{lem:relations}, every result below applies to the
sample-gradient oracle of Assumption~\ref{ass:sg}. We normalize
$F(x^*)=0$. If $\mathcal H$ denotes the history before a query $x$,
then $x$ is $\mathcal H$-measurable and $\xi$ is independent of
$\mathcal H$, so that \eqref{eq:wg} holds conditionally:
\begin{equation}
    \mathbb E[g(x,\xi)\mid\mathcal H]=\nabla F(x),
    \qquad
    \mathbb E\bigl[\|g(x,\xi)\|^2\mid\mathcal H\bigr]\le2H_sF(x).
    \label{up:wg}
\end{equation}

\subsection{Preliminaries}
\label{app:B-prelim}

\begin{lemma}[Smooth functions]
\label{lem:smooth}
Let $\psi:\R^d\to\R$ be differentiable with an $M$-Lipschitz gradient,
$M>0$, and let $u,v\in\R^d$.
\begin{enumerate}
    \item
    $\psi(v)\le \psi(u)+\langle\nabla \psi(u),v-u\rangle+\frac M2\|v-u\|^2$.

    \item
    If $\psi\ge0$, then $\|\nabla \psi(u)\|^2\le2M\psi(u)$ and
    $\psi(v)\le2\psi(u)+M\|v-u\|^2$.

    \item
    If $\psi$ is convex, then
    $\psi(v)-\psi(u)-\langle\nabla \psi(u),v-u\rangle
    \ge\frac1{2M}\|\nabla \psi(v)-\nabla \psi(u)\|^2$ and
    $\|\nabla \psi(v)-\nabla \psi(u)\|^2
    \le M\langle\nabla \psi(v)-\nabla \psi(u),v-u\rangle$.

    \item
    If $x^*$ minimizes $\psi$, then $\nabla \psi(x^*)=0$,
    $0\le \psi(u)-\psi(x^*)\le\frac M2\|u-x^*\|^2$, and
    $\|\nabla \psi(u)\|\le M\|u-x^*\|$.
\end{enumerate}
\end{lemma}

\begin{proof}
\emph{Part 1.}
By the fundamental theorem of calculus and the Lipschitz property of
$\nabla \psi$,
\[
    \psi(v)-\psi(u)-\langle\nabla \psi(u),v-u\rangle
    =\int_0^1\bigl\langle\nabla \psi(u+s(v-u))-\nabla \psi(u),v-u\bigr\rangle\,ds
    \le\int_0^1Ms\|v-u\|^2\,ds
    =\frac M2\|v-u\|^2.
\]

\emph{Part 2.}
Apply Part~1 with $v=u-\frac1M\nabla \psi(u)$:
\[
    0\le \psi\Bigl(u-\frac1M\nabla \psi(u)\Bigr)
    \le \psi(u)-\frac1M\|\nabla \psi(u)\|^2+\frac1{2M}\|\nabla \psi(u)\|^2
    =\psi(u)-\frac1{2M}\|\nabla \psi(u)\|^2 .
\]
This proves the first inequality. For the second, Part~1, the
Cauchy--Schwarz inequality, and the first inequality give
\[
    \psi(v)\le \psi(u)+\sqrt{2M\psi(u)}\,\|v-u\|+\frac M2\|v-u\|^2 .
\]
Since $2\sqrt{ab}\le a+b$ with $a=\psi(u)$ and $b=\frac M2\|v-u\|^2$, the
middle term is at most $\psi(u)+\frac M2\|v-u\|^2$.

\emph{Part 3.}
The function
$r_u(v)=\psi(v)-\psi(u)-\langle\nabla \psi(u),v-u\rangle$
is nonnegative by convexity and has the $M$-Lipschitz gradient
$\nabla r_u(v)=\nabla \psi(v)-\nabla \psi(u)$. Part~2 applied to $r_u$ at $v$
gives the first inequality. Adding it to the same inequality with $u$
and $v$ interchanged, and using
$r_u(v)+r_v(u)=\langle\nabla \psi(v)-\nabla \psi(u),v-u\rangle$, gives the
second.

\emph{Part 4.}
The gradient of a differentiable function vanishes at a minimizer. The
upper bound on $\psi(u)-\psi(x^*)$ is Part~1 at the point $x^*$, and
$\|\nabla \psi(u)\|=\|\nabla \psi(u)-\nabla \psi(x^*)\|\le M\|u-x^*\|$.
\end{proof}

\begin{proof}[Proof of Lemma~\ref{lem:relations} and of \eqref{eq:pool-fresh}]
Fix an instance of $\mathcal S_d$. For almost every $\xi$, the function
$f_\xi$ is convex and $\nabla f_\xi(x^*)=0$, so $x^*$ minimizes
$f_\xi$. Lemma~\ref{lem:smooth}, Parts~4 and~2, applied to the
nonnegative function $f_\xi-f_\xi(x^*)$, gives for every $x$
\begin{gather}
    0\le f_\xi(x)-f_\xi(x^*)\le\frac{H_s}2\|x-x^*\|^2,
    \qquad
    \|\nabla f_\xi(x)\|\le H_s\|x-x^*\|,
    \label{up:sample-bounds}\\
    \|\nabla f_\xi(x)\|^2\le2H_s\bigl(f_\xi(x)-f_\xi(x^*)\bigr).
    \label{up:sample-growth}
\end{gather}
Fix $x,u\in\R^d$. For $0<|s|\le1$, the mean value theorem and
\eqref{up:sample-bounds} give
\[
    \frac{|f_\xi(x+su)-f_\xi(x)|}{|s|}
    \le\sup_{|\sigma|\le1}\|\nabla f_\xi(x+\sigma u)\|\,\|u\|
    \le H_s\bigl(\|x-x^*\|+\|u\|\bigr)\|u\|,
\]
a bound independent of $\xi$ and $s$. By dominated convergence, the
directional derivative of $F$ at $x$ along $u$ equals
$\mathbb E\langle\nabla f_\xi(x),u\rangle$. Since $F$ is differentiable,
$\nabla F(x)=\mathbb E\nabla f_\xi(x)$. Taking expectations in
\eqref{up:sample-growth} gives
$\mathbb E\|\nabla f_\xi(x)\|^2\le2H_sF(x)=2H_s(F(x)-F(x^*))$. The
function $F$ is convex, has an $L$-Lipschitz gradient, and is minimized
at $x^*$. Hence the instance, with the oracle
$g(x,\xi)=\nabla f_\xi(x)$, belongs to $\mathcal W_d$, and to
$\mathcal W_d^{\mathrm{quad}}$ if $F$ is quadratic.

The algorithm observes the same vectors in both interpretations of the
instance, so every algorithm has the same risk in both, which gives
$\mathfrak M_d^{\mathrm{sg}}\le\mathfrak M_d^{\mathrm{wg}}$ and
$\mathfrak M_d^{\mathrm{sg,quad}}\le\mathfrak M_d^{\mathrm{wg,quad}}$.
For \eqref{eq:pool-fresh}, an algorithm
with fresh access is run with a pool of size $S=\infty$ by assigning the
label $((n-1)K+k-1)m+w$ to the $k$-th query of worker $w$ in round $n$.
These labels are distinct, so the replies are computed at independent
fresh samples and the run has the same law as with fresh access.
\end{proof}

\begin{lemma}[Moments]
\label{up:moments}
Consider an instance of $\mathcal W_d$ with $F(x^*)=0$.
\begin{enumerate}
    \item
    $F(x)\le\frac L2\|x-x^*\|^2$ and
    $\|\nabla F(x)\|^2\le 2LF(x)$.

    \item
    If $\bar g_b(x)$ is the average of $b\ge1$ independent fresh
    responses at the same point $x$ that is measurable with respect to
    the history $\mathcal H$ before these responses, then
    \begin{equation}
        \mathbb E[\bar g_b(x)\mid\mathcal H]=\nabla F(x),
        \qquad
        \mathbb E\!\left[
            \|\bar g_b(x)-\nabla F(x)\|^2\mid\mathcal H
        \right]
        \le \frac{2H_s}{b}F(x).
        \label{up:batch-variance}
    \end{equation}

    \item
    Every finite recurrence built from deterministic finite linear
    combinations of previous states and fresh oracle responses has
    finite second moments, provided its initial states have finite
    second moments.
\end{enumerate}
\end{lemma}

\begin{proof}
\emph{Part 1.}
This is Lemma~\ref{lem:smooth}, Parts~4 and~2, applied to $\psi=F$, which
is nonnegative, has an $L$-Lipschitz gradient, and is minimized at
$x^*$.

\medskip
\noindent
\emph{Part 2.}
Write
\[
    g_i=g(x,\xi_i),
    \qquad
    \bar g_b(x)=\frac1b\sum_{i=1}^b g_i,
    \qquad
    \zeta_i=g_i-\nabla F(x).
\]
By \eqref{up:wg},
\[
    \mathbb E[g_i\mid\mathcal H]=\nabla F(x),
    \qquad
    \mathbb E[\zeta_i\mid\mathcal H]=0,
    \qquad
    \mathbb E[\bar g_b(x)\mid\mathcal H]
    =\frac1b\sum_{i=1}^b
       \mathbb E[g_i\mid\mathcal H]
    =\nabla F(x).
\]
Conditional on $\mathcal H$, the vectors $\zeta_i$
are independent. Thus, for $i\ne j$,
\[
    \mathbb E[
        \langle\zeta_i,\zeta_j\rangle
        \mid\mathcal H
    ]
    =
    \left\langle
        \mathbb E[\zeta_i\mid\mathcal H],
        \mathbb E[\zeta_j\mid\mathcal H]
    \right\rangle
    =0.
\]
Expanding the squared norm therefore gives
\begin{align*}
    \mathbb E\!\left[
        \|\bar g_b(x)-\nabla F(x)\|^2
        \mid\mathcal H
    \right]
    &=
    \frac1{b^2}\sum_{i=1}^b
        \mathbb E[\|\zeta_i\|^2\mid\mathcal H]
    +\frac2{b^2}\sum_{i<j}
        \mathbb E[
            \langle\zeta_i,\zeta_j\rangle
            \mid\mathcal H
        ]
    =
    \frac1b\,
    \mathbb E\!\left[
        \|g(x,\xi)-\nabla F(x)\|^2
        \mid\mathcal H
    \right],
\end{align*}
where the last equality also uses that the $b$ responses have the same
conditional distribution. Expanding one centered second moment and using
\eqref{up:wg} once more,
\begin{align*}
    \mathbb E\!\left[
        \|g(x,\xi)-\nabla F(x)\|^2
        \mid\mathcal H
    \right]
    &=
    \mathbb E[\|g(x,\xi)\|^2\mid\mathcal H]
    -2\left\langle
        \mathbb E[g(x,\xi)\mid\mathcal H],
        \nabla F(x)
    \right\rangle
    +\|\nabla F(x)\|^2\\
    &=
    \mathbb E[\|g(x,\xi)\|^2\mid\mathcal H]
    -\|\nabla F(x)\|^2
    \le 2H_sF(x).
\end{align*}
Combining the two displays proves \eqref{up:batch-variance}.

\medskip
\noindent
\emph{Part 3.}
First consider a query $X$ that is measurable with respect to the
history $\mathcal H$ before it and satisfies $\mathbb E\|X\|^2<\infty$,
and let $\xi$ be a fresh sample. By \eqref{up:wg} and Part~1,
\begin{align*}
    \mathbb E\|g(X,\xi)\|^2
    &=\mathbb E\!\left[
        \mathbb E[\|g(X,\xi)\|^2\mid\mathcal H]
      \right]
    \le 2H_s\,\mathbb E F(X)
    \le H_sL\,\mathbb E\|X-x^*\|^2
    \le 2H_sL
       \bigl(\mathbb E\|X\|^2+\|x^*\|^2\bigr)
    <\infty.
\end{align*}
Thus a fresh response at a square-integrable query is itself
square-integrable. Next, let $Z_1,\ldots,Z_s$ be square-integrable
random vectors and let $a_1,\ldots,a_s$ be deterministic coefficients.
The triangle and Cauchy--Schwarz inequalities give
\[
    \left\|\sum_{j=1}^s a_jZ_j\right\|^2
    \le
    \left(\sum_{j=1}^s |a_j|\|Z_j\|\right)^2
    \le
    s\sum_{j=1}^s a_j^2\|Z_j\|^2,
\]
so every such linear combination is square-integrable; independence of
the $Z_j$ is not required. Starting from the square-integrable initial
states, forming a query, drawing a fresh response at it, and forming the
next states each preserve square integrability. Induction over the
finite sequence of these operations proves Part~3.
\end{proof}

\begin{lemma}[Serial SGD]
\label{lem:serial}
Let $s\ge1$. Set $X_0=x_0$ and
\[
    X_{k+1}=X_k-\frac1{2H_s}\,g(X_k,\xi_k),
    \qquad k=0,\ldots,s-1,
\]
with independent fresh samples $\xi_k$, and let
$\bar X=\frac1s\sum_{k=0}^{s-1}X_k$. Then
$\mathbb E F(\bar X)\le\frac{2H_s\|x_0-x^*\|^2}{s}$.
\end{lemma}

\begin{proof}
All iterates are square-integrable by Lemma~\ref{up:moments}, Part~3.
Conditional on $X_k$ and the preceding history, \eqref{up:wg} and the
convexity inequality $\langle\nabla F(X_k),X_k-x^*\rangle\ge F(X_k)$
give
\begin{align*}
    \mathbb E\bigl[\|X_{k+1}-x^*\|^2\mid X_k\bigr]
    &=\|X_k-x^*\|^2
      -\frac1{H_s}\langle\nabla F(X_k),X_k-x^*\rangle
      +\frac1{4H_s^2}\,\mathbb E\bigl[\|g(X_k,\xi_k)\|^2\mid X_k\bigr]\\
    &\le\|X_k-x^*\|^2-\frac1{H_s}F(X_k)+\frac1{2H_s}F(X_k)
    =\|X_k-x^*\|^2-\frac1{2H_s}F(X_k).
\end{align*}
Taking expectations and summing over $k=0,\ldots,s-1$ gives
$\sum_{k=0}^{s-1}\mathbb E F(X_k)\le2H_s\|x_0-x^*\|^2$. Convexity of
$F$ gives
$\mathbb E F(\bar X)\le\frac1s\sum_{k=0}^{s-1}\mathbb E F(X_k)$.
\end{proof}

\subsection{Quadratic populations}
\label{app:appendix_B_quadratic}

In this subsection
\[
    F(x)=\frac12\langle x-x^*,\mathsf A(x-x^*)\rangle,
    \qquad 0\preceq\mathsf A\preceq L\mathrm I,
\]
so that $\nabla F(x)=\mathsf A(x-x^*)$ and
$\langle\nabla F(x),x-x^*\rangle=2F(x)$.

\paragraph{One local block.}
Let $y$ be a point that is measurable with respect to the history
$\mathcal H$ at the start of a round, and write
$\mathbb E_{\mathcal H}[\cdot]=\mathbb E[\cdot\mid\mathcal H]$. A worker
performs $K$ fresh queries,
\begin{equation}
    z_0=y,
    \qquad
    z_{j+1}=z_j-\frac1{H_s}\,g(z_j,\xi_j),
    \qquad 0\le j<K,
    \label{quad:block}
\end{equation}
and forms
\[
    G=\frac{H_s}{K}(y-z_K),
    \qquad
    \bar z=\frac1K\sum_{j=0}^{K-1}z_j .
\]
The worker sends $G$ and $\bar z$. The quantity
\[
    \Phi(y)=\frac1K\sum_{j=0}^{K-1}\mathbb E_{\mathcal H}F(z_j)
\]
is used only in the analysis.

\begin{lemma}[Energy of one block]
\label{quad:energy}
If $\mathbb E\|y\|^2<\infty$, all states and responses in
\eqref{quad:block} are square-integrable, and
\begin{equation}
    \Phi(y)+\frac{K}{2H_s}\,\mathbb E_{\mathcal H}\|G\|^2
    \le\langle\mathbb E_{\mathcal H}G,\,y-x^*\rangle,
    \qquad
    \mathbb E_{\mathcal H}F(\bar z)\le\Phi(y).
    \label{quad:joint-energy}
\end{equation}
\end{lemma}

\begin{proof}
Square integrability follows from Lemma~\ref{up:moments}, Part~3. Fix
$j<K$ and condition on the history up to $z_j$. By \eqref{up:wg} and
$\langle\mathsf A(z_j-x^*),z_j-x^*\rangle=2F(z_j)$,
\begin{align*}
    \mathbb E\bigl[\|z_{j+1}-x^*\|^2\mid z_j\bigr]
    &=\|z_j-x^*\|^2
      -\frac2{H_s}\langle\mathsf A(z_j-x^*),z_j-x^*\rangle
      +\frac1{H_s^2}\,\mathbb E\bigl[\|g(z_j,\xi_j)\|^2\mid z_j\bigr]\\
    &\le\|z_j-x^*\|^2-\frac4{H_s}F(z_j)+\frac2{H_s}F(z_j)
    =\|z_j-x^*\|^2-\frac2{H_s}F(z_j).
\end{align*}
Applying $\mathbb E_{\mathcal H}$ and summing over $j=0,\ldots,K-1$
gives
\[
    \Phi(y)\le\frac{H_s}{2K}
    \Bigl(\|y-x^*\|^2-\mathbb E_{\mathcal H}\|z_K-x^*\|^2\Bigr).
\]
Since $z_K=y-\frac K{H_s}G$,
\[
    \|z_K-x^*\|^2
    =\|y-x^*\|^2-\frac{2K}{H_s}\langle G,y-x^*\rangle
     +\frac{K^2}{H_s^2}\|G\|^2 ,
\]
and substituting this identity into the previous bound proves the first
inequality in \eqref{quad:joint-energy}. The second follows from
convexity of $F$.
\end{proof}

\paragraph{The spectral surrogate.}
Define
\begin{equation}
    \mathsf A_K=\frac{H_s}{K}
        \left[\mathrm I-\Bigl(\mathrm I-\frac{\mathsf A}{H_s}\Bigr)^K\right],
    \qquad
    F_K(y)=\frac12\langle y-x^*,\mathsf A_K(y-x^*)\rangle,
    \qquad
    L_K=\min\left\{L,\frac{H_s}K\right\}.
    \label{quad:surrogate}
\end{equation}
Neither $\mathsf A_K$ nor $F_K$ is used by the algorithm.

\begin{lemma}[Spectral mean of a block]
\label{quad:certificate}
The function $F_K$ is convex, has an $L_K$-Lipschitz gradient, and has
the same minimizers as $F$. For the block \eqref{quad:block},
\begin{equation}
    \mathbb E_{\mathcal H}G=\nabla F_K(y),
    \qquad
    \Phi(y)+\frac{K}{2H_s}\,\mathbb E_{\mathcal H}\|G\|^2\le2F_K(y).
    \label{quad:certificate-formula}
\end{equation}
\end{lemma}

\begin{proof}
By \eqref{up:wg},
$\mathbb E[z_{j+1}-x^*\mid z_j]=(\mathrm I-\mathsf A/H_s)(z_j-x^*)$.
By the tower property,
$\mathbb E_{\mathcal H}(z_K-x^*)=(\mathrm I-\mathsf A/H_s)^K(y-x^*)$,
and therefore
\[
    \mathbb E_{\mathcal H}G
    =\frac{H_s}K\Bigl[(y-x^*)-\Bigl(\mathrm I-\frac{\mathsf A}{H_s}\Bigr)^K(y-x^*)\Bigr]
    =\mathsf A_K(y-x^*)=\nabla F_K(y).
\]
The matrices $\mathsf A$ and $\mathsf A_K$ have common eigenvectors. If
$\theta\in[0,L]$ is an eigenvalue of $\mathsf A$, the corresponding
eigenvalue of $\mathsf A_K$ is
\[
    \theta_K=\frac{H_s}{K}\left[1-\Bigl(1-\frac\theta{H_s}\Bigr)^K\right].
\]
Since $0\le\theta/H_s\le1$ and, by Bernoulli's inequality,
$0\le1-(1-c)^K\le\min\{Kc,1\}$ for $c\in[0,1]$, we get
$0\le\theta_K\le\min\{\theta,H_s/K\}\le L_K$; moreover, $\theta_K=0$ if
and only if $\theta=0$. Hence $0\preceq\mathsf A_K\preceq L_K\mathrm I$
and $\ker\mathsf A_K=\ker\mathsf A$. This proves convexity and smoothness
of $F_K$, and
$\arg\min F_K=x^*+\ker\mathsf A=\arg\min F$. Finally,
$\langle\nabla F_K(y),y-x^*\rangle=2F_K(y)$, so the inequality in
\eqref{quad:certificate-formula} follows from
\eqref{quad:joint-energy}.
\end{proof}

\paragraph{The algorithm.}

\begin{algorithm}[H]
\caption{Accelerated spectral blocks}
\label{alg:spectral-blocks}
\begin{algorithmic}[1]
\Require Initial point $x_0$, parameters $L,H_s>0$, integers
         $m,K,N\ge1$, and oracle $g$.
\State $\displaystyle L_K\gets\min\left\{L,\frac{H_s}K\right\}$,
       $\displaystyle\gamma\gets\min\left\{\frac1{L_K},
       \frac{mK}{2H_s(N+2)}\right\}$
\State $v_0\gets x_0$, $A_0\gets0$
\For{$t=1,\ldots,N$}
    \State $\displaystyle a_t\gets\frac{\gamma(t+1)}2$,
           $\displaystyle A_t\gets\frac{\gamma t(t+3)}4$
    \State $\displaystyle y_t\gets\frac{A_{t-1}x_{t-1}+a_tv_{t-1}}{A_t}$,
           and broadcast $y_t$
    \State Each worker $i$ runs the block \eqref{quad:block} from $y_t$
           with $K$ fresh queries and sends $(G_{i,t},\bar z_{i,t})$
    \State $\displaystyle\bar G_t\gets\frac1m\sum_{i=1}^mG_{i,t}$,
           $\displaystyle\hat z_t\gets\frac1m\sum_{i=1}^m\bar z_{i,t}$
    \State $x_t\gets y_t-\gamma\bar G_t$,
           $v_t\gets v_{t-1}-a_t\bar G_t$
\EndFor
\State \Return $\displaystyle\widehat x=\frac1{A_N}\sum_{t=1}^Na_t\hat z_t$
\end{algorithmic}
\end{algorithm}

Algorithm~\ref{alg:spectral-blocks} uses exactly $K$ fresh queries per
worker in each of its $N$ rounds; the weighted sum in the output is
accumulated online.

\begin{theorem}[Upper bound for quadratic populations]
\label{quad:rate}
Let $\|x_0-x^*\|\le R$. Algorithm~\ref{alg:spectral-blocks} satisfies
\begin{equation}
    \mathbb E F(\widehat x)\le\frac{R^2}{A_N}
    \le R^2\left[\frac{4L_K}{N^2}+\frac{8H_s}{mKN}\right].
    \label{quad:rate-formula}
\end{equation}
\end{theorem}

\begin{proof}
All states have finite second moments by Lemma~\ref{quad:energy} and
induction over the rounds, since the server states are linear
combinations of the received vectors with deterministic coefficients.
Let $\mathcal H_t$ be the history at the beginning of round $t$, after
$y_t$ is formed, and write $\mathbb E_t[\cdot]=\mathbb E[\cdot\mid
\mathcal H_t]$. Conditional on $\mathcal H_t$, the pairs
$(G_{i,t},\bar z_{i,t})$, $i=1,\ldots,m$, are independent and identically
distributed, and Lemmas~\ref{quad:energy} and~\ref{quad:certificate}
apply to each of them with $y=y_t$. Consider the nonnegative potential
\[
    P_t=A_tF_K(x_t)+\frac12\|v_t-x^*\|^2,
    \qquad t\ge0 .
\]

\emph{One-step inequality.}
Lemma~\ref{lem:smooth}, Part~1, for $F_K$ and
$\mathbb E_t\bar G_t=\nabla F_K(y_t)$ give
\[
    \mathbb E_tF_K(x_t)
    \le F_K(y_t)-\gamma\|\nabla F_K(y_t)\|^2
       +\frac{L_K\gamma^2}2\,\mathbb E_t\|\bar G_t\|^2 .
\]
The update of $v_t$ gives
\[
    \frac12\mathbb E_t\|v_t-x^*\|^2
    =\frac12\|v_{t-1}-x^*\|^2
     -a_t\langle\nabla F_K(y_t),v_{t-1}-x^*\rangle
     +\frac{a_t^2}2\,\mathbb E_t\|\bar G_t\|^2 .
\]
Multiplying the first inequality by $A_t$ and adding the second, we
obtain
\begin{align}
    \mathbb E_tP_t-P_{t-1}
    \le{}&A_tF_K(y_t)-A_{t-1}F_K(x_{t-1})
          -a_t\langle\nabla F_K(y_t),v_{t-1}-x^*\rangle\notag\\
    &-A_t\gamma\|\nabla F_K(y_t)\|^2
     +\frac{A_tL_K\gamma^2+a_t^2}2\,\mathbb E_t\|\bar G_t\|^2 .
    \label{quad:potential-step}
\end{align}
Since $A_t=A_{t-1}+a_t$ and $A_ty_t=A_{t-1}x_{t-1}+a_tv_{t-1}$, we have
$a_t(v_{t-1}-x^*)=A_{t-1}(y_t-x_{t-1})+a_t(y_t-x^*)$. Together with
$\langle\nabla F_K(y_t),y_t-x^*\rangle=2F_K(y_t)$, this shows that the
first line of the right-hand side of \eqref{quad:potential-step} equals
\begin{align*}
    &A_{t-1}\bigl[F_K(y_t)-F_K(x_{t-1})
        -\langle\nabla F_K(y_t),y_t-x_{t-1}\rangle\bigr]
     +a_t\bigl[F_K(y_t)-\langle\nabla F_K(y_t),y_t-x^*\rangle\bigr]\\
    &\qquad=-A_{t-1}D_{F_K}(x_{t-1},y_t)-a_tF_K(y_t),
\end{align*}
where $D_{F_K}(u,v)=F_K(u)-F_K(v)-\langle\nabla F_K(v),u-v\rangle\ge0$.

\emph{Coefficients.}
For $t\ge1$, $(t+1)^2\le t(t+3)$, hence $a_t^2\le\gamma A_t$. Since
$\gamma L_K\le1$, also $A_tL_K\gamma^2\le\gamma A_t$, so the coefficient
of $\mathbb E_t\|\bar G_t\|^2$ in \eqref{quad:potential-step} is at most
$\gamma A_t$. The vectors $G_{i,t}-\nabla F_K(y_t)$ are conditionally
independent and centered, which gives the exact identity
\begin{equation}
    \mathbb E_t\|\bar G_t\|^2
    =\Bigl(1-\frac1m\Bigr)\|\nabla F_K(y_t)\|^2
     +\frac1m\,\mathbb E_t\|G_{1,t}\|^2 .
    \label{quad:pooling}
\end{equation}

\emph{Spending the certificate.}
By \eqref{quad:certificate-formula} with $y=y_t$,
\[
    -a_tF_K(y_t)\le-\frac{a_t}2\Phi(y_t)
    -\frac{a_tK}{4H_s}\,\mathbb E_t\|G_{1,t}\|^2 .
\]
Substituting this bound, the bound on the coefficient, and
\eqref{quad:pooling} into \eqref{quad:potential-step} yields
\[
    \mathbb E_tP_t-P_{t-1}
    \le-A_{t-1}D_{F_K}(x_{t-1},y_t)-\frac{a_t}2\Phi(y_t)
       -\frac{\gamma A_t}{m}\|\nabla F_K(y_t)\|^2
       +\left[\frac{\gamma A_t}m-\frac{a_tK}{4H_s}\right]
        \mathbb E_t\|G_{1,t}\|^2 .
\]
Since $t(t+3)\le(t+1)(t+2)$ and $t\le N$,
\[
    \frac{4\gamma A_tH_s}{a_tmK}
    =\frac{2\gamma H_s}{mK}\cdot\frac{t(t+3)}{t+1}
    \le\frac{2\gamma H_s(N+2)}{mK}\le1,
\]
so the last bracket is nonpositive. All terms except
$-\frac{a_t}2\Phi(y_t)$ on the right-hand side are therefore
nonpositive, and
\begin{equation}
    \mathbb E_tP_t+\frac{a_t}2\Phi(y_t)\le P_{t-1}.
    \label{quad:telescoping}
\end{equation}

\emph{The output.}
Taking expectations in \eqref{quad:telescoping}, summing over
$t=1,\ldots,N$, and using $P_N\ge0$ and $P_0=\frac12\|x_0-x^*\|^2\le
\frac{R^2}2$, we get $\sum_{t=1}^Na_t\mathbb E\Phi(y_t)\le R^2$. By
convexity and Lemma~\ref{quad:energy},
$\mathbb E_tF(\hat z_t)\le\frac1m\sum_{i=1}^m\mathbb E_tF(\bar z_{i,t})
\le\Phi(y_t)$. Since $\sum_{t=1}^Na_t=A_N$, convexity of $F$ gives
\[
    \mathbb E F(\widehat x)
    \le\frac1{A_N}\sum_{t=1}^Na_t\,\mathbb E F(\hat z_t)
    \le\frac{R^2}{A_N}.
\]
Finally,
\[
    \frac1{A_N}
    =\frac4{N(N+3)}\max\left\{L_K,\frac{2H_s(N+2)}{mK}\right\}
    \le\frac{4L_K}{N^2}+\frac{8H_s}{mKN}.
\]
\end{proof}

\begin{corollary}[Quadratic selector]
\label{quad:selector}
Selecting in advance between Algorithm~\ref{alg:spectral-blocks} and
returning $x_0$ gives
\begin{equation}
    \mathbb E F(\widehat x)
    \le R^2\min\left\{\frac L2,\;
       \frac{4\min\{L,H_s/K\}}{N^2}+\frac{8H_s}{mKN}\right\}
    \le8R^2\min\left\{L,\;
       \frac{\min\{L,H_s/K\}}{N^2}+\frac{H_s}{mKN}\right\}.
    \label{quad:selector-rate}
\end{equation}
\end{corollary}

\begin{proof}
By Lemma~\ref{up:moments}, Part~1, $F(x_0)\le\frac L2R^2$, and
Theorem~\ref{quad:rate} bounds the risk of the other branch. Both bounds
depend only on the public parameters, so the selection is made before the
run. The second inequality follows by increasing all coefficients to
$8$.
\end{proof}

\subsection{General convex populations}
\label{app:appendix_B_convex}

The upper bound uses the following accelerated method with a fresh
batch at every update.

\begin{algorithm}[H]
\caption{Fresh-batch two-state accelerated recurrence}
\label{alg:fresh-batch-acsa}
\begin{algorithmic}[1]
\Require Initial point $x_0$, parameters $L,H_s>0$,
         integers $q,b\ge1$, and oracle $g$.
\State $\displaystyle
    \gamma \gets
    \min\left\{\frac{1}{48L},\frac{b}{48H_s q}\right\}$
\State $z_0\gets x_0$, $v_0\gets x_0$
\For{$t=0,\ldots,q-1$}
    \State $\displaystyle \alpha_t\gets\frac{6}{t+6}$,
           $\eta_t\gets\gamma(t+1)$
    \State $y_t\gets(1-\alpha_t)z_t+\alpha_t v_t$
    \State Draw $\xi_{t,1},\ldots,\xi_{t,b}$ independently
           from the oracle law, independently of the past.
    \State $\displaystyle
        \bar g_t\gets\frac1b\sum_{i=1}^b g(y_t,\xi_{t,i})$
    \State $v_{t+1}\gets v_t-\eta_t\bar g_t$
    \State $z_{t+1}\gets(1-\alpha_t)z_t+\alpha_t v_{t+1}$
\EndFor
\State \Return $z_q$
\end{algorithmic}
\end{algorithm}

The algorithm uses exactly $bq$ fresh oracle responses.
It requires no loss values, projections, or exact
population-gradient evaluations.

\begin{theorem}[Accelerated common-query weak-growth bound]
\label{thm:fresh-batch-acsa}
Consider an instance of $\mathcal W_d$ with $F(x^*)=0$ and
$\|x_0-x^*\|\le R$. Then Algorithm~\ref{alg:fresh-batch-acsa} satisfies
\begin{equation}
    \mathbb E F(z_q)
    \le
    288R^2
    \max\left\{
        \frac{L}{q^2},
        \frac{H_s}{bq}
    \right\}.
    \label{eq:fresh-batch-acsa-rate}
\end{equation}
\end{theorem}

\begin{proof}
By Lemma~\ref{up:moments}, Part~3, all iterates and oracle
responses have finite second moments.
Let $\mathcal F_t$ denote the history before the batch
at iteration $t$ is sampled, and write
\[
    \mathbb E_t[\cdot]
    =\mathbb E[\cdot\mid\mathcal F_t],
    \qquad
    A_t=\frac{\eta_t}{\alpha_t}
        =\frac{\gamma(t+1)(t+6)}6 .
\]
The vectors $z_t,v_t,y_t$ are $\mathcal F_t$-measurable.

\medskip
\noindent
\emph{Conditional variance.}
Let $\sigma_t^2=\mathbb E_t\|\bar g_t-\nabla F(y_t)\|^2$. By
\eqref{up:batch-variance},
\[
    \mathbb E_t\bar g_t=\nabla F(y_t),
    \qquad
    \sigma_t^2\le\frac{2H_s}{b}F(y_t).
\]
Since $y_t-z_t=\alpha_t(v_t-z_t)$, Lemma~\ref{lem:smooth}, Part~2,
applied to the nonnegative function $F$ gives
$F(y_t)\le2F(z_t)+L\alpha_t^2\|v_t-z_t\|^2$, and hence
\begin{equation}
    \sigma_t^2
    \le
    \frac{4H_s}{b}F(z_t)
    +\frac{2H_sL}{b}\alpha_t^2\|v_t-z_t\|^2.
    \label{eq:acsa-query-variance}
\end{equation}

\medskip
\noindent
\emph{One-step potential inequality.}
The two updates give
$z_{t+1}=(1-\alpha_t)z_t+\alpha_t(v_t-\eta_t\bar g_t)
=y_t-\alpha_t\eta_t\bar g_t$. Lemma~\ref{lem:smooth}, Part~1, together
with $\mathbb E_t\bar g_t=\nabla F(y_t)$ and
$\mathbb E_t\|\bar g_t\|^2=\|\nabla F(y_t)\|^2+\sigma_t^2$, implies
\begin{equation}
\label{eq:helper_b2_0}
    \mathbb E_t F(z_{t+1})
    \le
    F(y_t)-\alpha_t\eta_t\|\nabla F(y_t)\|^2
    +\frac{L\alpha_t^2\eta_t^2}{2}
       \bigl(\|\nabla F(y_t)\|^2+\sigma_t^2\bigr).
\end{equation}
Convexity at $y_t$, applied to $z_t$ and $x^*$
with weights $1-\alpha_t$ and $\alpha_t$, together with $F(x^*)=0$ and
$(1-\alpha_t)z_t+\alpha_tx^*-y_t=\alpha_t(x^*-v_t)$, gives
\begin{equation}
\label{eq:helper_b2_1}
    F(y_t)
    \le
    (1-\alpha_t)F(z_t)
    +\alpha_t\langle \nabla F(y_t),v_t-x^*\rangle .
\end{equation}
The update of the auxiliary sequence satisfies the exact identity
\begin{equation}
    \frac12\mathbb E_t\|v_{t+1}-x^*\|^2
    =
    \frac12\|v_t-x^*\|^2
    -\eta_t\langle \nabla F(y_t),v_t-x^*\rangle
    +\frac{\eta_t^2}{2}
       \bigl(\|\nabla F(y_t)\|^2+\sigma_t^2\bigr).
    \label{eq:helper_b2_2}
\end{equation}
Substitute \eqref{eq:helper_b2_1} into \eqref{eq:helper_b2_0}, multiply
the result by $A_t=\eta_t/\alpha_t$, and add \eqref{eq:helper_b2_2}.
The terms $\pm\eta_t\langle\nabla F(y_t),v_t-x^*\rangle$ cancel, and
$A_t\alpha_t\eta_t=\eta_t^2$, $A_tL\alpha_t^2\eta_t^2=\eta_t^2\,
L\alpha_t\eta_t$; this gives
\begin{align}
    \mathbb E_t\left[
        A_tF(z_{t+1})
        +\frac12\|v_{t+1}-x^*\|^2
    \right]
    \le{}&
    (A_t-\eta_t)F(z_t)
    +\frac12\|v_t-x^*\|^2\notag\\
    &-
    \frac{\eta_t^2}{2}
       (1-L\alpha_t\eta_t)\|\nabla F(y_t)\|^2
    +\frac{\eta_t^2}{2}
       (1+L\alpha_t\eta_t)\sigma_t^2.
    \label{eq:acsa-one-step}
\end{align}
Since
\[
    \alpha_t\eta_t
    =6\gamma\frac{t+1}{t+6}
    \le6\gamma,
    \qquad
    L\gamma\le\frac1{48},
\]
we have $L\alpha_t\eta_t\le1/8$.
Thus the coefficient of $\|\nabla F(y_t)\|^2$ is nonpositive,
and the coefficient of $\sigma_t^2$ is at most
$\eta_t^2$. Substituting
\eqref{eq:acsa-query-variance} into
\eqref{eq:acsa-one-step} yields
\begin{align}
    \mathbb E_t\left[
        A_tF(z_{t+1})
        +\frac12\|v_{t+1}-x^*\|^2
    \right]
    \le{}&
    \left(
        A_t-\eta_t+\frac{4H_s\eta_t^2}{b}
    \right)F(z_t)
    +\frac12\|v_t-x^*\|^2\notag\\
    &+
    \frac{2H_sL}{b}
       (\alpha_t\eta_t)^2\|v_t-z_t\|^2.
    \label{eq:acsa-reduced-step}
\end{align}

For $1\le t\le q-1$, direct calculation gives
\[
    A_{t-1}-(A_t-\eta_t)=\frac23\gamma t.
\]
Moreover, $(t+1)^2\le4t^2\le4tq$ and
$\gamma\le b/(48H_s q)$ imply
\[
    \frac{4H_s\eta_t^2}{b}
    =\frac{4H_s\gamma^2(t+1)^2}{b}
    \le\frac{16H_s\gamma^2tq}{b}
    \le\frac13\gamma t.
\]
Therefore
\begin{equation}
    A_t-\eta_t+\frac{4H_s\eta_t^2}{b}
    \le A_{t-1},
    \qquad 1\le t\le q-1.
    \label{eq:acsa-coefficient}
\end{equation}

\medskip
\noindent
\emph{Controlling the auxiliary states.}
For $t\ge1$, define
\[
    \mathcal E_t
    =
    \mathbb E\left[
        A_{t-1}F(z_t)
        +\frac12\|v_t-x^*\|^2
    \right],
    \qquad
    M_t
    =
    \max\left\{
        R^2,\;
        2\max_{1\le s\le t}\mathcal E_s
    \right\}.
\]
Since $F\ge0$, we have $\mathbb E\|v_s-x^*\|^2\le2\mathcal E_s\le M_t$
for $1\le s\le t$. Because $0<\alpha_s\le1$, the recurrence for $z_s$
expresses $z_t$ as a convex combination, with deterministic weights, of
$x_0,v_1,\ldots,v_t$. Jensen's inequality therefore gives
$\mathbb E\|z_t-x^*\|^2\le M_t$, and consequently
\[
    \mathbb E\|v_t-z_t\|^2
    \le
    2\mathbb E\|v_t-x^*\|^2
    +2\mathbb E\|z_t-x^*\|^2
    \le4M_t.
\]
Taking expectations in \eqref{eq:acsa-reduced-step} and using
\eqref{eq:acsa-coefficient} and $\alpha_t\eta_t\le6\gamma$, we obtain
\[
    \mathcal E_{t+1}
    \le
    \mathcal E_t
    +\frac{288H_sL\gamma^2}{b}M_t,
    \qquad 1\le t\le q-1.
\]
The two upper bounds defining $\gamma$ imply
\[
    \gamma^2\le\frac1{48L}\cdot\frac{b}{48H_sq},
    \qquad\text{hence}\qquad
    \frac{288H_sL\gamma^2}{b}\le\frac1{8q}.
\]
It follows that $2\mathcal E_{t+1}\le M_t+\frac{M_t}{4q}$, that is,
\[
    M_{t+1}\le\Bigl(1+\frac1{4q}\Bigr)M_t,
    \qquad 1\le t\le q-1.
\]

For the initial step, $\alpha_0=1$, $A_0=\eta_0$, and
$z_0=v_0=y_0=x_0$. By \eqref{up:batch-variance} and
Lemma~\ref{up:moments}, Part~1,
\[
    \sigma_0^2
    \le\frac{2H_s}{b}F(x_0)
    \le\frac{H_sL}{b}R^2.
\]
Using \eqref{eq:acsa-one-step} at $t=0$ gives
\[
    \mathcal E_1
    \le
    \left(\frac12+\frac{H_sL\gamma^2}{b}\right)R^2
    \le\left(\frac12+\frac1{8q}\right)R^2,
    \qquad\text{hence}\qquad
    M_1\le\Bigl(1+\frac1{4q}\Bigr)R^2 .
\]
Consequently, including the case $q=1$,
\[
    M_q
    \le\Bigl(1+\frac1{4q}\Bigr)^qR^2
    \le e^{1/4}R^2
    <2R^2,
\]
and therefore $\mathcal E_q<R^2$.

\medskip
\noindent
\emph{Final bound.}
By nonnegativity of the squared-distance term,
\[
    A_{q-1}\,\mathbb E F(z_q)
    \le\mathcal E_q\le R^2.
\]
Finally,
\[
    A_{q-1}
    =\frac{\gamma q(q+5)}6
    \ge\frac{\gamma q^2}{6},
\]
and hence
\[
    \mathbb E F(z_q)
    \le
    \frac{6R^2}{\gamma q^2}
    =
    288R^2
    \max\left\{
        \frac{L}{q^2},
        \frac{H_s}{bq}
    \right\}.
\]
\end{proof}

\begin{corollary}[Upper bound for general convex populations]
\label{cor:selector}
There is an algorithm in $\mathrm{IC}(m,K,N)$ such that, for every
instance of $\mathcal W_d$,
\begin{align}
    \mathbb E F(\widehat x)
    &\le R^2\min\left\{
          \frac L2,\;\frac{2H_s}{KN},\;
          288\max\left\{\frac{L}{N^2},\frac{H_s}{mKN}\right\}
       \right\}\notag\\
    &\le288R^2\min\left\{L,\;\frac{H_s}{KN},\;
          \frac{L}{N^2}+\frac{H_s}{mKN}\right\}.
    \label{up:selector-rate}
\end{align}
\end{corollary}

\begin{proof}
The algorithm selects, before the run, the branch with the smallest of
the following three guarantees; the selection uses only the public
parameters.

\emph{Zero queries.} Returning $x_0$ gives $F(x_0)\le\frac L2R^2$ by
Lemma~\ref{up:moments}, Part~1.

\emph{Serial branch.} One worker runs the recursion of
Lemma~\ref{lem:serial} with $s=KN$, making $K$ queries in each round
and keeping its state between rounds; the other workers are idle. Its
risk is at most $2H_sR^2/(KN)$.

\emph{Common-query branch.} Run Algorithm~\ref{alg:fresh-batch-acsa}
with $q=N$ and $b=mK$. At the start of round $t+1$ the point $y_t$ is
broadcast, every worker makes its $K$ queries at $y_t$, and at the end of
the round the responses are averaged and the updates of
Algorithm~\ref{alg:fresh-batch-acsa} are performed. Thus every iteration
is one round, and Theorem~\ref{thm:fresh-batch-acsa} bounds the risk by
$288R^2\max\{L/N^2,H_s/(mKN)\}$.

The second inequality in \eqref{up:selector-rate} follows by increasing
all coefficients to $288$ and using $\max\{a,b\}\le a+b$.
\end{proof}

\section{Upper bound with sample reuse}
\label{app:appendix_C}

Let Assumption~\ref{ass:pool} hold, and fix an instance of
$\mathcal S_d$. Replacing $f_\xi$ by
$f_\xi-f_\xi(x^*)$ does not change any gradient, so we assume that
$f_\xi\ge0$ and $f_\xi(x^*)=0$ for almost every $\xi$, and
$F=\mathbb E f_\xi$. Let
\[
    \mathcal C=\{x\in\R^d:\|x-x_0\|\le R\},
\]
which is known to the algorithm and contains $x^*$, and let
$\Pi_{\mathcal C}$ be the Euclidean projection onto $\mathcal C$. A
\emph{fresh batch} of size $n$ is a set of $n$ labels that have not been
queried before; since labels carry no information, the corresponding
samples are independent of the history and of each other.

\subsection{Acceleration with biased gradients}
\label{app:C-acceleration}

The following projected variant of
Algorithm~\ref{alg:fresh-batch-acsa} is the accelerated scheme of
\citet{woodworth2021even}; it is applied below both to the sample
functions and to their empirical envelopes.

\begin{lemma}[Projected acceleration with biased gradients]
\label{up:accelerated}
Let $q_\zeta$ be jointly measurable nonnegative convex functions of
$x\in\R^d$, indexed by a random sample $\zeta$, with $H'$-Lipschitz
gradients and a common zero $x^*\in\mathcal C$, and let
$Q=\mathbb E q_\zeta$ have an $L'$-Lipschitz gradient, where $L',H'>0$.
Fix integers $b,T\ge1$ and put
\[
    \gamma=\min\left\{\frac1{12L'},\frac b{24H'(T+1)}\right\},
    \qquad
    \alpha_t=\frac6{t+6},
    \qquad
    \eta_t=\gamma(t+1).
\]
Let $z_0=v_0=x_0$ and, for $t=0,\ldots,T-1$,
\begin{align*}
    y_t&=(1-\alpha_t)z_t+\alpha_tv_t,
    &
    v_{t+1}&=\Pi_{\mathcal C}(v_t-\eta_t\widetilde g_t),
    &
    z_{t+1}&=(1-\alpha_t)z_t+\alpha_tv_{t+1},
\end{align*}
where $\widetilde g_t=g_t+e_t$, $g_t$ is the average of
$\nabla q_\zeta(y_t)$ over $b$ samples drawn independently of the
history $\mathcal F_t$ before them, and the error $e_t$ may depend on
these samples and on additional randomness and satisfies
\begin{equation}
    \mathbb E\bigl[\|e_t\|\mid\mathcal F_t\bigr]\le\varepsilon .
    \label{up:bias-assumption}
\end{equation}
Then all $y_t,v_t,z_t$ belong to $\mathcal C$, and
\begin{equation}
    \mathbb E Q(z_T)
    \le\frac{36L'R^2}{T^2}+\frac{288H'R^2}{bT}+12R\varepsilon .
    \label{up:accel-rate}
\end{equation}
\end{lemma}

\begin{proof}
The iterates belong to $\mathcal C$ by induction and convexity of
$\mathcal C$; in particular, $\|v_t-z_t\|\le2R$ and
$\|v_{t+1}-x^*\|\le2R$. On $\mathcal C$ every $\nabla q_\zeta$ is bounded
by $2H'R$ (Lemma~\ref{lem:smooth}, Part~4), so all quantities below are
integrable, and, as in the proof of Lemma~\ref{lem:relations}, the bound
$\|\nabla q_\zeta(x)\|\le H'\|x-x^*\|$ justifies
$\nabla Q=\mathbb E\nabla q_\zeta$. Write $\mathbb E_t[\cdot]=\mathbb E[\cdot\mid\mathcal F_t]$,
$A_t=\eta_t/\alpha_t=\frac{\gamma(t+1)(t+6)}6$, and
$\omega_t=g_t-\nabla Q(y_t)$, so that $\mathbb E_t\omega_t=0$.

\emph{Variance.}
For fixed $x,y$, Lemma~\ref{lem:smooth}, Part~3, applied to each
$q_\zeta$, and the smoothness of $Q$ give
\[
    \mathbb E\|\nabla q_\zeta(x)-\nabla q_\zeta(y)\|^2
    \le H'\langle\nabla Q(x)-\nabla Q(y),x-y\rangle
    \le H'L'\|x-y\|^2 ,
\]
and Lemma~\ref{lem:smooth}, Part~2, gives
$\mathbb E\|\nabla q_\zeta(x)\|^2\le2H'Q(x)$. As in the proof of
Lemma~\ref{up:moments}, Part~2, averaging $b$ independent samples divides
the variance by $b$, and the variance is at most the second moment.
Using $\|y_t-z_t\|=\alpha_t\|v_t-z_t\|\le2R\alpha_t$, we obtain
\begin{equation}
    \mathbb E_t\|\omega_t\|^2
    \le\frac2b\,\mathbb E_t\|\nabla q_\zeta(y_t)-\nabla q_\zeta(z_t)\|^2
       +\frac2b\,\mathbb E_t\|\nabla q_\zeta(z_t)\|^2
    \le\frac{8H'L'R^2\alpha_t^2}{b}+\frac{4H'}{b}Q(z_t).
    \label{up:accel-variance}
\end{equation}
At $t=0$ we use instead $y_0=x_0$ and Lemma~\ref{lem:smooth}, Part~4,
for $Q$:
$\mathbb E\|\omega_0\|^2\le\frac{2H'}bQ(x_0)\le\frac{H'L'R^2}b$.

\emph{One step.}
As in \eqref{eq:helper_b2_1}, convexity gives
$Q(y_t)\le(1-\alpha_t)Q(z_t)+\alpha_t\langle\nabla Q(y_t),v_t-x^*\rangle$.
Since $z_{t+1}-y_t=\alpha_t(v_{t+1}-v_t)$, Lemma~\ref{lem:smooth},
Part~1, gives
\[
    Q(z_{t+1})
    \le Q(y_t)+\alpha_t\langle\nabla Q(y_t),v_{t+1}-v_t\rangle
       +\frac{L'\alpha_t^2}2\|v_{t+1}-v_t\|^2 .
\]
Combining the two inequalities and multiplying by $A_t$,
\[
    A_tQ(z_{t+1})
    \le(A_t-\eta_t)Q(z_t)
       +\eta_t\langle\nabla Q(y_t),v_{t+1}-x^*\rangle
       +\frac{L'\alpha_t\eta_t}2\|v_{t+1}-v_t\|^2 .
\]
The optimality condition of the projection,
$\langle v_t-\eta_t\widetilde g_t-v_{t+1},x^*-v_{t+1}\rangle\le0$, gives
\[
    \eta_t\langle\widetilde g_t,v_{t+1}-x^*\rangle
    \le\frac12\Bigl(\|v_t-x^*\|^2-\|v_{t+1}-x^*\|^2
        -\|v_{t+1}-v_t\|^2\Bigr).
\]
Since $\nabla Q(y_t)=\widetilde g_t-\omega_t-e_t$, the last two displays
yield
\begin{align}
    A_tQ(z_{t+1})+\frac12\|v_{t+1}-x^*\|^2
    \le{}&(A_t-\eta_t)Q(z_t)+\frac12\|v_t-x^*\|^2
       -\frac{1-L'\alpha_t\eta_t}2\|v_{t+1}-v_t\|^2\notag\\
    &+\eta_t\langle\omega_t,x^*-v_{t+1}\rangle
     +\eta_t\langle e_t,x^*-v_{t+1}\rangle .
    \label{up:one-step-pathwise}
\end{align}
We have $\alpha_t\eta_t\le6\gamma\le\frac1{2L'}$, so the coefficient of
$\|v_{t+1}-v_t\|^2$ is at most $-\frac14$. Split the noise term at $v_t$:
$\eta_t\langle\omega_t,x^*-v_t\rangle$ has zero conditional expectation,
and
$\eta_t\langle\omega_t,v_t-v_{t+1}\rangle\le\eta_t^2\|\omega_t\|^2
+\frac14\|v_{t+1}-v_t\|^2$. The bias term is at most
$2R\eta_t\|e_t\|$. Taking expectations and using
\eqref{up:bias-assumption},
\begin{equation}
    \mathbb E\Bigl[A_tQ(z_{t+1})+\frac12\|v_{t+1}-x^*\|^2\Bigr]
    \le\mathbb E\Bigl[(A_t-\eta_t)Q(z_t)+\frac12\|v_t-x^*\|^2\Bigr]
       +\eta_t^2\,\mathbb E\|\omega_t\|^2+2R\eta_t\varepsilon .
    \label{up:one-step-expected}
\end{equation}

\emph{Summation.}
As in the proof of Theorem~\ref{thm:fresh-batch-acsa},
$A_{t-1}-(A_t-\eta_t)=\frac23\gamma t$ for $t\ge1$, and $A_0=\eta_0$.
By \eqref{up:accel-variance}, the coefficient of $\mathbb E Q(z_t)$
contributed by $\eta_t^2\mathbb E\|\omega_t\|^2$ is
$\frac{4H'\gamma^2(t+1)^2}{b}\le\frac{\gamma(t+1)^2}{6(T+1)}
\le\frac{\gamma(t+1)}6\le\frac{\gamma t}3$ for $1\le t\le T-1$, where we
used $\gamma\le\frac b{24H'(T+1)}$ and $t+1\le2t$. Hence these terms are
absorbed when \eqref{up:one-step-expected} is summed over
$t=0,\ldots,T-1$, and, using the bound for $t=0$ and
$\|x_0-x^*\|\le R$,
\[
    A_{T-1}\,\mathbb E Q(z_T)
    \le\frac{R^2}2
       +\frac{8H'L'R^2}{b}\sum_{t=0}^{T-1}(\alpha_t\eta_t)^2
       +2R\varepsilon\sum_{t=0}^{T-1}\eta_t .
\]
Since $A_{T-1}\ge\frac{\gamma T^2}6$, $\alpha_t\eta_t\le6\gamma$, and
$\sum_{t<T}\eta_t=\frac{\gamma T(T+1)}2$,
\[
    \mathbb E Q(z_T)
    \le\frac{3R^2}{\gamma T^2}+\frac{1728H'L'R^2\gamma}{bT}
       +\frac{6R\varepsilon(T+1)}{T} .
\]
Finally, $\gamma^{-1}\le12L'+\frac{24H'(T+1)}b$, $\gamma\le\frac1{12L'}$,
and $T+1\le2T$ give \eqref{up:accel-rate}.
\end{proof}

With $q_\zeta=f_\xi$, $Q=F$, $L'=L$, $H'=H_s$, and $e_t=0$,
Lemma~\ref{up:accelerated} gives
\begin{equation}
    \mathbb E F(z_T)\le\frac{36LR^2}{T^2}+\frac{288H_sR^2}{bT}.
    \label{up:fresh-rate}
\end{equation}
For $T=N$ and $b\le mK$, the $b$ queries of every iteration are made at
the common point $y_t$ and distributed among the workers, at most $K$
per worker; every iteration is one round.

\subsection{Empirical envelopes}
\label{up:envelopes}

For a fresh batch with samples $\mathsf S=(\xi_1,\ldots,\xi_n)$ and
$\tau>0$, define
\begin{gather*}
    f_{\mathsf S}=\frac1n\sum_{i=1}^nf_{\xi_i},
    \qquad
    h_{\mathsf S,y}(u)=f_{\mathsf S}(u)+\frac\tau2\|u-y\|^2,
    \qquad
    p_{\mathsf S}(y)=\arg\min_uh_{\mathsf S,y}(u),\\
    \varphi_{\mathsf S}(y)=\min_uh_{\mathsf S,y}(u),
    \qquad
    \bar\varphi(y)=\mathbb E_{\mathsf S}\varphi_{\mathsf S}(y).
\end{gather*}
These functions are used only in the analysis; the algorithm has no
proximal oracle.

\begin{lemma}[Envelopes]
\label{up:envelope-properties}
For every batch, $\varphi_{\mathsf S}$ is nonnegative and convex, has the
$\tau$-Lipschitz gradient
\begin{equation}
    \nabla\varphi_{\mathsf S}(y)=\tau\bigl(y-p_{\mathsf S}(y)\bigr),
    \label{up:envelope-gradient}
\end{equation}
and satisfies $\varphi_{\mathsf S}(x^*)=0$,
$\varphi_{\mathsf S}(y)\le\frac\tau2\|y-x^*\|^2$, and
$\|p_{\mathsf S}(y)-x^*\|\le\|y-x^*\|$. The map
$(\mathsf S,y)\mapsto p_{\mathsf S}(y)$ is jointly measurable, and
$\bar\varphi$ is nonnegative and convex, has a $\tau$-Lipschitz gradient
$\nabla\bar\varphi=\mathbb E\nabla\varphi_{\mathsf S}$, and vanishes at
$x^*$. For $y\in\mathcal C$,
$h_{\mathsf S,y}(y)-\varphi_{\mathsf S}(y)\le2H_sR^2$.
\end{lemma}

\begin{proof}
The function $h_{\mathsf S,y}$ is continuous and $\tau$-strongly convex,
so $p_{\mathsf S}(y)$ is well defined. Write $p=p_{\mathsf S}$. The
optimality conditions at the centers $y$ and $y'$,
$\nabla f_{\mathsf S}(p(y))=\tau(y-p(y))$, and monotonicity of
$\nabla f_{\mathsf S}$ give
$\langle p(y)-p(y'),y-y'\rangle\ge\|p(y)-p(y')\|^2$. Hence $p$ is
nonexpansive, and so is $\mathrm I-p$: with $\Delta=y-y'$ and
$\Delta_p=p(y)-p(y')$,
$\|\Delta-\Delta_p\|^2=\|\Delta\|^2-2\langle\Delta,\Delta_p\rangle
+\|\Delta_p\|^2\le\|\Delta\|^2-\|\Delta_p\|^2$.

Evaluating $h_{\mathsf S,y'}$ at $p(y)$ and $h_{\mathsf S,y}$ at $p(y')$
gives
\begin{align*}
    \varphi_{\mathsf S}(y')-\varphi_{\mathsf S}(y)
    &\le\tau\langle y-p(y),y'-y\rangle+\frac\tau2\|y'-y\|^2,\\
    \varphi_{\mathsf S}(y')-\varphi_{\mathsf S}(y)
    &\ge\tau\langle y'-p(y'),y'-y\rangle-\frac\tau2\|y'-y\|^2 .
\end{align*}
Since $\mathrm I-p$ is nonexpansive,
$|\tau\langle(y'-p(y'))-(y-p(y)),y'-y\rangle|\le\tau\|y'-y\|^2$, and
therefore
$|\varphi_{\mathsf S}(y')-\varphi_{\mathsf S}(y)
-\tau\langle y-p(y),y'-y\rangle|\le\frac{3\tau}2\|y'-y\|^2$. This proves
differentiability and \eqref{up:envelope-gradient}, and nonexpansiveness
of $\mathrm I-p$ gives the $\tau$-Lipschitz property. Convexity holds
because $\varphi_{\mathsf S}$ is a partial minimum of a jointly convex
function, and nonnegativity because $f_{\mathsf S}\ge0$. Taking $u=x^*$
in the minimum gives $\varphi_{\mathsf S}(y)\le\frac\tau2\|y-x^*\|^2$
and $\varphi_{\mathsf S}(x^*)=0$; since $x^*$ is then the unique minimizer
of $h_{\mathsf S,x^*}$, $p(x^*)=x^*$, and nonexpansiveness gives
$\|p(y)-x^*\|\le\|y-x^*\|$.

For measurability, enumerate a countable dense set $\{u_k\}\subseteq\R^d$.
The minimum of the continuous function $h_{\mathsf S,y}$ equals its
infimum over $\{u_k\}$ and is measurable in $(\mathsf S,y)$. The first
$u_k$ whose value is within $1/j$ of the minimum is a measurable function
of $(\mathsf S,y)$ and, by strong convexity, lies within
$\sqrt{2/(\tau j)}$ of $p(y)$; letting $j\to\infty$ shows that $p$ is
jointly measurable. The bound
$\|\nabla\varphi_{\mathsf S}(y)\|=\|\nabla\varphi_{\mathsf S}(y)
-\nabla\varphi_{\mathsf S}(x^*)\|\le\tau\|y-x^*\|$ justifies
differentiation under the expectation, as in the proof of
Lemma~\ref{lem:relations}, and the properties of $\bar\varphi$ follow.
Finally, for $y\in\mathcal C$, $\|y-x^*\|\le2R$ and
$h_{\mathsf S,y}(y)-\varphi_{\mathsf S}(y)\le f_{\mathsf S}(y)
\le\frac{H_s}2\|y-x^*\|^2\le2H_sR^2$ by \eqref{up:sample-bounds}.
\end{proof}

\begin{lemma}[From the envelope to the population]
\label{up:prox-transfer}
Let $\tau=8H_s/n$. For every fixed $y$,
\begin{equation}
    \mathbb E_{\mathsf S}F\bigl(p_{\mathsf S}(y)\bigr)\le2\bar\varphi(y).
    \label{up:exact-prox-transfer}
\end{equation}
If $y$ is a random point independent of the batch and $\widetilde p$
satisfies
$\mathbb E[h_{\mathsf S,y}(\widetilde p)-\varphi_{\mathsf S}(y)\mid y]
\le\delta$, then
\begin{equation}
    \mathbb E F(\widetilde p)\le4\,\mathbb E\bar\varphi(y)+\frac{2H_s\delta}\tau .
    \label{up:approx-prox-transfer}
\end{equation}
\end{lemma}

\begin{proof}
Fix $y$ and $i\le n$, let $\xi_i'$ be an independent copy of $\xi_i$, and
let $\mathsf S'$ be the batch with $\xi_i$ replaced by $\xi_i'$. Put
$u=p_{\mathsf S}(y)$, $u'=p_{\mathsf S'}(y)$, and $\Delta=\|u-u'\|$.
Strong convexity of $h_{\mathsf S,y}$ and $h_{\mathsf S',y}$ at their
minimizers gives
$h_{\mathsf S,y}(u')-h_{\mathsf S,y}(u)\ge\frac\tau2\Delta^2$ and
$h_{\mathsf S',y}(u)-h_{\mathsf S',y}(u')\ge\frac\tau2\Delta^2$. Adding
these inequalities and using
$h_{\mathsf S,y}-h_{\mathsf S',y}=\frac1n(f_{\xi_i}-f_{\xi_i'})$ and
convexity of $f_{\xi_i}$ and $f_{\xi_i'}$,
\begin{equation}
    \tau\Delta^2
    \le\frac1n\Bigl(f_{\xi_i}(u')-f_{\xi_i}(u)+f_{\xi_i'}(u)-f_{\xi_i'}(u')\Bigr)
    \le\frac\Delta n\Bigl(\|\nabla f_{\xi_i}(u')\|+\|\nabla f_{\xi_i'}(u)\|\Bigr).
    \label{up:replace-one}
\end{equation}
Let $R_y=\mathbb E F(u)$, which is finite by
Lemma~\ref{up:envelope-properties}. The point $u'$ is independent of
$\xi_i$, the point $u$ is independent of $\xi_i'$, and $u,u'$ have the
same law. Hence, by \eqref{up:sample-growth},
$\mathbb E\|\nabla f_{\xi_i}(u')\|^2\le2H_sR_y$ and
$\mathbb E\|\nabla f_{\xi_i'}(u)\|^2\le2H_sR_y$, and squaring
\eqref{up:replace-one} gives
$\mathbb E\Delta^2\le\frac{8H_sR_y}{\tau^2n^2}$. Convexity and the
Cauchy--Schwarz inequality then give
\[
    \mathbb E\bigl[f_{\xi_i}(u')-f_{\xi_i}(u)\bigr]
    \le\sqrt{\mathbb E\|\nabla f_{\xi_i}(u')\|^2}\sqrt{\mathbb E\Delta^2}
    \le\frac{4H_s}{\tau n}R_y=\frac{R_y}2 .
\]
Since $u'$ is independent of $\xi_i$ and has the law of $u$,
$\mathbb E f_{\xi_i}(u')=R_y$. Averaging over $i$ gives
$R_y-\mathbb E f_{\mathsf S}(u)\le\frac{R_y}2$, so
$R_y\le2\mathbb E f_{\mathsf S}(u)\le2\mathbb E\varphi_{\mathsf S}(y)
=2\bar\varphi(y)$.

For the second claim, strong convexity gives
$\mathbb E[\|\widetilde p-p_{\mathsf S}(y)\|^2\mid y]\le2\delta/\tau$.
Lemma~\ref{lem:smooth}, Part~2, applied to $F$, which is nonnegative and
$H_s$-smooth, gives
$F(\widetilde p)\le2F(p_{\mathsf S}(y))+H_s\|\widetilde p-p_{\mathsf S}(y)\|^2$.
Taking expectations conditional on $y$ and using
\eqref{up:exact-prox-transfer} proves \eqref{up:approx-prox-transfer}.
\end{proof}

\subsection{The local solver}
\label{up:svrg}

The approximate proximal points are computed by the variance-reduced
method of \citet{johnson2013accelerating}; it uses only gradients.

\begin{lemma}[One epoch of SVRG]
\label{up:svrg-epoch}
Let $h=\frac1n\sum_{i=1}^nh_i$, where every $h_i$ is convex with an
$M$-Lipschitz gradient and $h$ is $\mu$-strongly convex with minimizer
$u$. Given a snapshot $\bar x$, store all $\nabla h_i(\bar x)$ and their
average $\nabla h(\bar x)$. Set $w_0=\bar x$ and, for $k=0,\ldots,I-1$,
select $i_k$ uniformly from $\{1,\ldots,n\}$, independently, and set
\[
    V_k=\nabla h_{i_k}(w_k)-\nabla h_{i_k}(\bar x)+\nabla h(\bar x),
    \qquad
    w_{k+1}=w_k-\beta V_k .
\]
Return $w_\kappa$ with $\kappa$ uniform on $\{0,\ldots,I-1\}$, independent
of everything else. If $0<2M\beta<1$, then
\begin{equation}
    \mathbb E\bigl[h(w_\kappa)-h(u)\mid\bar x\bigr]
    \le\left[\frac1{\mu\beta(1-2M\beta)I}
       +\frac{2M\beta}{1-2M\beta}\right]\bigl(h(\bar x)-h(u)\bigr).
    \label{up:svrg-contraction}
\end{equation}
\end{lemma}

\begin{proof}
By Lemma~\ref{lem:smooth}, Part~3, applied to each $h_i$ at $u$, and
$\nabla h(u)=0$,
\[
    \frac1n\sum_{i=1}^n\|\nabla h_i(x)-\nabla h_i(u)\|^2
    \le2M\bigl(h(x)-h(u)\bigr)
    \qquad\text{for every }x .
\]
Fix the inner history up to $w_k$ and write
$V_k=a_{i_k}-(c_{i_k}-\bar c)$ with
$a_i=\nabla h_i(w_k)-\nabla h_i(u)$, $c_i=\nabla h_i(\bar x)-\nabla h_i(u)$,
and $\bar c=\frac1n\sum_ic_i=\nabla h(\bar x)$. Then
$\mathbb E[V_k\mid w_k]=\nabla h(w_k)$, and since the variance of
$c_{i_k}$ is at most its second moment,
\[
    \mathbb E\bigl[\|V_k\|^2\mid w_k\bigr]
    \le2\,\mathbb E\|a_{i_k}\|^2+2\,\mathbb E\|c_{i_k}\|^2
    \le4M\bigl(h(w_k)-h(u)\bigr)+4M\bigl(h(\bar x)-h(u)\bigr).
\]
Expanding the square and using
$\langle\nabla h(w_k),w_k-u\rangle\ge h(w_k)-h(u)$,
\[
    \mathbb E\bigl[\|w_{k+1}-u\|^2\mid w_k\bigr]
    \le\|w_k-u\|^2-2\beta(1-2M\beta)\bigl(h(w_k)-h(u)\bigr)
       +4M\beta^2\bigl(h(\bar x)-h(u)\bigr).
\]
Summing over $k=0,\ldots,I-1$, discarding the final squared distance, and
using $\|\bar x-u\|^2\le\frac2\mu(h(\bar x)-h(u))$ by strong convexity,
\[
    2\beta(1-2M\beta)\sum_{k=0}^{I-1}\mathbb E\bigl[h(w_k)-h(u)\mid\bar x\bigr]
    \le\Bigl(\frac2\mu+4M\beta^2I\Bigr)\bigl(h(\bar x)-h(u)\bigr).
\]
Dividing by $2\beta(1-2M\beta)I$ proves \eqref{up:svrg-contraction},
because $w_\kappa$ is uniform over $w_0,\ldots,w_{I-1}$.
\end{proof}

\paragraph{The local routine.}
Fix integers $n\ge1$, $T\ge0$, and $j\ge1$. Given a point
$y\in\mathcal C$ and a fresh batch with samples $\mathsf S$, the
procedure \textsc{LocalProx} of Algorithm~\ref{alg:envelope} runs $j$
epochs of Lemma~\ref{up:svrg-epoch} on
$h=h_{\mathsf S,y}$ with
$h_i(x)=f_{\xi_i}(x)+\frac\tau2\|x-y\|^2$, the first snapshot $y$, every
later snapshot equal to the output of the preceding epoch, and
\begin{equation}
    \tau=\frac{8H_s}{n},
    \qquad
    \beta=\frac1{10(H_s+\tau)},
    \qquad
    I=\left\lceil\frac{100(H_s+\tau)}\tau\right\rceil .
    \label{up:local-parameters}
\end{equation}
It returns the output $\widetilde p$ of the last epoch and
$\widetilde g=\tau(y-\widetilde p)$. The quadratic part of the gradient is
computed locally, so
$V_k=\nabla f_{\xi_{i_k}}(w_k)-\nabla f_{\xi_{i_k}}(\bar x)
+\frac1n\sum_i\nabla f_{\xi_i}(\bar x)+\tau(w_k-y)$, and an epoch makes
$n$ queries at the snapshot and $I$ further queries, all with labels of
the batch.

\begin{lemma}[Local routine]
\label{up:local}
If
\[
    j\ge\left\lceil\log_2\bigl(625n(n+8)(T+1)^4\bigr)\right\rceil,
\]
then the local routine makes $j(n+I)\le114jn$ queries with $n$ distinct
labels, and
\begin{gather}
    \mathbb E\bigl[h_{\mathsf S,y}(\widetilde p)-\varphi_{\mathsf S}(y)
      \mid\mathsf S,y\bigr]\le\delta
    :=\frac{\tau^2R^2}{20000(H_s+\tau)(T+1)^4},
    \label{up:local-gap}\\
    \mathbb E\bigl[\|\widetilde g-\nabla\varphi_{\mathsf S}(y)\|\mid y\bigr]
    \le\frac{\tau R}{100(T+1)^2},
    \qquad
    \frac{2H_s\delta}\tau\le\frac{\tau R^2}{10^4(T+1)^4}.
    \label{up:local-gradient-error}
\end{gather}
\end{lemma}

\begin{proof}
Each $h_i$ is convex with an $(H_s+\tau)$-Lipschitz gradient, and $h$ is
$\tau$-strongly convex. With \eqref{up:local-parameters},
$2(H_s+\tau)\beta=\frac15$, and the factor in
\eqref{up:svrg-contraction} is at most
$\frac{12.5(H_s+\tau)}{\tau I}+\frac14\le\frac18+\frac14<\frac12$. By
Lemma~\ref{up:envelope-properties}, the gap at the first snapshot is at
most $2H_sR^2$, so iterating \eqref{up:svrg-contraction} over the epochs
gives the bound $2^{-j}\cdot2H_sR^2$ in \eqref{up:local-gap}. Since
$H_s/\tau=n/8$ and $(H_s+\tau)/\tau=(n+8)/8$,
\[
    \frac{2H_sR^2}{\delta}=\frac{40000H_s(H_s+\tau)(T+1)^4}{\tau^2}
    =625n(n+8)(T+1)^4\le2^j ,
\]
which proves \eqref{up:local-gap}. By \eqref{up:envelope-gradient},
$\|\widetilde g-\nabla\varphi_{\mathsf S}(y)\|=\tau\|\widetilde p
-p_{\mathsf S}(y)\|$, and strong convexity and Jensen's inequality give
the bound $\tau\sqrt{2\delta/\tau}=\sqrt{2\tau\delta}\le\frac{\tau R}{100(T+1)^2}$,
using $\tau\le H_s+\tau$. The last inequality in
\eqref{up:local-gradient-error} follows from $H_s\le H_s+\tau$. Finally,
$n+I=n+100+\lceil12.5n\rceil\le114n$.
\end{proof}

All maps of the routine are measurable, and its iterates have finite
moments of all orders: by \eqref{up:sample-bounds}, the norms of the
inner iterates obey a finite linear recursion with deterministic
coefficients.

\subsection{The distributed envelope method}
\label{up:distributed}

Algorithm~\ref{alg:envelope} combines the method of
Lemma~\ref{up:accelerated}, applied to the empirical envelopes, with the
local routine. Only $m'\le m$ workers are active; this is used for small
pools in Theorem~\ref{up:final-upper}.

\begin{algorithm}[H]
\caption{Accelerated envelope method with local SVRG}
\label{alg:envelope}
\begin{algorithmic}[1]
\Require Initial point $x_0$, parameters $R,H_s>0$, integers
         $1\le m'\le m$ and $n,j,N\ge1$, and a pool of labels.
\State $\displaystyle\tau\gets\frac{8H_s}n$,
       $\displaystyle\beta\gets\frac1{10(H_s+\tau)}$,
       $\displaystyle I\gets\left\lceil\frac{100(H_s+\tau)}\tau\right\rceil$,
       $T\gets N-1$
\State $\displaystyle\gamma\gets\min\left\{\frac1{12\tau},
       \frac{m'}{24\tau(T+1)}\right\}$, $z_0\gets x_0$, $v_0\gets x_0$
\For{$t=0,\ldots,T-1$} \Comment{round $t+1$}
    \State $\displaystyle\alpha_t\gets\frac6{t+6}$,
           $\eta_t\gets\gamma(t+1)$,
           $y_t\gets(1-\alpha_t)z_t+\alpha_tv_t$, and broadcast $y_t$
    \State Each worker $i\le m'$ takes a fresh batch $\mathsf S_{i,t}$ of
           $n$ labels and sends
           $\widetilde g_{i,t}\gets\tau\bigl(y_t-\Call{LocalProx}{y_t,\mathsf S_{i,t}}\bigr)$
    \State $\displaystyle\widetilde g_t\gets\frac1{m'}\sum_{i=1}^{m'}\widetilde g_{i,t}$
    \State $v_{t+1}\gets\Pi_{\mathcal C}(v_t-\eta_t\widetilde g_t)$,
           $z_{t+1}\gets(1-\alpha_t)z_t+\alpha_tv_{t+1}$
\EndFor
\State In round $N$, one worker takes a fresh batch $\mathsf S$ and
       \Return $\widehat x\gets\Call{LocalProx}{z_T,\mathsf S}$
\Statex
\Procedure{LocalProx}{$y$, batch with labels $i_1,\ldots,i_n$}
    \State $\bar x\gets y$
    \For{$e=1,\ldots,j$} \Comment{one SVRG epoch}
        \State Query and store $\nabla f_{\xi_{i_l}}(\bar x)$, $l=1,\ldots,n$,
               and set $\displaystyle\bar g\gets\frac1n\sum_{l=1}^n\nabla f_{\xi_{i_l}}(\bar x)$
        \State $w_0\gets\bar x$
        \For{$k=0,\ldots,I-1$}
            \State Draw $l$ uniformly from $\{1,\ldots,n\}$ and query
                   $\nabla f_{\xi_{i_l}}(w_k)$
            \State $V_k\gets\nabla f_{\xi_{i_l}}(w_k)-\nabla f_{\xi_{i_l}}(\bar x)
                   +\bar g+\tau(w_k-y)$
            \State $w_{k+1}\gets w_k-\beta V_k$
        \EndFor
        \State $\bar x\gets w_\kappa$ with $\kappa$ uniform on $\{0,\ldots,I-1\}$
    \EndFor
    \State \Return $\bar x$
\EndProcedure
\end{algorithmic}
\end{algorithm}

A fresh batch consists of labels that have not been queried before. For
$N=1$ the loop is empty and the output is computed at $z_0=x_0$. Every
call of \textsc{LocalProx} makes $j(n+I)$ queries.

\begin{proposition}[Envelope method]
\label{up:envelope-method}
Let $1\le m'\le m$, $n\ge1$, and
$j\ge\lceil\log_2(625n(n+8)N^4)\rceil$ satisfy $114jn\le K$, and let
$\tau=8H_s/n$. Then Algorithm~\ref{alg:envelope} belongs to
$\mathrm{IC}(m,K,N)$, uses at most $m'Nn$ distinct labels, and satisfies
\begin{equation}
    \mathbb E F(\widehat x)
    \le\frac{150\tau R^2}{(N-1)^2}+\frac{1152\tau R^2}{m'(N-1)}
    \quad\text{if }N\ge2,
    \qquad
    \mathbb E F(\widehat x)\le3\tau R^2
    \quad\text{if }N=1 .
    \label{up:distributed-rate}
\end{equation}
\end{proposition}

\begin{proof}
Let $N\ge2$ and $T=N-1$. The loop of Algorithm~\ref{alg:envelope} is
the method of Lemma~\ref{up:accelerated} for the envelope samples $q_{\mathsf S}=\varphi_{\mathsf S}$ and their
mean $Q=\bar\varphi$, with $L'=H'=\tau$, $b=m'$, and $T$ iterations;
the hypotheses hold by Lemma~\ref{up:envelope-properties}. In
iteration $t$, every active worker runs the local routine on its own
fresh batch at the common point $y_t\in\mathcal C$ with horizon $T$, and
the average of the returned vectors is $\widetilde g_t$. The average of the
exact envelope gradients $\nabla\varphi_{\mathsf S}(y_t)$ over the
$m'$ batches is the average $g_t$ of Lemma~\ref{up:accelerated}, and by
the triangle inequality and \eqref{up:local-gradient-error} the error
$e_t=\widetilde g_t-g_t$ satisfies \eqref{up:bias-assumption} with
$\varepsilon=\frac{\tau R}{100(T+1)^2}$. Hence
\[
    \mathbb E\bar\varphi(z_T)
    \le\frac{36\tau R^2}{T^2}+\frac{288\tau R^2}{m'T}
       +\frac{12\tau R^2}{100(T+1)^2}
    \le\frac{37\tau R^2}{T^2}+\frac{288\tau R^2}{m'T}.
\]
In round $N$, the batch is fresh and hence independent of $z_T$, and
$\widehat x$ is the output $\widetilde p$ of the local routine at
$z_T\in\mathcal C$. Lemma~\ref{up:prox-transfer} and
\eqref{up:local-gradient-error} give
\[
    \mathbb E F(\widehat x)
    \le4\,\mathbb E\bar\varphi(z_T)+\frac{\tau R^2}{10^4(T+1)^4}
    \le\frac{150\tau R^2}{T^2}+\frac{1152\tau R^2}{m'T}.
\]
For $N=1$, the output is computed by the local routine at $x_0$ with
horizon $T=0$; since $\bar\varphi(x_0)\le\frac\tau2R^2$, Lemma~\ref{up:prox-transfer}
gives $\mathbb E F(\widehat x)\le2\tau R^2+\frac{\tau R^2}{10^4}\le3\tau R^2$.

In every round each active worker makes $j(n+I)\le114jn\le K$ queries,
the condition on $j$ in Lemma~\ref{up:local} holds because $T+1\le N$,
and no communication within a round is needed. The number of distinct
labels is at most $(m'(N-1)+1)n\le m'Nn$.
\end{proof}

\begin{proposition}[Integer parameters]
\label{up:integers}
Let $J$ be as in Theorem~\ref{thm:pool} and $1\le m'\le m$. There is an
algorithm in $\mathrm{IC}(m,K,N)$ that uses at most $m'KN$ distinct
labels and satisfies
\begin{equation}
    \mathbb E F(\widehat x)
    \le5\cdot10^6\,JH_sR^2\left(\frac1{KN^2}+\frac1{m'KN}\right).
    \label{up:expanding-final}
\end{equation}
\end{proposition}

\begin{proof}
If $K\ge114J$, apply Proposition~\ref{up:envelope-method} with
$n=\lfloor K/(114J)\rfloor$ and $j=J$. Then $1\le n\le K$,
$n\ge\frac K{228J}$, $114Jn\le K$, and
$J\ge\lceil\log_2(625n(n+8)N^4)\rceil$, so the proposition applies, the
number of labels is at most $m'Nn\le m'KN$, and
$\tau\le\frac{1824JH_s}K$. For $N\ge2$ we have $N-1\ge N/2$, and
\eqref{up:distributed-rate} gives
\[
    \mathbb E F(\widehat x)
    \le\frac{600\tau R^2}{N^2}+\frac{2304\tau R^2}{m'N}
    \le\frac{1\,094\,400\,JH_sR^2}{KN^2}
       +\frac{4\,202\,496\,JH_sR^2}{m'KN}.
\]
For $N=1$, \eqref{up:distributed-rate} gives
$\mathbb E F(\widehat x)\le3\tau R^2\le\frac{5472JH_sR^2}K$.

If $K<114J$, use \eqref{up:fresh-rate} with $b=m'K$ and $T=N$, that is,
$m'$ workers query a fresh label at every query. It uses $m'KN$ labels,
and, since $L\le H_s$ and $1<\frac{114J}K$,
\[
    \mathbb E F(\widehat x)
    \le\frac{36H_sR^2}{N^2}+\frac{288H_sR^2}{m'KN}
    \le\frac{4104\,JH_sR^2}{KN^2}+\frac{288H_sR^2}{m'KN}.
\]
In all cases \eqref{up:expanding-final} holds.
\end{proof}

\subsection{The full bound}
\label{up:fixed-pool}

\begin{theorem}[Upper bound with sample reuse]
\label{up:final-upper}
Let $S\in\mathbb N\cup\{\infty\}$, $S\ge1$. There is an algorithm in
$\mathrm{IC}(m,K,N)$ with access to a pool of size $S$ such that, for
every instance of $\mathcal S_d$,
\[
    \mathbb E F(\widehat x)\le10^7J\,R^2\min\left\{L,\;
        \frac{\min\{L,H_s/K\}}{N^2}+\frac{H_s}{\min\{S,mKN\}}\right\}.
\]
\end{theorem}

\begin{proof}
Put $n_S=\min\{S,mKN\}$. In each branch below the algorithm assigns
disjoint sets of previously unqueried labels in $\{1,\ldots,S\}$ to the
batches and queries, so that every batch is fresh; the guarantees of the
branches depend only on the public parameters, and the algorithm selects
the branch with the smallest guarantee before the run. Returning $x_0$
gives $F(x_0)\le\frac L2R^2$ in all cases.

\emph{Case $S\le KN$.} Here $n_S=S$. One worker runs serial SGD
(Lemma~\ref{lem:serial}) with $s=S$ queries on distinct labels, at most
$K$ per round, which gives $\mathbb E F(\widehat x)\le\frac{2H_sR^2}S$.
Together with $x_0$, the guarantee is
$R^2\min\{\frac L2,\frac{2H_s}{S}\}\le2R^2\min\{L,\frac{H_s}{n_S}\}$,
which is at most the bound of the theorem.

\emph{Case $S>KN$.} Let $m'=\min\{m,\lfloor S/(KN)\rfloor\}$, with
$m'=m$ if $S=\infty$. Then $1\le m'\le m$, $m'KN\le S$, and
$m'KN\ge n_S/2$: if $m'=m$, then $m'KN=mKN\ge n_S$, and otherwise
$\lfloor S/(KN)\rfloor\ge\frac S{2KN}$. Proposition~\ref{up:integers}
gives
\[
    \mathbb E F(\widehat x)
    \le5\cdot10^6JH_sR^2\left(\frac1{KN^2}+\frac2{n_S}\right)
    \le10^7JR^2\left(\frac{H_s}{KN^2}+\frac{H_s}{n_S}\right).
\]
Next, let $b=\min\{mK,\lfloor S/N\rfloor\}$, with $b=mK$ if $S=\infty$.
Then $b\ge1$ and $bN\le S$, and $bN\ge n_S/2$ by the same argument. The
method \eqref{up:fresh-rate} with $T=N$ and $b$ queries per round at a
common point, distributed among the workers with at most $K$ each, gives
\[
    \mathbb E F(\widehat x)
    \le R^2\left(\frac{36L}{N^2}+\frac{576H_s}{n_S}\right).
\]
The smallest of the three guarantees is at most
\[
    10^7JR^2\min\left\{L,\;
      \frac{\min\{L,H_s/K\}}{N^2}+\frac{H_s}{n_S}\right\},
\]
since $J\ge1$. This proves the theorem in both cases.
\end{proof}

\section{Proofs of the main results}
\label{app:main-proofs}

\begin{proof}[Proof of Theorem~\ref{thm:weak-growth-convex}]
The upper bound is Corollary~\ref{cor:selector}, and the lower bound is
Theorem~\ref{wg:lower-theorem}.
\end{proof}

\begin{proof}[Proof of Theorem~\ref{thm:weak-growth-quadratic}]
The upper bound is Corollary~\ref{quad:selector}, whose proof uses only
\eqref{eq:wg} and the quadratic form of $F$. By
Lemma~\ref{lem:relations},
$\mathfrak M_d^{\mathrm{wg,quad}}\ge\mathfrak M_d^{\mathrm{sg,quad}}$,
and Theorem~\ref{qlow:allbudget} gives the lower bound for
$d\ge2048m(KN)^2$.
\end{proof}

\begin{proof}[Proof of Theorem~\ref{thm:sample-gradient-quadratic}]
By Lemma~\ref{lem:relations}, every instance of
$\mathcal S_d^{\mathrm{quad}}$ is an instance of
$\mathcal W_d^{\mathrm{quad}}$, so Corollary~\ref{quad:selector} gives the
upper bound in every dimension. The lower bound is
Theorem~\ref{qlow:allbudget}.
\end{proof}

\begin{proof}[Proof of Theorem~\ref{thm:sample-gradient-convex}]
By Lemma~\ref{lem:relations}, Corollary~\ref{cor:selector} applies to
every instance of $\mathcal S_d$ and gives the upper bound in every
dimension. Since $\mathcal S_d^{\mathrm{quad}}\subseteq\mathcal S_d$, we
have $\mathfrak M_d^{\mathrm{sg}}\ge\mathfrak M_d^{\mathrm{sg,quad}}$, and
Theorem~\ref{qlow:allbudget} gives the lower bound for
$d\ge2048m(KN)^2$.
\end{proof}

The lower bound of Theorem~\ref{thm:sample-gradient-convex} is also
implied, with the constant $10^6$ and for $d\ge d_{\mathrm{pool}}$, by
Theorem~\ref{lo:combined} with $S=\infty$ and \eqref{eq:pool-fresh}. That
construction has a non-quadratic population and remains hard under
arbitrary reuse of samples.

\begin{proof}[Proof of Theorem~\ref{thm:pool}]
The upper bound is Theorem~\ref{up:final-upper}, and the lower bound is
Theorem~\ref{lo:combined}.
\end{proof}

\end{document}